\documentclass[a4paper,10pt]{amsart}
\usepackage{amsmath,amssymb,amsthm,graphicx}
\usepackage[all]{xy}
\usepackage[top=30truemm,bottom=30truemm,left=20truemm,right=20truemm]{geometry}

\allowdisplaybreaks[1]

\newtheorem{Th}{Theorem}[section]
\newtheorem{Cor}[Th]{Corollary}
\newtheorem{Prop}[Th]{Proposition}
\newtheorem{Lem}[Th]{Lemma}
\newtheorem*{Th*}{Theorem}

\theoremstyle{definition}
\newtheorem{Def}[Th]{Definition}

\theoremstyle{remark}
\newtheorem{Rem}[Th]{Remark}

\newcommand{\Z}{\boldsymbol{Z}}

\newcommand{\bS}{\boldsymbol{S}}

\newcommand{\rb}{\mathrm{b}}
\newcommand{\id}{\mathrm{id}}
\newcommand{\cZ}{\mathcal{Z}}
\newcommand{\cD}{\mathcal{D}}

\DeclareMathOperator{\module}{mod}
\DeclareMathOperator{\umod}{\underline{mod}}
\DeclareMathOperator{\proj}{proj}
\DeclareMathOperator{\add}{add}
\DeclareMathOperator{\Aut}{Aut_{tr}}
\DeclareMathOperator{\Hom}{Hom}
\DeclareMathOperator{\rad}{rad}
\DeclareMathOperator{\End}{End}
\DeclareMathOperator{\Ext}{Ext}

\DeclareMathOperator{\Mat}{Mat}
\DeclareMathOperator{\GL}{GL}

\DeclareMathOperator{\Coker}{Cok}
\DeclareMathOperator{\Ker}{Ker}
\DeclareMathOperator{\Image}{Im}

\renewcommand{\Gamma}{\varGamma}
\renewcommand{\Delta}{\varDelta}
\renewcommand{\Lambda}{\varLambda}
\renewcommand{\epsilon}{\varepsilon}
\renewcommand{\Phi}{\varPhi}
\renewcommand{\Psi}{\varPsi}
\renewcommand{\phi}{\varphi}
\renewcommand{\Im}{\Image}
\renewcommand{\mod}{\module}

\title{The Grothendieck groups and stable equivalences of mesh algebras}
\author{Sota Asai}
\address{Graduate School of Mathematics, Nagoya University, Furo-cho,
Chikusa-ku, Nagoya-shi, Aichi-ken, 464-8602, Japan}
\email{m14001v@math.nagoya-u.ac.jp}

\begin{document}

\begin{abstract}
We deal with the finite-dimensional mesh algebras given by stable translation quivers.
These algebras are self-injective, and thus the stable categories 
have a structure of triangulated categories. 
Our main result determines the Grothendieck groups of these stable categories. 
As an application, 
we give an complete classification of the mesh algebras up to stable equivalences.
\end{abstract}

\maketitle

\tableofcontents

\section{Introduction}\label{Intro}

%

Let $K$ be a field and $\Lambda$ be a finite-dimensional $K$-algebra.
The representation theory of finite-dimensional $K$-algebras investigates 
the category of finite-dimensional modules $\mod \Lambda$.
One of the useful methods is studying relationships 
between two finite-dimensional algebras
$\Lambda_1$ and $\Lambda_2$.

First, there is an important relationship called \textit{derived equivalence},
that is, the bounded derived categories 
$D^\mathrm{b}(\mod \Lambda_1)$ and $D^\mathrm{b}(\mod \Lambda_2)$ are
equivalent as triangulated categories.
Rickard characterized derived equivalence in terms of
\textit{tilting complexes} \cite{Rickard-Morita}.
A typical example of derived equivalences is given by reflections of quivers 
\cite{Happel-on}.
Derived equivalences have been actively studied,
see \cite{AHK,Happel,HJR,KZ},
and references therein.
 
In the rest, we assume that $\Lambda$ is self-injective.
Then the category $\mod \Lambda$ becomes a \textit{Frobenius category}, 
and thus the stable module category 
$\umod \Lambda$ has a structure of a triangulated category
with its shift $[1] \colon \umod \Lambda \to \umod \Lambda$
defined by taking cosyzygies (see \cite{Happel}).
For two self-injective finite-dimensional algebras $\Lambda_1,\Lambda_2$,
an important relationship is a \textit{stable equivalence},
that is, $\umod \Lambda_1 \cong \umod \Lambda_2$ as triangulated categories.
Rickard showed 
that $\umod \Lambda \cong D^\rb(\mod \Lambda)/K^\rb(\proj \Lambda)$ 
as triangulated categories,
and that derived equivalent self-injective algebras are stable equivalent
in \cite{Rickard-DS}.

In this paper, we deal with a certain class of finite-dimensional self-injective algebras
called \textit{mesh algebras (categories)} introduced by Riedtmann
associated with \textit{translation quivers}. 
The Auslander--Reiten quivers (AR quivers) of module categories or derived categories
are important examples, and  after Riedtmann, it is known that
many important categories are recovered from
their AR quivers as mesh categories.
For example, if $Q$ is a Dynkin quiver and $KQ$ is its path algebra,
it is shown that the bounded derived category $D^\rb(\mod KQ)$ is equivalent to 
the mesh category of the AR quiver $\Z Q$ \cite{Happel}.
 
It is known that $\Z Q$ is locally bounded 
if and only if $Q$ is a Dynkin quiver \cite{Riedtmann},
and in this case, 
$\Z Q$ does not depend on the orientation of arrows up to isomorphisms.
We write $\Z \Delta$ instead of $\Z Q$
if $\Delta$ is the underlying Dynkin graph of $Q$.
Considering an ``admissible'' automorphism $\rho \in \Aut \Z \Delta$,
the mesh algebra of $\Z \Delta/\langle \rho \rangle$ 
is a finite-dimensional $K$-algebra.

In this paper, we consider the mesh algebra of a \textit{stable translation quiver},
that is, a translation quivers such that the translation is a bijection 
on the vertices.
A stable translation quiver which has the finite-dimensional mesh algebra
is the form of $\Z \Delta/\langle \rho \rangle$, 
see \cite{Dugas,Riedtmann}.
In this case, the mesh algebra is self-injective.
More explicitly, these are all stable translation quivers 
with finite-dimensional mesh algebras.
\begin{center}
\begin{tabular}{c|l||c|l}
type & quiver &
type & quiver \\
\hline
I    & $\Z A_n/ \langle \tau^k \rangle$ &
II   & $\Z A_n/ \langle \tau^k \psi \rangle$ ($n \notin 2\Z$) \\
III  & $\Z A_n/ \langle \tau^k \phi \rangle$ ($n \in 2\Z$) &
IV   & $\Z D_n/ \langle \tau^k \rangle$ \\
V    & $\Z D_n/ \langle \tau^k \psi \rangle$ &
VI   & $\Z D_4/ \langle \tau^k \chi \rangle$ \\
VII  & $\Z E_6/ \langle \tau^k \rangle$ &
VIII & $\Z E_6/ \langle \tau^k \psi \rangle$ \\
IX   & $\Z E_7/ \langle \tau^k \rangle$ &
X    & $\Z E_8/ \langle \tau^k \rangle$
\end{tabular}
\end{center}
The symbol $\tau$ denotes the translation of $\Z \Delta$
and $\psi,\phi,\chi$ are automorphisms satisfying 
$\psi^2=\id$, $\phi^2=\tau^{-1}$, $\chi^3=\id$
(see Section 2 for the detail).
For example, the preprojective algebras of Dynkin type are included in the list above
as the mesh algebra of $\Z \Delta/\langle \tau \rangle$.  
 
Our main result is to determine
the \textit{Grothendieck groups} of the stable module categories of these mesh algebras.
The Grothendieck group is an important invariant of triangulated categories.
For a finite-dimensional algebra $\Lambda$, 
the Grothendieck group of the bounded derived category,
$K_0(D^\mathrm{b}(\mod \Lambda))$, is a free abelian group with its basis given by
the nonisomorphic simple $\Lambda$-modules. 
On the other hand, if $\Lambda$ is self-injective, 
the Grothendieck group of the stable module category, 
$K_0(\umod \Lambda)$, is isomorphic to the quotient $K_0(D^\mathrm{b}(\mod \Lambda))/H$,
where $H$ is the subgroup generated by 
the projective $\Lambda$-modules.
Using this description, 
we proved the following main result of this paper,
which will be shown in Section 3.

\begin{Th}\label{main}
Let $Q=\Z \Delta/\langle \rho \rangle$ be a stable translation quiver 
whose mesh algebra $\Lambda$ is finite-dimensional,
and $c$ be the Coxeter number of $\Delta$, and put
$d=\gcd(c,2k-1)/2$ if $\Z \Delta/\langle \rho \rangle=\Z A_n/\langle \tau^k \phi \rangle$ (i.e. $Q$ is type III) and 
$d=\gcd(c,k)$ otherwise,
and $r=c/d$.
Then we have 
\begin{align*}
K_0(\umod \Lambda) \cong \Z^a \oplus (\Z/2\Z)^b \oplus H,
\end{align*}
where $a,b,H$ are given in Table \ref{Grothendieck}.
\end{Th}

\begin{table}
\begin{tabular}{cll|ccccc}
type & quiver & condition & $a$ & $b$ & $H$ \\
\hline
I & 
$\Z A_n/ \langle \tau^k \rangle$ & $r \in 2\Z$ & $(nd-3d+2)/2$ & $d-1$ & \\
&
& $r \notin 2\Z$ & $(nd-2d+2)/2$ & & \\
\hline
II & 
$\Z A_n/ \langle \tau^k \psi \rangle$ & $r \in 4\Z$ & $(nd-3d)/2$ & $d-1$ & $\Z/4\Z$ \\
&
& $r \in 2+4\Z$ & & $nd-2d+1$ & \\
&
& $r \notin 2\Z$ & $(nd-d)/4$ & & \\
\hline
III &
$\Z A_n/ \langle \tau^k \phi \rangle$ & & & $nd-2d+1$ \\
\hline
IV &
$\Z D_n/ \langle \tau^k \rangle$ & $k \in 2\Z$, $r \in 2\Z$ & $d-1$ & $nd-3d$ & $\Z/r\Z$ \\
&
& $k \in 2\Z$, $r \notin 2\Z$ & $(nd-d-2)/2$ & & $\Z/r\Z$ \\
&
& $k \notin 2\Z$, $r \in 4\Z$ & $d$ & $nd-3d$ & \\
&
& $k \notin 2\Z$, $r \notin 4\Z$ & & $nd-d-1$ & \\
\hline
V &
$\Z D_n/ \langle \tau^k \psi \rangle$ & $k \in 2\Z$, $r \in 4\Z$ & $d$ & $nd-3d$ & \\
&
& $k \in 2\Z$, $r \in 2+4\Z$ & & $nd-d-1$ & \\
&
& $k \in 2\Z$, $r \notin 2\Z$ & $(nd-2d)/2$ & & \\
&
& $k \notin 2\Z$ & $d-1$ & $nd-3d$ & $\Z/r\Z$ \\
\hline
VI &
$\Z D_4/ \langle \tau^k \chi \rangle$ & $k \in 2\Z$ & $4$ & \\
&
& $k \notin 2\Z$ & & $4$ \\
\hline
VII &
$\Z E_6/ \langle \tau^k \rangle$ & $d=1,3$ & $d+1$ & $d+1$ & $(\Z/4\Z)^{d-1}$ \\
&
& $d=2,6$ & $(3d+2)/2$ & $(3d+2)/2$ & \\
&
& $d=4,12$ & $(9d+12)/4$ & & \\
\hline
VIII &
$\Z E_6/ \langle \tau^k \psi \rangle$ & $d=1,3$ & $2d$ & $d+1$ & \\
&
& $d=2,6$ & & $(9d+6)/2$ & \\
&
& $d=4,12$ & $(3d+4)/2$ & & \\
\hline
IX &
$\Z E_7/ \langle \tau^k \rangle$ & $d=1$ & & $6$ & \\
&
& $d=3,9$ & & $6d+2$ & \\
&
& $d=2$ & $6$ & & $\Z/3\Z$ \\
&
& $d=6,18$ & $3d+2$ & & \\
\hline
X &
$\Z E_8/ \langle \tau^k \rangle$ & $d=1,3,5$ & & $8d$ \\
&
& $d=15$ & & $112$ \\
&
& $d=2,6,10$ & $4d$ & \\
&
& $d=30$ & $112$
\end{tabular}
\caption{The Grothendieck groups of the stable module categories}
\label{Grothendieck}
\end{table}

The key ingredient of the proof is a well-known property 
of mesh algebras, i.e. 
the simple modules are closed under taking 3rd cosyzygies (cf. \cite{AR,Dugas}),
see Proposition \ref{proj_resol}.
As an application of this result,
we give a complete classification of the mesh algebras up to stable equivalences
in the case $K$ is algebraically closed.

\begin{Th}\label{not_stable_eq}
Assume that $K$ is an algebraically closed field.
Let $Q=\Z\Delta/\langle \rho \rangle$, 
$Q'=\Z\Delta'/\langle \rho' \rangle$ 
be stable translation quivers whose mesh algebras 
$\Lambda$, $\Lambda'$ are finite-dimensional.
\begin{itemize}
\item[(1)]
If $\Lambda$ and $\Lambda'$ are stable equivalent, then we have
either $\Delta=A_1=\Delta'$ or $Q \cong Q'$ as translation quivers.
\item[(2)]
If $\Lambda$ and $\Lambda'$ are derived equivalent, then we have
$Q \cong Q'$ as translation quivers.
\end{itemize}
\end{Th}

This theorem says that finite-dimensional mesh algebras
are stable equivalent (or derived equivalent)
only for trivial cases.
To prove Theorem \ref{not_stable_eq}, we compare
the Grothendieck groups given in Theorem \ref{main} and also 
the following invariants under stable equivalences of mesh algebras.

As it is well-known, the functor $[-2] \circ \bS$ commutes with stable equivalences
up to functorial isomorphisms,
where $\bS$ is the Serre functor of $\umod \Lambda$.
Thus we can use the order of $[-2] \circ \bS$
as an invariant under stable equivalences,
and actually, this coincide with the order of the functor 
$\tau_* \colon \umod \Lambda \to \umod \Lambda$
induced from the translation $\tau \in \Aut Q$ in most cases,
see Proposition \ref{l=l'}.

We also use the invariant 
given as the number of nonisomorphic indecomposable direct summands of 
a $([-2] \circ \bS)$-stable cluster-tilting object 
(or more generally, maximal $([-2] \circ \bS)$-stable rigid object) 
in the stable module category $\umod \Lambda$.
We generalize the method of \cite{BIRS},
which gives a construction of cluster-tilting objects 
for a preprojective algebra by
\textit{reduced expressions} of 
the longest element of the Coxeter group.
These invariants are given in Theorems \ref{CT_not_III} and \ref{no_CT_III}.
Especially, for a finite-dimensional mesh algebra $\Lambda$,
the stable module category $\umod \Lambda$ has 
$([-2] \circ \bS)$-stable cluster tilting objects 
if and only if $\Lambda$ is not type III, see Corollary \ref{CT_iff}.

The last invariant is the order of the shift $[1] \colon \umod \Lambda \to \umod \Lambda$
determined by Andreu Juan and Saor\'{i}n in \cite{AS}, 
which is given in Proposition \ref{shift_order}.

Using these invariants, we give a proof of Theorem \ref{not_stable_eq} in Section 5.

\subsection{Conventions}

In this paper, $K$ is a field.

The term ``Dynkin diagrams'' mean ``simply-laced Dynkin diagrams'',
$A_n,D_n,E_6,E_7,E_8$.

We denote by $\tau$ the translation of a stable translation quiver.
Note that we do not consider the Auslander--Reiten translation 
of the corresponding mesh algebra in this paper. 

If $f \colon X \to Y$ and $g \colon Y \to Z$ are maps,
the composition of these two maps are denoted by $gf \colon X \to Z$.

For a finite-dimensional algebra $\Lambda$,
$\mod \Lambda$ denotes the category of finite-dimensional right $\Lambda$-modules
and $\operatorname{proj} \Lambda$ denotes the category of 
finite-dimensional projective right $\Lambda$-modules.
We denote by $\umod \Lambda$ the stable module category $\mod \Lambda/{\proj \Lambda}$,
and it has a structure of a triangulated category if $\Lambda$ is self-injective.

For a quiver $Q$, the set of its vertices is denoted by $Q_0$,
and the set of its arrows is denoted by $Q_1$.
We denote by $KQ$ the \textit{path algebra} of $Q$.
We put $\Lambda=KQ/I$, 
where $I$ is an \textit{admissible ideal}.
We denote by $e_i$ the idempotent in $\Lambda$ corresponding the vertex $i \in Q_0$,
by $P_i=e_i \Lambda$ the indecomposable projective $\Lambda$-module, 
and by $S_i=e_i \Lambda/(e_i \operatorname{rad} \Lambda)$ the simple $\Lambda$-module.
Here, $\operatorname{rad} \Lambda$ is the \textit{Jacobson radical}.
%

\section{Preliminary}\label{Pre}

In this section, we recall some properties of Grothendieck groups and mesh algebras,
and define the quivers giving finite-dimensional mesh algebras. 

\subsection{Grothendieck groups}
For a triangulated category $\mathcal{T}$,
the Grothendieck group $K_0(\mathcal{T})$ is an abelian group defined as follows.

\begin{Def}
Let $\mathcal{T}$ be a triangulated category with its shift 
$[1] \colon \mathcal{T} \to \mathcal{T}$.
The Grothendieck group $K_0(\mathcal{T})$ is defined as 
$\mathcal{F}(\mathcal{T})/\mathcal{F}_0(\mathcal{T})$,
where $\mathcal{F}(T)$ is a free abelian group 
with its basis all isomorphic classes in $\mathcal{T}$,
and $\mathcal{F}_0(\mathcal{T})$ is the subgroup of $\mathcal{F}(T)$
generated by the set $\{ [X]-[Y]+[Z] \mid \mbox{$X \to Y \to Z \to X[1]$: a triangle}\}$.
\end{Def}

%
%
%
The facts in the following lemma are well-known and fundamental  
for the calculation of $K_0(\umod \Lambda)$.
The part (3) is deduced by (1) and (2).

\begin{Lem}\label{quot_by_P}
Let $Q$ be a finite quiver with $Q_0=\{1,\ldots,m\}$,
$I$ be an admissible ideal of the path algebra $KQ$,
and $\Lambda=KQ/I$.
\begin{itemize}
\item[(1)] \cite[III.1.2]{Happel}
The family of the simple $\Lambda$-modules 
$([S_1],\ldots,[S_m])$ is a $\Z$-basis of $K_0(D^\rb(\mod \Lambda))$.
If $X$ is a $\Lambda$-module in $\mod \Lambda$
and $0 \to X^0 \to \cdots \to X^l \to 0$ is exact in $\mod \Lambda$, then we have
\begin{align*}
[X]=\sum_{i=1}^m (\dim_K X e_i) [S_i], \quad 
\sum_{j=0}^l (-1)^j[X^j]=\sum_{j=0}^l \sum_{i=1}^m (-1)^j(\dim_K X^j e_i) [S_i]=0.
\end{align*} 
\item[(2)] \cite[Theorem 2.1]{Rickard-DS}
Assume that $\Lambda$ is self-injective.
Then 
$K^\rb(\proj \Lambda)$ can be considered as a thick subcategory of $D^\rb(\mod \Lambda)$, 
and we have $\umod \Lambda \cong D^\rb(\mod \Lambda)/K^\rb(\proj \Lambda)$
as triangulated categories.
\item[(3)] 
Assume that $\Lambda$ is self-injective. Then we have
\begin{align*}
K_0(\umod \Lambda)\cong
K_0(D^\rb(\mod \Lambda))/\langle [P_1], \ldots, [P_m] \rangle.
\end{align*}
\end{itemize} 
\end{Lem}


If $\Lambda$ is self-injective, 
the Grothendieck group $K_0(\umod \Lambda)$ is isomorphic to 
the cokernel of the \textit{Cartan matrix} $C=(c_{ij})$,
where $c_{ij}=\dim_K e_j \Lambda e_i$.
However, the entries in the Cartan matrix are often complicated,
and the straightforward calculation of the cokernel is very nasty.
For mesh algebras, we will give another set of generators of 
the subgroup $\langle [P_1], \ldots, [P_m] \rangle$
having ``simple'' coefficients than the Cartan matrix in Section 3.

\subsection{Mesh algebras}

A quiver $Q$ is called \textit{locally finite} 
if each vertex $u \in Q_0$ has only finitely many arrows from $u$ and to $u$.
A locally finite quiver $Q=(Q,\tau)$ 
with an automorphism $\tau \colon Q_0 \to Q_0$ on the set of vertices
is called a \textit{stable translation quiver} 
if the number of arrows from $u$ to $v$ coincides with
the number of arrows from $v$ to $\tau^{-1} u$ for any $u,v \in Q_0$, 
and then $\tau$ is called the \textit{translation} of $Q$.

For the convinience, we assume that $Q$ has no multiple arrows; 
that is, for $u,v \in Q_0$, 
there exists at most one arrow from $u$ to $v$.
The translation quivers appearing in this paper satisfy this condition.

For $u \in Q_0$,
let $u^+ \subset Q_0$ be the set of direct successors of $u$ and 
$v_1,\ldots,v_m$ be all distinct elements of $u^+$.
Then the fullsubquiver
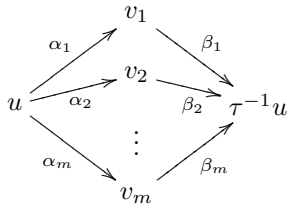
\begin{align*}
\begin{xy}
( 0,  0) *+{u}           = "1",
(16, 12) *+{v_1}         = "2",
(16,  4) *+{v_2}         = "3",
(16, -4) *+{\vdots}      = "4",
(16,-12) *+{v_m}         = "5",
(32,  0) *+{\tau^{-1} u} = "6"
\ar ^{\alpha_1} "1"; "2"
\ar _{\alpha_2} "1"; "3"
\ar _{\alpha_m} "1"; "5"
\ar ^{\beta_1}  "2"; "6"
\ar _{\beta_2}  "3"; "6"
\ar _{\beta_m}  "5"; "6"
\end{xy}
\end{align*}
of $Q$ is called a \textit{mesh} and
the relation $\alpha_1 \beta_1 + \cdots + \alpha_m \beta_m=0$ is 
called the \textit{mesh relation} of each mesh.

We can construct a \textit{mesh algebra}
from a stable translation quiver $Q$.
It is the quotient of the path algebra $KQ$
by the all mesh relations.
Note that the mesh algebra may be infinite-dimensional even 
if $Q$ is a finite stable translation quiver.
In the next subsection, we define the stable translation quivers
such that the mesh algebra is finite-dimensional. 

\subsection{Definitions of quivers}

We define a translation quiver $\Z \Delta$ for a Dynkin diagram $\Delta$.

Let $Q$ be an acyclic finite quiver with no multiple arrows.
We define a translation quiver $\Z Q$ as follows (see \cite{ASS,Happel});
the set of vertices $(\Z Q)_0$ is $Q_0 \times \Z$,
the set of arrows $(\Z Q)_1$ is 
\begin{align*}
\{ (i,a) \to (j,a) \mid (i \to j) \in Q_1, a \in \Z \}
\amalg \{ (j,a) \to (i,a+1) \mid (j \to i) \in Q_1, a \in \Z \},
\end{align*}
and the translation $\tau$ is given by $\tau(i,a)=(i,a-1)$. 

Because Dynkin diagrams are trees,
for two quivers $Q,Q'$ such that their underlying diagrams are 
the same Dynkin diagram $\Delta$,
we have $\Z Q \cong \Z Q'$ up to isomorphisms of translation quivers.
Thus, we can write $\Z \Delta$ for these quivers.
However, we would like to
fix the numbering of the vertices of $\Z \Delta$ in this paper,
so we assume that each Dynkin diagram is oriented and numbered as follows; 
\begin{align*}
A_n \colon & 
\begin{xy}
( 0,  0) *+{1}      = "1",
(16,  0) *+{2}      = "2",
(32,  0) *+{\cdots} = "3",
(48,  0) *+{n}      = "4",
\ar "1"; "2"
\ar "2"; "3"
\ar "3"; "4"
\end{xy}, \\
D_n \colon & 
\begin{xy}
( 0,  0) *+{1}      = "1",
(16,  0) *+{2}      = "2",
(32,  0) *+{\cdots} = "3",
(48,  0) *+{n-2}    = "4",
(64,  0) *+{n-1}    = "5",
(48, -8) *+{n}      = "6",
\ar "1"; "2"
\ar "2"; "3"
\ar "3"; "4"
\ar "4"; "5"
\ar "4"; "6"
\end{xy}, \\
E_6 \colon & 
\begin{xy}
( 0,  0) *+{1} = "1",
(16,  0) *+{2} = "2",
(32,  0) *+{3} = "3",
(48,  0) *+{4} = "4",
(64,  0) *+{5} = "5",
(32, -8) *+{6} = "6",
\ar "1"; "2"
\ar "2"; "3"
\ar "3"; "4"
\ar "4"; "5"
\ar "3"; "6"
\end{xy}, \\
E_7 \colon & 
\begin{xy}
( 0,  0) *+{1} = "1",
(16,  0) *+{2} = "2",
(32,  0) *+{3} = "3",
(48,  0) *+{4} = "4",
(64,  0) *+{5} = "5",
(80,  0) *+{6} = "6",
(48, -8) *+{7} = "7",
\ar "1"; "2"
\ar "2"; "3"
\ar "3"; "4"
\ar "4"; "5"
\ar "5"; "6"
\ar "4"; "7"
\end{xy}, \\
E_8 \colon & 
\begin{xy}
( 0,  0) *+{1} = "1",
(16,  0) *+{2} = "2",
(32,  0) *+{3} = "3",
(48,  0) *+{4} = "4",
(64,  0) *+{5} = "5",
(80,  0) *+{6} = "6",
(96,  0) *+{7} = "7",
(64, -8) *+{8} = "8",
\ar "1"; "2"
\ar "2"; "3"
\ar "3"; "4"
\ar "4"; "5"
\ar "5"; "6"
\ar "6"; "7"
\ar "5"; "8"
\end{xy}.
\end{align*}
The symbol $\Z \Delta$ denotes the translation quiver 
based on these orientations and numberings.

First, the translation $\tau$ can be extended to an automorphism on $\Z \Delta$.
We can construct finite stable translation quivers using $\tau$.

\begin{Def}
Let $\Delta$ be a Dynkin diagram with $n$ vertices and $k \ge 1$ be an integer. 
Then we can consider a finite stable translation quiver $\Z \Delta/\langle \tau^k \rangle$.
We set the indices of the vertices of $\Z \Delta/\langle \tau^k \rangle$
as the elements of $\{ 1,\ldots,n \} \times (\Z/k\Z)$.
\end{Def}

%

For some Dynkin diagram $\Delta$, 
$\tau$ does not generate the automorphism group $\Aut \Z \Delta$
as a translation quiver,
so we define other automorphisms on $\Z \Delta$.

\begin{Def}\label{torsion_aut}
We define $\psi,\phi,\chi \in \Aut \Z \Delta$ as the following.
\begin{itemize}
\item[(1)]
If $\Delta$ is $A_n$ with $n \notin 2\Z$, $D_n$ or $E_6$, 
$\psi$ is given as follows, and then $\psi$ satisfies $\psi^2=\id$;
\begin{align*}
A_n \ (n \notin 2\Z) \colon & (i,a) \mapsto (n+1-i,a+i-(n+1)/2), \\
D_n \colon & (i,a) \mapsto (i,a) \quad (i \in \{1, \ldots, n-2\}), \quad 
(n-1,a) \mapsto (n,a), \quad (n,a) \mapsto (n-1,a), \\
E_6 \colon & (i,a) \mapsto (6-i,a+i-3) \quad (i \in \{1, \ldots, 5\}), \quad
(6,a) \mapsto (6,a).
\end{align*}
\item[(2)]
If $\Delta$ is $A_n$ with $n \in 2\Z$, 
$\phi$ is given as $(i,a) \mapsto (n+1-i,a+i-n/2)$,
and then $\phi$ satisfies $\phi^2=\tau^{-1}$.
\item[(3)]
If $\Delta$ is $D_4$, 
$\chi$ is given as 
$(1,a) \mapsto (3,a-1)$,  
$(2,a) \mapsto (2,a)$, 
$(3,a) \mapsto (4,a)$,
$(4,a) \mapsto (1,a+1)$,
and then $\chi$ satisfies $\chi^3=\id$.
\end{itemize}
\end{Def}

We can consider an automorphism 
$\tau^k \psi$, $\tau^k \phi$ or $\tau^k \chi$ on $\Z \Delta$
and a finite translation quiver 
$\Z \Delta/ \langle \tau^k \psi \rangle$,
$\Z \Delta/ \langle \tau^k \phi \rangle$, or
$\Z \Delta/ \langle \tau^k \chi \rangle$ for $k \ge 1$ in each case above. 
Each quiver automorphism of $\tau,\psi,\chi,\phi$ on $\Z \Delta$ 
can uniquely induce the quiver automorphism on $\Z \Delta/\langle \tau^k \rangle$,
and the induced automorphism is also denoted by the same symbol.
By the definition of the quivers, we have the following lemma.

\begin{Lem}\label{torsion}
Let $k \ge 1$ be an integer.
\begin{itemize}
\item[(1)]
If $\Delta$ is $A_n$ with $n \notin 2\Z$, $D_n$ or $E_6$, 
the translation quiver $\Z \Delta/ \langle \tau^k \psi \rangle$ is isomorphic to 
the quotient of $\Z \Delta/ \langle \tau^{2k} \rangle$ 
by $\tau^k \psi \in \Aut(\Z \Delta/ \langle \tau^{2k} \rangle)$.
\item[(2)]
If $\Delta$ is $A_n$ with $n \in 2\Z$, 
the translation quiver $\Z \Delta/ \langle \tau^k \phi \rangle$ is isomorphic to 
the quotient of $\Z \Delta/ \langle \tau^{2k-1} \rangle$ 
by $\tau^k \phi \in \Aut(\Z \Delta/ \langle \tau^{2k-1} \rangle)$.
\item[(3)]
If $\Delta$ is $D_4$, 
the translation quiver $\Z \Delta/ \langle \tau^k \chi \rangle$ is isomorphic to 
the quotient of $\Z \Delta/ \langle \tau^{3k} \rangle$ 
by $\tau^k \chi \in \Aut(\Z \Delta/ \langle \tau^{3k} \rangle)$.
\end{itemize}
\end{Lem}

\begin{Def}\label{quiver_def}
We denote the stable translation quivers defined above as follows.
\begin{center}
\begin{tabular}{c|ll||c|ll}
type & symbol           & quiver &
type & symbol           & quiver \\
\hline
I    & $Q_{A_n,k,1}$    & $\Z A_n/ \langle \tau^k \rangle$ &
II   & $Q_{A_n,2k,2}$   & $\Z A_n/ \langle \tau^k \psi \rangle$ ($n \notin 2\Z$) \\
III  & $Q_{A_n,2k-1,2}$ & $\Z A_n/ \langle \tau^k \phi \rangle$ ($n \in 2\Z$) &
IV   & $Q_{D_n,k,1}$    & $\Z D_n/ \langle \tau^k \rangle$ \\
V    & $Q_{D_n,2k,2}$   & $\Z D_n/ \langle \tau^k \psi \rangle$ &
VI   & $Q_{D_4,3k,3}$   & $\Z D_4/ \langle \tau^k \chi \rangle$ \\
VII  & $Q_{E_6,k,1}$    & $\Z E_6/ \langle \tau^k \rangle$ &
VIII & $Q_{E_6,2k,2}$   & $\Z E_6/ \langle \tau^k \psi \rangle$ \\
IX   & $Q_{E_7,k,1}$    & $\Z E_7/ \langle \tau^k \rangle$ &
X    & $Q_{E_8,k,1}$    & $\Z E_8/ \langle \tau^k \rangle$
\end{tabular}
\end{center}
The corresponding mesh algebra for $Q_{\Delta,l,t}$ is denoted by $\Lambda_{\Delta,l,t}$.
\end{Def}

Now we can state Riedtmann's structure theorem,
see also \cite[Theorem 3.1]{Dugas}.

\begin{Th}
Let $Q$ be a connected stable translation quiver.
\begin{itemize}
\item[(1)]
\cite[1.5, STRUKTURASATZ]{Riedtmann}
Assume that $Q$ has no multiple arrows.
Then there exist an oriented tree $B$ and a subgroup $G \subset \Aut \Z B$
such that $Q \cong \Z B/G$. 
\item[(2)]
\cite[2.1, SATZ 2]{Riedtmann}
Let $B$ be an oriented tree, and assume that there exists an integer $n$ such that 
any path in $\Z B$ with its length greater than or equal to $n$ 
is zero in the mesh algebra $K(\Z B)$.
Then the underlying graph $\bar{B}$ is a Dynkin diagram, 
namely $A_n,D_n,E_6,E_7,E_8$.
\item[(3)]
If $Q$ gives a finite-dimensional mesh algebra, then
$Q$ is isomorphic to one of the quivers in Definition \ref{quiver_def}.
\end{itemize}
\end{Th}

In the rest, 
the term ``mesh algebra'' means the mesh algebra of the form $\Lambda_{\Delta,l,t}$
unless otherwise stated.
From the next section, we begin the calculation of the Grothendieck groups
$K_0(\umod \Lambda_{\Delta,l,t})$.

\section{The Grothendieck groups of mesh algebras}

Let $\Lambda_{\Delta,l,t}$ be a finite-dimensional mesh algebra given by 
the stable translation quiver $Q_{\Delta,l,t}$.
Now we start the calculation of the Grothendieck groups of the stable categories
$K_0(\umod \Lambda_{\Delta,l,t})$.

\subsection{The main proposition and some definitions}

In this section, we describe the Grothendieck group $K_0(\umod \Lambda_{\Delta,l,t})$
using the cokernels of matrices on $\Z$.

We define some notations first.
For a ring $R$ and positive integers $m,n$,
we denote $\Mat_{m,n}(R)$ by the set of $m \times n$ matrices with entries in $R$, 
and by $\GL_m(R)$ the set of $m \times m$ invertible matrices in $\Mat_{m,m}(R)$,
and by $1_m$ the identity matrix in $\Mat_{m,m}(R)$.
Each $M \in \Mat_{m,n}(R)$ is regarded as an $R$-homomorphism $M \colon R^n \to R^m$,
and $\Ker M$, $\Im M$, $\Coker M$ mean the kernel, the image, and the cokernel of 
the map.
The symbol $M^{\oplus l}$ denotes the $ml \times nl$ matrix
obtained by placing $M$ diagonally $l$ times.
Let $M_i \in \Mat_{m,n_i}(R)$ ($i=1,\ldots,l$) and $n=n_1+\cdots+n_l$.   
The symbol $\begin{bmatrix} M_1 & \cdots & M_l \end{bmatrix}$ denotes 
a matrix in $\Mat_{m,n}(R)$, 
where $M_1,\ldots,M_l$ are seen as blocks of 
$\begin{bmatrix} M_1 & \cdots & M_l \end{bmatrix}$.

We use the following polynomials and matrices.

\begin{Def}
Let $m \ge 1$ be an integer.
\begin{itemize}
\item[(1)]
We define $s_m(x)=1-x+x^2-\cdots+(-1)^{m-1}x^{m-1} \in \Z[x]$.
\item[(2)]
We define $X_m \in \GL_m(\Z)$ as the permutation matrix of 
the cyclic permutation $(1,2,\ldots,m)$ in the symmetric group $\mathfrak{S}_m$;
that is,
\begin{align*}
X_m=\begin{bmatrix}
0       & 1 \\
1_{m-1} & 0
\end{bmatrix}
\end{align*}
\end{itemize}
\end{Def}

We show the following proposition in this section.

\begin{Prop}\label{tanninsi}
Let $n,k \ge 1$ be integers.
The Grothendieck group $K_0(\umod \Lambda_{\Delta,l,t})$ is isomorphic to the following;
\begin{align*}
& (\mathrm{I} \colon \Lambda_{A_n,k,1}) &&
\begin{cases}
(\Coker (1_k-X_k^{n+1}))^{(n-3)/2} \oplus \Coker (1_k-X_k)(1_k+X_k^{(n+1)/2}) 
& (n \notin 2\Z) \\
(\Coker (1_k-X_k^{n+1}))^{(n-2)/2} \oplus \Coker (1_k-X_k) 
& (n \in 2\Z)
\end{cases}, \\
& (\mathrm{II} \colon \Lambda_{A_n,2k,2}) &&
(\Coker (1_{2k}+X_{2k}^{k-(n+1)/2}))^{(n-3)/2} 
\oplus \Coker \begin{bmatrix} (1_{2k}-X_{2k})(1_{2k}+X_{2k}^{(n+1)/2}) & 
1_{2k}+X_{2k}^{k-(n+1)/2} \end{bmatrix}, \\
& (\mathrm{III} \colon \Lambda_{A_n,2k-1,2}) &&
(\Coker (1_{2k-1}+X_{2k-1}^{k-(n+2)/2}))^{(n-2)/2} \oplus 
\Coker \begin{bmatrix} 1_{2k-1}-X_{2k-1} & 2 \cdot 1_{2k-1} \end{bmatrix}, \\
& (\mathrm{IV} \colon \Lambda_{D_n,k,1}) &&
\begin{cases}
(\Coker(1_k+X_k^{n-1}))^{n-3} 
\oplus \Coker s_{2n-2}(X_k) & (n \notin 2\Z)\\
(\Coker(1_k+X_k^{n-1}))^{n-2} 
\oplus \Coker s_{n-1}(X_k) & (n \in 2\Z)
\end{cases}, \\
& (\mathrm{V} \colon \Lambda_{D_n,2k,2}) &&
\begin{cases}
(\Coker \begin{bmatrix} 1_{2k}\!+\!X_{2k}^{n-1} & 1_{2k}\!-\!X_{2k}^k \end{bmatrix})^{n-3}
\oplus \Coker \begin{bmatrix} s_{2n-2}(X_{2k}) 
& 1_{2k}\!+\!X_{2k}^{k-(n-1)} \end{bmatrix} & (n \notin 2\Z) \\
(\Coker \begin{bmatrix} 1_{2k}\!+\!X_{2k}^{n-1} & 1_{2k}\!-\!X_{2k}^k \end{bmatrix})^{n-3}
\oplus \Coker \begin{bmatrix} 1_{2k}\!+\!X_{2k}^{n-1}
& (1\!-\!X_{2k}^k)s_{n-1}(X_{2k}) \end{bmatrix} & (n \in 2\Z) \\
\end{cases}, \\
& (\mathrm{VI} \colon \Lambda_{D_4,3k,3}) &&
\Coker (1_{3k}+X_{3k}^3) \oplus \Coker (1_{3k}+X_{3k}), \\
& (\mathrm{VII} \colon \Lambda_{E_6,k,1}) && 
\Coker (1_k-X_k^{12}) \oplus \Coker ((1_k-X_k)(1_k+X_k^3+X_k^6+X_k^9)) 
\oplus \Coker (1_k+X_k^6) \oplus \Coker (1_k+X_k^2), \\
& (\mathrm{VIII} \colon \Lambda_{E_6,2k,2}) &&
\Coker (1_{2k}+X_{2k}^{k-6}) \oplus \Coker \begin{bmatrix} 
(1_{2k}-X_{2k})(1_{2k}+X_{2k}^3+X_{2k}^6+X_{2k}^9) & 1_{2k}+X_{2k}^{k-6} \end{bmatrix}\\
&&& \quad \oplus \Coker \begin{bmatrix} 1_{2k}+X_{2k}^6 & 1_{2k}+X_{2k}^{k-6} \end{bmatrix} 
\oplus \Coker \begin{bmatrix} 1_{2k}+X_{2k}^2 & 1_{2k}+X_{2k}^{k-6} \end{bmatrix}, \\
& (\mathrm{IX} \colon \Lambda_{E_7,k,1}) &&
(\Coker (1_k+X_k^9))^6 \oplus \Coker (1_k-X_k+X_k^2), \\
& (\mathrm{X} \colon \Lambda_{E_8,k,1}) &&
(\Coker (1_k+X_k^{15}))^7 \oplus \Coker ((1_k-X_k+X_k^2)(1_k+X_k^5)).
\end{align*}
\end{Prop}

As mentioned after Lemma \ref{quot_by_P}, 
using the isomorphism in Lemma \ref{quot_by_P} (3)
straightforwardly is not a good way to calculate the Grothendieck group.
Thus, we will give a simpler set of generators of 
$\langle [P_u] \mid u \in (Q_{\Delta,l,t})_0 \rangle$ 
in the next subsection.
The proof of Proposition \ref{tanninsi} is based on the new generators.

\subsection{Another set of generators and polynomial matrices}

In the proof of Proposition \ref{tanninsi}, 
we use \textit{Nakayama permutations} effectively. 
Let $\pi$ be the Nakayama permutation of $\Lambda_{\Delta,l,t}$, 
namely defined as $e_u \Lambda_{\Delta,l,t} 
\cong \Hom_K(\Lambda_{\Delta,l,t} e_{\pi(u)},K)$ 
in $\mod \Lambda_{\Delta,l,t}$.
We can write $\pi$ as follows.
\begin{center}
\begin{tabular}{c|ccccccc}
$\Delta$ & $A_n$ ($n \notin 2\Z$) & $A_n$ ($n \in 2\Z$) & 
$D_n$ ($n \notin 2\Z$) & $D_n$ ($n \in 2\Z$) & $E_6$ & $E_7$ & $E_8$ \\
\hline
$\pi$ & $\tau^{-(n-1)/2} \psi$ & $\phi^{n-1}$ & 
$\tau^{-(n-2)} \psi$ & $\tau^{-(n-2)}$ & $\tau^{-5} \psi$ & 
$\tau^{-8}$ & $\tau^{-14}$
\end{tabular}
\end{center}
To make the calculation easier, 
the following proposition by Dugas is very crucial.
The part (2) is proved by applying $(S_u \otimes_\Lambda {?})$ to (1).
Note that we define 
the right action of $\lambda \in \Lambda$ 
on a twisted bimodule ${_1 \Lambda_\mu}$ 
as $\lambda' \cdot \lambda = \lambda' \mu(\lambda)$ for $\lambda' \in {_1 \Lambda_\mu}$,
whereas $\lambda' \cdot \lambda = \lambda' \mu^{-1}(\lambda)$ in \cite{Dugas}.

\begin{Prop}\label{proj_resol}
Let $\Delta$ be a Dynkin diagram and $k \ge 1$ be an integer.
Put $Q=Q_{\Delta,k,1}$, $\Lambda=\Lambda_{\Delta,k,1}$.
\begin{itemize}
\item[(1)]\cite[(4.1)--(4.3), Corollary 4.3]{Dugas}
A projective resolution of $\Lambda$ as
a $\Lambda$-$\Lambda$-bimodule is given by 
$0 \to L \to U_2 \to U_1 \to U_0 \to \Lambda \to 0$, 
where
\begin{align*}
U_2=\bigoplus_{u \in Q_0}(\Lambda e_u \otimes_K e_{\tau^{-1}u} \Lambda), \quad
U_1=\bigoplus_{u \in Q_0, \, v \in u^+}(\Lambda e_u \otimes_K e_v \Lambda), \quad
U_0=\bigoplus_{u \in Q_0}(\Lambda e_u \otimes_K e_u \Lambda),
\end{align*}
and $L$ is a $\Lambda$-$\Lambda$ twisted bimodule ${_1 \Lambda_\mu}$, 
and $\mu \in \operatorname{Aut}_K(\Lambda)$ satisfies $\mu^{-1}(e_u)=e_{\pi \tau^{-1}u}$
for $u \in Q_0$.
\item[(2)]
For $u \in Q_0$,
a projective resolution of the simple $\Lambda$-module $S_u$ is given by
\begin{align*}
& 0 \to S_{\pi \tau^{-1}u} \to P_{\tau^{-1}u} \to 
\bigoplus_{v \in u^+} P_v \to P_u \to S_u \to 0.
\end{align*}
\end{itemize}
\end{Prop}

\begin{Rem}
Let $c$ be the Coxeter number of $\Delta$.
It is well-known and easy to see 
that $(\pi \tau^{-1})^2=\tau^{-c}$.
The part (2) of Proposition \ref{proj_resol} implies that 
$[S_u]=-[S_{\pi \tau^{-1}u}]$ and thus $[S_u]=[S_{\tau^{-c}u}]$ in $K_0(\umod \Lambda)$.
This observation also holds 
for the locally bounded mesh algebra $K(\Z \Delta)$ of $\Z \Delta$,
and we have $K_0(\umod K(\Z \Delta)) \cong K_0(\umod \Lambda_{\Delta,1,c})$.
This is isomorphic to $\Z^{n(c-2)/2}$ by Theorem \ref{main},
where $n$ is the number of vertices of $\Delta$.
\end{Rem}

Now, we can prove the following lemma, which gives 
``simpler'' generators of $\langle [P_u] \mid u \in Q_0 \rangle$,
and is the key ingredient of the calculation.
Though the number of generators may increase,
the elements of the new generators are much easier to 
express as linear combinations of $[S_u]$'s than the original ones.

\begin{Lem}\label{change_gen}
Let $\Delta$ be a Dynkin diagram and $k \ge 1$ be an integer.
Define $H_{\Delta,k} = \langle [P_u] \mid u \in (Q_{\Delta,k})_0 \rangle$
as a subgroup of $K_0(D^\rb(\mod \Lambda_{\Delta,k,1}))$. 
Then the following conditions hold.
\begin{itemize}
\item[(1)]
\begin{itemize}
\item[(i)]
Let $H'_{\Delta,k},H''_{\Delta,k}$ be subgroups of 
$K_0(D^\rb(\mod \Lambda_{\Delta,k,1}))$ defined by
\begin{align*}
H'_{\Delta,k}=
\langle [S_u]+[S_{\pi\tau^{-1}u}] 
\mid u \in (Q_{\Delta,k,1})_0 \rangle, \quad
H''_{\Delta,k}=
\langle [P_u] \mid u \in \{1,n\} \times (\Z/k\Z) \rangle.
\end{align*}
Then we have $H_{\Delta,k}=H'_{\Delta,k}+H''_{\Delta,k}$, and especially,
\begin{align*}
K_0(\umod \Lambda_{\Delta,k,1}) \cong 
K_0(D^\rb(\mod \Lambda_{\Delta,k,1}))/(H'_{\Delta,k}+H''_{\Delta,k}).
\end{align*}
\item[(ii)]
If $\Delta=A_n$, (i) holds even if 
\begin{align*}
H''_{\Delta,k}=\langle [P_u] \mid u \in \{1\} \times (\Z/k\Z) \rangle.
\end{align*}
\item[(iii)]
If $\Delta=D_n$, (i) holds even if
\begin{align*}
H'_{\Delta,k}=\langle [S_u]+[S_{\pi\tau^{-1}u}] 
\mid u \in \{1,\ldots,n-2\} \times (\Z/k\Z) \rangle.
\end{align*}
\end{itemize}
\item[(2)]
\begin{itemize}
\item[(i)]
If $\Delta=A_n \ (n \notin 2\Z), D_n, E_6$,
let $H^\psi_{\Delta,2k}$ be a subgroup of $K_0(D^\rb(\mod \Lambda_{\Delta,2k,1}))$
defined by
\begin{align*}
H^\psi_{\Delta,2k}&=\langle [S_u]-[S_{(\tau^k \psi)^{-1} u}] 
\mid u \in (Q_{\Delta,2k,1})_0 \rangle.
\end{align*}
Then we have
$K_0(\umod \Lambda_{\Delta,2k,2}) \cong 
K_0(D^\rb(\mod \Lambda_{\Delta,2k,1}))/
(H'_{\Delta,2k}+H''_{\Delta,2k}+H^\psi_{\Delta,2k})$.
Moreover, if $\Delta=D_n$, 
we have 
\begin{align*}
H^\psi_{\Delta,2k}=\langle [S_u]-[S_{(\tau^k \psi)^{-1} u}] 
\mid u \in \{1,\ldots,n-2,n\} \times (\Z/2k\Z) \rangle.
\end{align*}
\item[(ii)]
If $\Delta=A_n \ (n \in 2\Z)$,
let $H^\phi_{\Delta,2k-1}$ be a subgroup of $K_0(D^\rb(\mod \Lambda_{\Delta,2k-1,1}))$
defined by
\begin{align*}
H^\phi_{\Delta,2k-1}&=\langle [S_u]-[S_{(\tau^k \phi)^{-1} u}] 
\mid u \in (Q_{\Delta,2k-1,1})_0 \rangle.
\end{align*}
Then we have
$K_0(\umod \Lambda_{\Delta,2k-1,2}) \cong 
K_0(D^\rb(\mod \Lambda_{\Delta,2k-1,1}))/
(H'_{\Delta,2k-1}+H''_{\Delta,2k-1}+H^\phi_{\Delta,2k-1})$.
\item[(iii)]
If $\Delta=D_4$,
let $H^\chi_{\Delta,3k}$ be a subgroup of $K_0(D^\rb(\mod \Lambda_{\Delta,3k,1}))$
defined by
\begin{align*}
H^\chi_{\Delta,3k}&=\langle [S_u]-[S_{(\tau^k \chi)^{-1} u}] 
\mid u \in (Q_{\Delta,3k,1})_0 \rangle \\
&=\langle [S_u]-[S_{(\tau^k \chi)^{-1} u}] 
\mid u \in \{3,2,4\} \times (\Z/3k\Z) \rangle.
\end{align*}
Then we have
$K_0(\umod \Lambda_{\Delta,3k,3}) \cong 
K_0(D^\rb(\mod \Lambda_{\Delta,3k,1}))
/({H'_{\Delta,3k}+H''_{\Delta,3k}+H^\chi_{\Delta,3k}})$.
\end{itemize} 
\end{itemize}
\end{Lem}

\begin{proof}
(1)
We prove (i) first.

We show $H'_{\Delta,k}+H''_{\Delta,k} \subset H_{\Delta,k}$.
It is clear that $H''_{\Delta,k} \subset H_{\Delta,k}$.
We prove $H'_{\Delta,k} \subset H_{\Delta,k}$.
By Proposition \ref{proj_resol}, we have
\begin{align*}
[S_{\pi\tau^{-1}u}]+[S_u]=[P_{\tau^{-1}u}]-\sum_{v \in u^+}[P_v]+[P_u] \in H_{\Delta,k}
\end{align*}
in $K_0(D^\rb(\mod \Lambda_{\Delta,k,1}))$.
Thus $H'_{\Delta,k} \subset H_{\Delta,k}$ holds and 
we have $H'_{\Delta,k}+H''_{\Delta,k} \subset H_{\Delta,k}$.

Next, we show $H_{\Delta,k} \subset H'_{\Delta,k}+H''_{\Delta,k}$.
It is sufficient to show $[P_{i,a+k\Z}] \in H'_{\Delta,k}+H''_{\Delta,k}$.
%
%
If $i=n$, we have $[P_{n,a+k\Z}] \in H''_{\Delta,k}$.
Thus we prove the remained assertion by induction on $i=1,\ldots,n-1$.
If $i=1$, we have $[P_{1,a+k\Z}] \in H''_{\Delta,k}$.
We assume that $2 \le i \le n-1$.
Put $u=(i-1,a+k\Z)$ and let $m$ be the number of the elements of $u^+$
(it does not depend on $a$).
We can deduce that $m=1,2,3$.

If $m=1$, we can deduce $i-1=1$ because $1 \le i-1 \le n-2$, and we have $u^+=\{(2,a+k\Z)\}$.
By Proposition \ref{proj_resol}, we have
\begin{align*}
[S_{\pi\tau^{-1}u}]-[P_{1,a+1+k\Z}]
+[P_{2,a+k\Z}]-[P_{1,a+k\Z}]+[S_u]=0 
\end{align*}
and 
\begin{align*}
[P_{2,a+k\Z}]=-([S_{\pi\tau^{-1}u}]+[S_u])
+([P_{1,a+1+k\Z}]+[P_{1,a+k\Z}]).
\end{align*}
Thus we can deduce $[P_{i,a+k\Z}] \in H'_{\Delta,k}+H''_{\Delta,k}$.

If $m=2$, we can deduce $2 \le i-1 \le n-2$ and $u^+=\{(i-2,a+1+k\Z),(i,a+k\Z)\}$. 
By Proposition \ref{proj_resol}, we have
\begin{align*}
[S_{\pi\tau^{-1}u}]-[P_{i-1,a+1+k\Z}]
+([P_{i-2,a+1+k\Z}]+[P_{i,a+k\Z}])-[P_{i-1,a+k\Z}]+[S_u]=0 
\end{align*}
and 
\begin{align*}
[P_{i,a+k\Z}]=-([S_{\pi\tau^{-1}u}]+[S_u])
+([P_{i-1,a+1+k\Z}]-[P_{i-2,a+1+k\Z}]+[P_{i-1,a+k\Z}]).
\end{align*}
The first term of the right-hand side is the element of $H'_{\Delta,k}$,
the second term belongs to $H'_{\Delta,k}+H''_{\Delta,k}$ by the induction hypothesis.
Thus we can deduce $[P_{i,a+k\Z}] \in H'_{\Delta,k}+H''_{\Delta,k}$.

If $m=3$, we can deduce $2 \le i-1 \le n-2$ and 
$u^+=\{(i-2,a+1+k\Z),(i,a+k\Z),(n,a+k\Z)\}$.
We have similarly 
\begin{align*}
[P_{i,a+k\Z}]=-([S_{\pi\tau^{-1}u}]+[S_u])
+([P_{i-1,a+1+k\Z}]-[P_{i-2,a+1+k\Z}]+[P_{i-1,a+k\Z}])-[P_{n,a+k\Z}].
\end{align*}
The first term of the right-hand side is the element of $H'_{\Delta,k}$,
the second term belongs to $H'_{\Delta,k}+H''_{\Delta,k}$ by the induction hypothesis,
and the third term is the element of $H''_{\Delta,k}$.
Thus we can deduce $[P_{i,a+k\Z}] \in H'_{\Delta,k}+H''_{\Delta,k}$.

The induction is completed, 
and we have $H_{\Delta,k} \subset H'_{\Delta,k}+H''_{\Delta,k}$.
The part (i) has been proved.

If $\Delta=A_n$, in the proof of 
$H_{\Delta,k} \subset H'_{\Delta,k}+H''_{\Delta,k}$ in (i), 
we can add 
\begin{align*} 
[P_{n,a+k\Z}]=-([S_{\pi\tau^{-1}u}]+[S_u])
+([P_{n-1,a+1+k\Z}]-[P_{n-2,a+1+k\Z}]+[P_{n-1,a+k\Z}]).
\end{align*}
to the case of $m=2$, where $u=(n-1,a+k\Z)$.
The case $m=3$ does not occur.
Thus we can remove $[P_{n,a+k\Z}]$ from the generators of $H''_{\Delta,k}$.
The part (ii) is proved.

If $\Delta=D_n$, in the proof of 
$H_{\Delta,k} \subset H'_{\Delta,k}+H''_{\Delta,k}$ in (i), 
the fact $[S_{\pi \tau^{-1}u}]+[S_u] \in H'_{\Delta,k}$ is used 
only in the case $u \in \{ 1,\ldots,n-2 \} \times (\Z/k\Z)$.
The part (iii) is proved.

(2)
We only prove (i). The remained parts (ii) and (iii) are shown similarly. 
The natural quiver epimorphism $Q_{\Delta,2k,1} \to Q_{\Delta,2k,2}$
by $\tau^k \psi$ induces the natural epimorphism
$K_0(D^\rb(\mod \Lambda_{\Delta,2k,1})) \to K_0(D^\rb(\mod \Lambda_{\Delta,2k,2}))$,
and it has $H^\psi_{\Delta,2k}$ as its cokernel.
Now we put 
$H_{\Delta,2k,2}=\langle [e_{\bar{u}} \cdot \Lambda_{\Delta,2k,2}] 
\mid \bar{u} \in (Q_{\Delta,2k,2})_0 \rangle
\subset K_0(D^\rb(\mod \Lambda_{\Delta,2k,2}))$.
For $u,v \in (Q_{\Delta,2k,1})_0$, we have
$e_{\bar{u}} \cdot \Lambda_{\Delta,2k,2} \cdot e_{\bar{v}} \cong
e_u \cdot \Lambda_{\Delta,2k,1} \cdot e_v 
\oplus e_u \cdot \Lambda_{\Delta,2k,1} \cdot e_{\tau^k \psi v}$
as $K$-vector spaces,
where $\bar{u}=\{u,\tau^k\psi u\}$ and $\bar{v}=\{v,\tau^k\psi v\}$ are 
the $\tau^k \psi$-orbits of $u$ and $v$, respectively.
Therefore, we have the following exact sequences and the commutative diagram
\begin{align*}
\begin{xy}
(-56,16) *+{0} ="1",
(-32,16) *+{H_{\Delta,2k} \cap H^\psi_{\Delta,2k}} ="2",
(  0,16) *+{H_{\Delta,2k}} ="3",
( 48,16) *+{H_{\Delta,2k,2}} ="4",
( 80,16) *+{0} ="5",
(-56, 0) *+{0} ="6",
(-32, 0) *+{H^\psi_{\Delta,2k}} ="7",
(  0, 0) *+{K_0(D^\rb(\mod \Lambda_{\Delta,2k,1}))} ="8",
( 48, 0) *+{K_0(D^\rb(\mod \Lambda_{\Delta,2k,2}))} ="9",
( 80, 0) *+{0} ="10"
\ar "1";"2"
\ar "2";"3"
\ar "3";"4"
\ar "4";"5"
\ar "6";"7"
\ar "7";"8"
\ar "8";"9"
\ar "9";"10"
\ar "2";"7"
\ar "3";"8"
\ar "4";"9"
\end{xy}.
\end{align*}
By diagram chasings, we have
\begin{align*}
K_0(\umod \Lambda_{\Delta,2k,2}) 
\cong \frac{K_0(D^\rb(\mod \Lambda_{\Delta,2k,2}))}{H_{\Delta,2k,2}} 
\cong \frac{K_0(D^\rb(\mod \Lambda_{\Delta,2k,1}))}{H_{\Delta,2k}+H^\psi_{\Delta,2k}}.
\end{align*}
The first assertion is proved by using the part (1),
and the second assertion for $\Delta=D_n$ is easy to prove by the definition of $\psi$.
\end{proof}

Our task is moved to express the generators of 
the subgroups appearing in the previous lemma.
To do this, we define some matrices on $\Z[x]$ and $\Z$.

\begin{Def}
For an integer $n \ge 1$, we define the following.
\begin{itemize}
\item[(1)]
We define $T_n(x) \in \Mat_{n,n}(\Z[x])$, $U_n(x) \in \Mat_{n,1}(\Z[x])$ as
\begin{align*}
T_n(x)=\begin{bmatrix}
  &     &                              & x^n \\
  &     & \mbox{\reflectbox{$\ddots$}} &     \\   
  & x^2 &                              &     \\
x &     &                              &     
\end{bmatrix}, \quad
U_n(x)=\begin{bmatrix} 1 \\ 1 \\ \vdots \\ 1 \end{bmatrix}.
\end{align*}
\item[(2)]
Assume $n \ge 4$.
We define $V_n(x) \in \Mat_{n-2,1}(\Z[x])$, $W_n(x) \in \Mat_{n-2,1}(\Z[x])$ as
\begin{align*}
V_n(x)=\begin{bmatrix} 1+x^{n-2} \\ \vdots \\ 1+x^2 \\ 1+x \end{bmatrix}, \quad
W_n(x)=\begin{bmatrix} 
x^{n-2} \\
\vdots \\
x^2+\cdots+x^{n-2} \\
x+x^2+\cdots+x^{n-2} 
\end{bmatrix}.
\end{align*}
\item[(3)]
Assume $n \ge 4$.
We define $f_n(x),g_n(x) \in \Z[x]$ as 
\begin{align*}  
& f_n(x)=\begin{cases}
1+x^2+\cdots+x^{n-3} & (n \notin 2\Z) \\
1+x^2+\cdots+x^{n-2} & (n \in    2\Z)
\end{cases}, \quad
g_n(x)=\begin{cases}
x+x^3+\cdots+x^{n-2} & (n \notin 2\Z) \\
x+x^3+\cdots+x^{n-3} & (n \in    2\Z)
\end{cases}.
\end{align*}
\end{itemize}
\end{Def}

\begin{Lem}
The Grothendieck group
$K_0(\umod \Lambda_{\Delta,l,t})$ is isomorphic to $\Coker M_{\Delta,l,t}(X_l)$,
where $M_{\Delta,l,t}(x)$ is a matrix on $\Z[x,x^{-1}]$ defined as follows;
\begin{align*}
& (\mathrm{I}) & M_{A_n,k,1}(x)&=\begin{bmatrix} 1_n+T_n(x) & U_n(x) \end{bmatrix}, \\
& (\mathrm{II}) & M_{A_n,2k,2}(x)&=\begin{bmatrix} 
1_n+T_n(x) & U_n(x) & 1_n-x^{k-(n+1)/2} T_n(x) 
\end{bmatrix} \quad (n \notin 2\Z), \\
& (\mathrm{III}) & M_{A_n,2k-1,2}(x)&=
\begin{bmatrix} 1_n+T_n(x) & U_n(x) & 1_n-x^{k-(n+2)/2} T_n(x) \end{bmatrix}
\quad (n \in 2\Z), \\
& (\mathrm{IV}) & M_{D_n,k,1}(x)&=\begin{bmatrix}
(1+x^{n-1}) \cdot 1_{n-2} & V_n(x) & W_n(x) \\
                          & 1      & g_n(x) \\
                          & 1      & f_n(x)
\end{bmatrix}, \\
& (\mathrm{V}) & M_{D_n,2k,2}(x)&=\begin{bmatrix}
(1+x^{n-1}) \cdot 1_{n-2} & V_n(x) & W_n(x) & (1-x^k) \cdot 1_{n-2} &      \\
                          & 1      & g_n(x) &                       & -x^k \\
                          & 1      & f_n(x) &                       & 1    
\end{bmatrix}, \\
& (\mathrm{VI}) & M_{D_n,3k,3}(x)&=\begin{bmatrix}
1+x^3 &       & 1  +x^2 &     x^2 & -x^{k+1} &       &       \\
      & 1+x^3 & 1+x     &   x+x^2 &          & 1-x^k &       \\
      &       & 1       &   x     & 1        &       & -x^k  \\
      &       & 1       & 1  +x^2 &          &       & 1    
\end{bmatrix}, \\
& (\mathrm{VII}) & M_{E_6,k,1}(x)&=\begin{bmatrix} \begin{bmatrix}
1_5+x^3 \cdot T_5(x) &       \\
                     & 1+x^6
\end{bmatrix} & \left[ \begin{smallmatrix}
1      +x^3 &          x^3+    x^5 \\
1  +x^2+x^3 &     x^2+ x^3+x^4+x^5 \\
1+x+x^2+x^3 &   x+x^2+2x^3+x^4+x^5 \\
1+x    +x^3 &   x+x^2+ x^3+x^4     \\
1      +x^3 &   x    + x^3         \\
1  +x^2     & 1  +x^2+ x^3    +x^5  
\end{smallmatrix} \right] \end{bmatrix}, \\
& (\mathrm{VIII}) & M_{E_6,2k,2}(x)
&=\begin{bmatrix}M_{E_6,2k,1}(x) & 
\begin{bmatrix}
1_5-x^{k-3} \cdot T_5(x) &       \\
                         & 1-x^k
\end{bmatrix} \end{bmatrix}, \\
& (\mathrm{IX}) & M_{E_7,k,1}(x)&=\begin{bmatrix} (1+x^9) \cdot 1_7 & \left[ \begin{smallmatrix}
1          +x^4            +x^8 &               x^4     + x^6    +x^8 \\ 
1      +x^3+x^4        +x^7     &          x^3+ x^4+ x^5+ x^6+x^7+x^8 \\
1  +x^2+x^3+x^4    +x^6         &     x^2+ x^3+2x^4+ x^5+2x^6+x^7+x^8 \\
1+x+x^2+x^3+x^4+x^5             &   x+x^2+2x^3+2x^4+2x^5+2x^6+x^7+x^8 \\
1+x    +x^3+x^4                 &   x+x^2+ x^3+2x^4+ x^5+ x^6+x^7     \\
1      +x^3                     &   x    + x^3+ x^4     + x^6         \\
1  +x^2    +x^4                 & 1  +x^2+ x^3+ x^4+ x^5+ x^6    +x^8
\end{smallmatrix} \right] \end{bmatrix}, \\
& (\mathrm{X}) & M_{E_8,k,1}(x)&=\begin{bmatrix} (1+x^{15}) \cdot 1_8 & (1+x^5) \cdot 
\left[ \begin{smallmatrix}
1                              +x^9 &                    x^5     + x^7    +x^9 \\
1          +x^4            +x^8     &               x^4+ x^5+ x^6+ x^7+x^8+x^9 \\
1      +x^3+x^4        +x^7         &          x^3+ x^4+2x^5+ x^6+2x^7+x^8+x^9 \\
1  +x^2+x^3+x^4    +x^6             &     x^2+ x^3+2x^4+2x^5+2x^6+2x^7+x^8+x^9 \\
1+x+x^2+x^3+x^4+x^5                 &   x+x^2+2x^3+2x^4+3x^5+2x^6+2x^7+x^8+x^9 \\
1+x    +x^3+x^4                     &   x+x^2+ x^3+2x^4+2x^5+ x^6+ x^7+x^8     \\
1      +x^3                         &   x    + x^3+ x^4+ x^5     + x^7         \\
1  +x^2    +x^4                     & 1  +x^2+ x^3+ x^4+ x^5+ x^6+ x^7    +x^9
\end{smallmatrix} \right] \end{bmatrix}.
\end{align*}
\end{Lem}

\begin{proof}
We correspond the $((i-1)l+(a+1))$th row of the matrix $M_{\Delta,l,t}(X_l)$
to $[S_{i,a+l\Z}] \in K_0(D^\rb(\mod \Lambda_{\Delta,l,1}))$
for $i \in \{1,\ldots,n\}$ and $a \in \{0,\ldots,l-1\}$.
Calculating the dimension vectors of the indecomposable projective modules
appearing in the generators of $H''_{\Delta,l}$,
one can straightly check the columns of $M_{\Delta,l,t}(X_l)$ correspond to 
the generators of the subgroups
$H'_{\Delta,l}, H''_{\Delta,l}, H^\psi_{\Delta,l}, H^\phi_{\Delta,l}, H^\chi_{\Delta,l}$
given in Lemma \ref{change_gen}.
Now the assertion is proved by Lemma \ref{quot_by_P}.
\end{proof}

\subsection{Transformations of polynomial matrices}

Now we can finish the proof of Proposition \ref{tanninsi}.
The method of the proof is the transformation of the polynomial matrices 
$M_{\Delta,l,t}(x)$.

\begin{proof}[Proof of Proposition \ref{tanninsi}]
We can transform $M_{\Delta,l,t}(x)$ as a matrix on $\Z[x]/(1-x^l)$.
In such transformations, we can put $x=X_l$, because $X_l^l=1_l$.
We use the following fact (*).
\begin{quote}
Let $A \in \Mat_{m,*}(R)$ and $B \in \Mat_{m,m}(R)$ be matrices on a ring $R$ 
with the numbers of their rows are the same.
Assume that $B$ is scalar; 
that is, there exists $\lambda \in R$ such that $B=\lambda \cdot 1_m$.
If the matrix $A$ can be transformed into $A'$ as matrices on $R$,
then the matrix $\begin{bmatrix} A & B \end{bmatrix}$ can be transformed into
$\begin{bmatrix} A' & B \end{bmatrix}$.
\end{quote}
Now we start transformations.

(I: $\Lambda_{A_n,k,1}$)
We consider the case $n \notin 2\Z$ first.
If $n=1$, it is clear, so we assume $n \ge 3$.
$\begin{bmatrix} 1_n+T_n(x) & U_n(x) \end{bmatrix}$ is of the form
\begin{align*}
\begin{bmatrix}
1 &        &             &               &             &        & x^n & 1      \\
  & \cdots &             &               &             & \cdots &     & \cdots \\
  &        & 1           &               & x^{(n+3)/2} &        &     & 1      \\
  &        &             & 1+x^{(n+1)/2} &             &        &     & 1      \\
  &        & x^{(n-1)/2} &               & 1           &        &     & 1      \\
  & \cdots &             &               &             & \cdots &     & \cdots \\
x &        &             &               &             &        & 1   & 1      \\
\end{bmatrix}.
\end{align*}
Using the left-upper ``1''s, it can be transformed into
\begin{align*}
\begin{bmatrix}
1_{(n-1)/2} &               &           &        &           &               \\
            & 1+x^{(n+1)/2} &           &        &           & 1             \\
            &               & 1-x^{n+1} &        &           & 1-x^{(n-1)/2} \\
            &               &           & \cdots &           & \cdots        \\
            &               &           &        & 1-x^{n+1} & 1-x           \\
\end{bmatrix}
\end{align*}
and
\begin{align*}
\begin{bmatrix}
1_{(n-1)/2} &               &                               &           &     \\
            & 1+x^{(n+1)/2} &                               &           & 1   \\
            &               & (1-x^{n+1}) \cdot 1_{(n-3)/2} &           &     \\
            &               &                               & 1-x^{n+1} & 1-x \\
\end{bmatrix}.
\end{align*}
Finally, we get 
\begin{align*}
\begin{bmatrix}
1_{(n-1)/2} &                       &                               &           &     \\
            & 0                     &                               &           & 1   \\
            &                       & (1-x^{n+1}) \cdot 1_{(n-3)/2} &           &     \\
            & -(1-x)(1+x^{(n+1)/2}) &                               & 1-x^{n+1} &     \\
\end{bmatrix}.
\end{align*}
Because $1-x^{n+1}$ can be divided by $(1-x)(1+x^{(n+1)/2})$, 
we have the assertion.

If $n \in 2\Z$, omitting the middle row and the middle column, 
similar transformations give
\begin{align*}
\begin{bmatrix}
1_{n/2} &                               &           &     \\
        & (1-x^{n+1}) \cdot 1_{(n-2)/2} &           &     \\
        &                               & 1-x^{n+1} & 1-x
\end{bmatrix}
\end{align*}
Because $1-x^{n+1}$ can be divided by $1-x$,
the assertion is proved.

(II: $\Lambda_{A_n,2k,2}$)
The matrix $M_{A_n,2k,2}(x)$ can be transformed into
\begin{align*}
\begin{bmatrix} 1_n+T_n(x) & U_n(x) & (-1-x^{k-(n+1)/2})T_n(x) \end{bmatrix}.
\end{align*}
Taking into account that $T_n(x) \in \GL_n(\Z[x]/(1-x^{2k}))$, 
the above matrix can be transformed into 
\begin{align*}
\begin{bmatrix} 1_n+T_n(x) & U_n(x) & (1+x^{k-(n+1)/2}) \cdot 1_n \end{bmatrix}.
\end{align*}
From the proof for $\Lambda_{A_n,k,1}$, 
the matrix $\begin{bmatrix} 1_n+T_n(x) & U_n(x) \end{bmatrix}$
is transformed into 
\begin{align*}
N(x)=\begin{bmatrix}
1_{(n-1)/2} &                       &                               &   &   \\
            & 0                     &                               &   & 1 \\
            &                       & (1-x^{n+1}) \cdot 1_{(n-3)/2} &   &   \\
            & -(1-x)(1+x^{(n+1)/2}) &                               & 0 &   \\
\end{bmatrix}.
\end{align*}
Therefore, from the fact (*), $M_{A_n,2k,2}(x)$ can be transformed into 
$\begin{bmatrix} N(x) & (1+x^{k-(n+1)/2}) \cdot 1_n\end{bmatrix}$,
and we have
\begin{align*}
K_0(\umod \Lambda_{A_n,2k,2})
& \cong (\Coker 
\begin{bmatrix} 1_{2k} & 1_{2k}+X_{2k}^{k-(n+1)/2} \end{bmatrix})^{(n+1)/2} \\
&\quad \oplus (\Coker \begin{bmatrix} 
1_{2k}-X_{2k}^{n+1} & 1+X_{2k}^{k-(n+1)/2} \end{bmatrix})^{(n-3)/2} \\
&\quad \oplus \Coker \begin{bmatrix} (1_{2k}-X_{2k})(1_{2k}+X_{2k}^{(n+1)/2}) & 
1_{2k}+X_{2k}^{k-(n+1)/2} \end{bmatrix}.
\end{align*}
The first component is clearly 0,
and the second one is isomorphic to
$(\Coker (1_{2k}+X_{2k}^{k-(n+1)/2}))^{(n-3)/2}$,
because we have $1-x^{n+1}=-x^{n+1}(1+x^{k-(n+1)/2})(1-x^{k-(n+1)/2})$ 
in $\Z[x]/(1-x^{2k})$.
Thus we have the assertion.

(III: $\Lambda_{A_n,2k-1,2}$)
Similarly to the proof for $\Lambda_{A_n,2k,2}$,
the matrix $M_{A_n,2k-1,2}(x)$ can be transformed into 
$\begin{bmatrix} N(x) & (1+x^{k-(n+2)/2}) \cdot 1_n\end{bmatrix}$, where 
\begin{align*}
N(x)=\begin{bmatrix}
1_{n/2} &                               &   &     \\
        & (1-x^{n+1}) \cdot 1_{(n-2)/2} &   &     \\
        &                               & 0 & 1-x
\end{bmatrix}.
\end{align*}
We have 
\begin{align*}
K_0(\umod \Lambda_{A_n,2k-1,2}) 
&\cong(\Coker \begin{bmatrix} 
1_{2k-1} & 1_{2k-1}+X_{2k-1}^{k-(n+2)/2} \end{bmatrix})^{n/2} \\
&\quad \oplus (\Coker \begin{bmatrix} 
1_{2k-1}-X_{2k-1}^{n+1} & 1_{2k-1}+X_{2k-1}^{k-(n+2)/2} \end{bmatrix})^{(n-2)/2} \\
&\quad \oplus \Coker \begin{bmatrix} 1_{2k-1}-X_{2k-1} & 1_{2k-1}+X_{2k-1}^{k-(n+2)/2} 
\end{bmatrix}.
\end{align*}
The first component is clearly 0,
and the second one is isomorphic to $(\Coker (1_{2k-1}+X_{2k-1}^{k-(n+2)/2}))^{(n-2)/2}$,
because we have $1-x^{n+1}=-x^{n+1}(1+x^{k-(n+2)/2})(1-x^{k-(n+2)/2})$ 
in $\Z[x]/(1-x^{2k-1})$.
The last summand is isomorphic to
$\Coker \begin{bmatrix} 1_{2k-1}-X_{2k-1} & 2 \cdot 1_{2k-1}
\end{bmatrix}$.
Thus we have the assertion.

(IV: $\Lambda_{D_n,k,1}$)
Multiplying the matrix below (invertible on $\Z[x]$) to $M_{D_n,k,1}(x)$ from the left,
\begin{align*}
\begin{bmatrix}
1_{n-2} & -U_{n-2}(x) & -U_{n-2}(x) \\
        & 1           &             \\
        &             & 1 
\end{bmatrix}
\end{align*} 
we have
\begin{align*}
\begin{bmatrix}
1+x^{n-1} &        &           &           & -1+x^{n-2} & -1-x-\cdots-x^{n-3} \\
          & \cdots &           &           & \cdots     & \cdots              \\
          &        & 1+x^{n-1} &           & -1+x^2     & -1-x                \\
          &        &           & 1+x^{n-1} & -1+x       & -1                  \\
          &        &           &           &  1         & g_n(x)              \\
          &        &           &           &  1         & f_n(x)
\end{bmatrix}.
\end{align*} 
This matrix can be transformed into
\begin{align*}
\begin{bmatrix}
(1+x^{n-1}) \cdot 1_{n-3} &           &            &                     \\
                          & 1+x^{n-1} & -1+x       & -1                  \\
                          &           &  1         & g_n(x)              \\
                          &           &  1         & f_n(x)
\end{bmatrix}.
\end{align*} 
Thus we have
$\Coker M_{D_n,k,1}(X_k) \cong (\Coker (1+X_k^{n-1}))^{n-3} \oplus \Coker M_1(X_k)$,
where
\begin{align*}
M_1(x)=\begin{bmatrix}
1+x^{n-1} & -1+x & -1     \\
          &  1   & g_n(x) \\
          &  1   & f_n(x) 
\end{bmatrix}.
\end{align*}
$M_1(x)$ can be transformed into
\begin{align*}
M_2(x)=\begin{bmatrix}
1+x^{n-1} &  1+x & -1+g_n(x)+f_n(x) \\
          &  1   & g_n(x)           \\
          &  1   & f_n(x) 
\end{bmatrix}.
\end{align*}

If $n \notin 2\Z$, considering the $(3,2)$ entry and the equations
\begin{align*}
-1+g_n(x)+f_n(x)=-1+(1+x)f_n(x), \quad g_n(x)-f_n(x)=-s_{n-1}(x),
\end{align*}
$M_2(x)$ can be transformed into
\begin{align*}
\begin{bmatrix}
1+x^{n-1} &      & -1          \\
          &      & -s_{n-1}(x) \\
          &  1   & 
\end{bmatrix} 
\end{align*}
and using $(1+x^{n-1})s_{n-1}(x)=s_{2n-2}(x)$, we have
\begin{align*} 
\begin{bmatrix}
             &      & -1 \\
-s_{2n-2}(x) &      &    \\
             &  1   & 
\end{bmatrix}. 
\end{align*} 
The assertion is proved for $n \notin 2\Z$.

If $n \in 2\Z$, considering the $(2,2)$ entry and the equations
\begin{align*}
-1+g_n(x)+f_n(x)=(1+x)g_n(x), \quad f_n(x)-g_n(x)=s_{n-1}(x),
\end{align*}
$M_2(x)$ can be transformed into
\begin{align*}
\begin{bmatrix}
1+x^{n-1} &      &            \\
          &  1   &            \\
          &      & s_{n-1}(x)
\end{bmatrix}.
\end{align*}
The assertion is proved.

(V: $\Lambda_{D_n,2k,2}$)
By similar calculations to the proof for $\Lambda_{D_n,k,1}$, we have
\begin{align*}
\Coker M_{D_n,2k,2}(X_{2k}) \cong (\Coker \begin{bmatrix}
1+X_{2k}^{n-1} & 1-X_{2k}^k
\end{bmatrix})^{n-3} \oplus \Coker M_1(X_{2k}),
\end{align*}
where
\begin{align*}
M_1(x)=\begin{bmatrix}
1+x^{n-1} & -1+x & -1     & 1-x^k &      \\
          &  1   & g_n(x) &       & -x^k \\
          &  1   & f_n(x) &       &  1   
\end{bmatrix}.
\end{align*}
$M_1(x)$ can be transformed into
\begin{align*}
M_2(x)=\begin{bmatrix}
1+x^{n-1} &  1+x & -1+g_n(x)+f_n(x) & 1-x^k &      \\
          &  1   & g_n(x)           &       & -x^k \\
          &  1   & f_n(x)           &       &  1   
\end{bmatrix}.
\end{align*}

If $n \notin 2\Z$, considering the $(3,2)$ entry and the equations
\begin{align*}
-1+g_n(x)+f_n(x)=-1+(1+x)f_n(x), \quad g_n(x)-f_n(x)=-s_{n-1}(x),
\end{align*}
$M_2(x)$ can be transformed into
\begin{align*}
\begin{bmatrix}
1+x^{n-1} &      & -1          & 1-x^k & -1-x   \\
          &      & -s_{n-1}(x) &       & -1-x^k \\
          &  1   & 
\end{bmatrix} 
\end{align*}
and using $(1+x^{n-1})s_{n-1}(x)=s_{2n-2}(x)$ and $(1+x)s_{n-1}(x)=1-x^{n-1}$, we have
\begin{align*}
\begin{bmatrix}
             &      & -1 &                    &              \\
-s_{2n-2}(x) &      &    & -(1-x^k)s_{n-1}(x) & -x^k-x^{n-1} \\
             &  1   &    &                    &  
\end{bmatrix}. 
\end{align*} 
Because $n$ is odd, we have
\begin{align*}
(1-x^k)s_{n-1}(x)=(1+x^{n-1})s_{n-1}(x)+(-x^k-x^{n-1})s_{n-1}(x)
=s_{2n-2}(x)+(-x^k-x^{n-1})s_{n-1}(x).
\end{align*}
The assertion is proved.

If $n \in 2\Z$, considering the $(2,2)$ entry and the equations
\begin{align*}
-1+g_n(x)+f_n(x)=(1+x)g_n(x),\quad
f_n(x)-g_n(x)=s_{n-1}(x),
\end{align*}
$M_2(x)$ can be transformed into
\begin{align*}
\begin{bmatrix}
1+x^{n-1} &      &            & 1-x^k & x^k+x^{k+1} \\
          &  1   &            &       &             \\
          &      & s_{n-1}(x) &       & 1+x^k     
\end{bmatrix}
\end{align*}
and we have
\begin{align*}
\begin{bmatrix}
 1+x^{n-1} &      &            &  1-x^k & x^k+x^{k+1} \\
           &  1   &            &        &             \\
-1-x^{n-1} &      & s_{n-1}(x) & -1+x^k & 1-x^{k+1}     
\end{bmatrix}.
\end{align*}
Now that $s_{n-1}(x)$ divides $1+x^{n-1}$, thus we have
\begin{align*}
\begin{bmatrix}
 1+x^{n-1} &      &            &  1-x^k & x+x^k \\
           &  1   &            &        &       \\
           &      & s_{n-1}(x) & -1+x^k & 1-x        
\end{bmatrix}.
\end{align*}
Because $n$ is even, $s_{n-1}(x)=1-(1-x)g_n(x)$ holds.
Thus, transformations lead to
\begin{align*}
\begin{bmatrix}
 1+x^{n-1} &      & (x+x^k)g_n(x) &  1-x^k & x+x^k \\
           &  1   &               &        &       \\
           &      & 1             & -1+x^k & 1-x        
\end{bmatrix}
\end{align*}
and 
\begin{align*}
\begin{bmatrix}
 1+x^{n-1} &      &   &  (1-x^k)(1+(x+x^k)g_n(x)) & (x+x^k)(1-(1-x)g_n(x)) \\
           &  1   &   &                           &                        \\
           &      & 1 &                           &         
\end{bmatrix}.
\end{align*}
Here, in $\Z[x]/(1-x^{2k})$, the equations
\begin{align*}
(1-x^k)(1+(x+x^k)g_n(x)) &= (1-x^k)+(x+x^k-x^{k+1}-x^{2k})g_n(x) \\
&= (1-x^k)+(-1+x+x^k-x^{k+1})g_n(x) \\
& = (1-x^k)(1-(1-x)g_n(x)) = (1-x^k)s_{n-1}(x), \\
(x+x^k)(1-(1-x)g_n(x)) &= (x+x^k)s_{n-1}(x) \\
&= (1+x)s_{n-1}(x)-(1-x^k)s_{n-1}(x) \\
&= (1+x^{n-1})-(1-x^k)s_{n-1}(x)
\end{align*}
hold.
Thus, as the matrix on $\Z[x]/(1-x^{2k})$,
the above matrix can be transformed into
\begin{align*}
\begin{bmatrix}
 1+x^{n-1} &      &   &  (1-x^k)s_{n-1}(x) & 0 \\
           &  1   &   &                    &   \\
           &      & 1 &                    &         
\end{bmatrix}.
\end{align*}
The assertion is proved.

(VI: $\Lambda_{D_4,3k,3}$)
Considering the $(4,7)$ entry, $M_{D_4,3k,3}(x)$ can be transformed into
\begin{align*}
\begin{bmatrix}
1+x^3 &       & 1+x^2 & x^2           & -x^{k+1} &       &   \\
      & 1+x^3 & 1+x   & x+x^2         &          & 1-x^k &   \\
      &       & 1+x^k & x+x^k+x^{k+2} & 1        &       &   \\
      &       &       &               &          &       & 1    
\end{bmatrix}
\end{align*}
and considering the $(3,5)$ entry, we have
\begin{align*}
\begin{bmatrix}
1+x^3 &       & 1+x^2+x^{k+1}+x^{2k+1} & x^2+x^{k+2}+x^{2k+1}+x^{2k+3} &   &       &   \\
      & 1+x^3 & 1+x                    & x+x^2                         &   & 1-x^k &   \\
      &       &                        &                               & 1 &       &   \\
      &       &                        &                               &   &       & 1    
\end{bmatrix}.
\end{align*}
In $\Z[x]/(1-x^{3k})$, the equations
\begin{align*}
1+x^2+x^{k+1}+x^{2k+1} &= (1+x^{k+1})(1+x^{2k+1}) 
= (1+x^{k+1})(1+x)s_{2k+1}(x), \\
1-x^k &= -x^k(1-x^{2k}) = -x^k(1+x)s_{2k}(x)  
\end{align*}
hold.
Thus as a matrix on $\Z[x]/(1-x^{3k})$, 
considering the $(2,3)$ entry, the above matrix is transformed into
\begin{align*}
\begin{bmatrix}
1+x^3 & h_1(x) & 0   & h_2(x) &   & h_3(x) & \\
      &        & 1+x &        &   &        & \\
      &        &     &        & 1 &        & \\
      &        &     &        &   &        & 1    
\end{bmatrix},
\end{align*}
where
\begin{align*}
h_1(x) &= -(1+x^3)(1+x^{k+1})s_{2k+1}(x), \\
h_2(x) &= (x^2+x^{k+2}+x^{2k+1}+x^{2k+3})-(x+x^2)(1+x^{k+1})s_{2k+1}(x) \\ 
&= (x^2+x^{k+2}+x^{2k+1}+x^{2k+3})-x(1+x^{k+1})(1+x^{2k+1}) \\
&= -x(1-x+x^2)(1-x^{2k}) = -x(1+x^3)s_{2k}(x), \\
h_3(x) &= -(1-x^k)(1+x^{k+1})s_{2k+1}(x) \\
&= x^k(1-x^{2k})(1+x^{k+1})s_{2k+1}(x) 
= x^k s_{2k}(x) (1+x^{k+1})(1+x^{2k+1}) = x^{k+2} s_{2k}(x) (1+x^{2k-1})(1+x^{k-1}).
\end{align*}
As elements of $\Z[x]/(1-x^{3k})$, $h_1(x)$ and $h_2(x)$ can be divided by $1+x^3$ and
$h_3(x)$ can be divided by $1+x^{2k-1}$, $1-x^{2k}$, and $1+x^{2k+1}$.
The polynomial $1+x^3$ can divide $1-x^{2k}$ if $k \in 3\Z$, 
can divide $1+x^{2k+1}$ if $k \in 1+3\Z$,
and can divide $1+x^{2k-1}$ if $k \in 2+3\Z$.
Thus $h_3(x)$ also can be divided by $1+x^3$ in $\Z[x]/(1-x^{3k})$.
The assertion is proved.

(VII: $\Lambda_{E_6,k,1}$)
Using the $(1,1)$ entry and the $(2,2)$ entry, 
$M_{E_6,k,1}(x)$ can be transformed into
\begin{align*}
\begin{bmatrix}
1_2 &        \\
    & M_1(x) 
\end{bmatrix},
\end{align*}
where $M_1(x)$ is  
\begin{align*}
\begin{bmatrix}
1+x^6 &          &          &       &
(1+x^2)( 1+x)                     & (1+x^2)(   x+x^2+x^3)                     \\
      & 1-x^{12} &          &       & 
(1+x^2)( 1+x-x^2    +x^4-x^5-x^6) & (1+x^2)(   x+x^2                -x^7-x^8) \\
      &          & 1-x^{12} &       &
(1+x^2)( 1  -x^2+x^3    -x^5)     & (1+x^2)(   x                    -x^7)     \\
      &          &          & 1+x^6 &
 1+x^2                            & (1+x^2)( 1+x^3)                           
\end{bmatrix}.
\end{align*}
Thus, we have $\Coker M_{E_6,k,1}(X_k)\cong\Coker M_1(X_k)$.
Next, considering $(4,5)$ entry of $M_1(x)$, we have
\begin{align*}
\begin{bmatrix}
1+x^6 &          &          &
(1+x^6)(-1-x)                     &       &-(1+x^6)                        \\
      & 1-x^{12} &          &
(1+x^6)(-1-x+x^2    -x^4+x^5+x^6) &       & (1+x^6)(-1  +x^2- x^3    +x^5) \\
      &          & 1-x^{12} &
(1+x^6)(-1  +x^2-x^3    +x^5)     &       & (1+x^6)(-1+x    - x^3+x^4)     \\
      &          &          &
0                                 & 1+x^2 & 0                              
\end{bmatrix}
\end{align*}
and 
\begin{align*}
\begin{bmatrix}
1+x^6 &          &          &
                                  &       &                                \\
      & 1-x^{12} &          &
(1+x^6)(  -x+x^2    -x^4+x^5)     &       & (1+x^6)(-1  +x^2- x^3    +x^5) \\
      &          & 1-x^{12} &
(1+x^6)(-1  +x^2-x^3    +x^5)     &       & (1+x^6)(-1+x    - x^3+x^4)     \\
      &          &          &
0                                 & 1+x^2 & 0                              
\end{bmatrix}.
\end{align*}
Now, $\Coker M_{E_6,k,1}(X_k) \cong 
\Coker (1+X_k^2) \oplus \Coker (1+X_k^6) \oplus \Coker M_2(x)$,
where
\begin{align*}
M_2(x) &= \begin{bmatrix}
1-x^{12} &          & (1+x^6)(  -x+x^2    -x^4+x^5) & (1+x^6)(-1  +x^2-x^3    +x^5) \\
         & 1-x^{12} & (1+x^6)(-1  +x^2-x^3    +x^5) & (1+x^6)(-1+x    -x^3+x^4)
\end{bmatrix} \\
&= (1-x)(1+x^3+x^6+x^9) \begin{bmatrix}
1+x+x^2 &         &   -x & -1-x \\
        & 1+x+x^2 & -1-x & -1\\ 
\end{bmatrix}.
\end{align*}
Considering its $(2,4)$ entry, it can be transformed into
\begin{align*}
(1-x)(1+x^3+x^6+x^9) \begin{bmatrix}
1+x+x^2 & -1+x^3  & 1+x+x^2 &    \\
        & 1+x+x^2 &         & -1 \\ 
\end{bmatrix}
\end{align*}
and
\begin{align*}
(1-x)(1+x^3+x^6+x^9) \begin{bmatrix}
1+x+x^2 &         & 0 &    \\
        & 1+x+x^2 &   & -1 \\ 
\end{bmatrix}.
\end{align*}
Thus, 
$\Coker M_2(X_k) \cong  
\Coker ((1-X_k)(1+X_k^3+X_k^6+X_k^9)) \oplus \Coker (1-X_k^{12})$
and the assertion has been proved.

(VIII: $\Lambda_{E_6,2k,2}$)
The matrix $M_{E_6,2k,2}(x)$ can be transformed into
\begin{align*}
\begin{bmatrix}
M_{E_6,2k,1}(x) & 
\begin{bmatrix}
-(x^3+x^{k-3}) \cdot T_5(x) &          \\
                            & -x^6-x^k
\end{bmatrix}
\end{bmatrix}
\end{align*}
and taking into account that $T_5(x) \in \GL_n(\Z[x]/(1-x^{2k}))$, 
the above matrix can be transformed into 
\begin{align*}
\begin{bmatrix} M_{E_6,2k,1}(x) & (1+x^{k-6}) \cdot 1_n \end{bmatrix}.
\end{align*}
From the proof for $\Lambda_{E_6,k,1}$,
the matrix $M_{E_6,2k,1}(x)$ can be transformed into 
\begin{align*}
N(x)=\begin{bmatrix}
1 &   &       &          &                      &       \\
  & 1 &       &          &                      &       \\
  &   & 1+x^6 &          &                      &       \\
  &   &       & 1-x^{12} &                      &       \\
  &   &       &          & (1-x)(1+x^3+x^6+x^9) &       \\
  &   &       &          &                      & 1+x^2 \\
\end{bmatrix},
\end{align*}
and thus $M_{E_6,2k,2}(x)$ can be transformed into 
$\begin{bmatrix} N(x) & (1+x^{k-6}) \cdot 1_n\end{bmatrix}$.
We have $K_0(\umod \Lambda_{E_6,2k,2})$ is isomorphic to
\begin{align*}
&(\Coker \begin{bmatrix} 1_{2k} & 1_{2k}+X_{2k}^{k-6} \end{bmatrix})^2 
\oplus \Coker \begin{bmatrix} 
1_{2k}+X_{2k}^6 & 1+X_{2k}^{k-6} \end{bmatrix} 
\oplus \Coker \begin{bmatrix} 
1_{2k}-X_{2k}^{12} & 1+X_{2k}^{k-6} \end{bmatrix} \\
&\quad \oplus \Coker \begin{bmatrix} (1_{2k}-X_{2k})(1_{2k}+X_{2k}^3+X_{2k}^6+X_{2k}^9) & 
1_{2k}+X_{2k}^{k-6} \end{bmatrix}
\oplus \Coker \begin{bmatrix} 1_{2k}+X_{2k}^2 & 1_{2k}+X_{2k}^{k-6} \end{bmatrix}.
\end{align*}
The first component is clearly 0,
and the third one is isomorphic to $\Coker (1+X_{2k}^{k-6})$,
because we have $1-x^{12}=(1+x^{k-6})(1-x^{k-6})$ in $\Z[x]/(1-x^{2k})$.
Thus we have the assertion.

(IX: $\Lambda_{E_7,k,1}$)
Considering the $(6,8)$ entry and the fact (*), 
$M_{E_7,k,1}(x)$ can be transformed into
\begin{align*}
\begin{bmatrix}
(1+x^9) \cdot 1_7 &
\begin{bmatrix}
1-x+x^2     &   -x    -x^3+x^4-x^5+x^6-x^7+x^8-x^9       -x^{11} \\
0           &   -x                                -x^{10}        \\
1-x+x^2     &   -x+x^2-x^3+x^4-x^5+x^6-x^7+x^8-x^9               \\
0           &  0                                                 \\
0           &  0                                                 \\
1      +x^3 &  0                                                 \\
1-x+x^2     &  1-x+x^2-x^3+x^4-x^5+x^6-x^7+x^8                   
\end{bmatrix}
\end{bmatrix}
\end{align*}
and using the $(7,9)$ entry,
\begin{align*}
\begin{bmatrix}
(1+x^9) \cdot 1_7 & 
\begin{bmatrix}
0           & -1  -x^2                        -x^9       -x^{11} \\
0           &   -x                                -x^{10}        \\
0           & -1                              -x^9               \\
0           &  0                                                 \\
0           &  0                                                 \\
0           & -1                              -x^9               \\
1-x+x^2     &  0                   
\end{bmatrix}
\end{bmatrix}.
\end{align*}
Because the entries in the rightest column are divided by $1+x^9$
and $1+x^9$ are divided by $1-x+x^2$,
the assertion is proved.

(X: $\Lambda_{E_8,k,1}$)
The matrix $M_{E_8,k,1}(x)$ is transformed into
\begin{align*}
\begin{bmatrix}
(1+x^{15}) \cdot 1_8 &
(1+x^5) \cdot \begin{bmatrix}
0           &   -x    -x^3        +x^6     +x^8           -x^{11}       -x^{13} \\
1-x+x^2     &     -x^2+x^3    +x^5    +2x^7-x^8+x^9-x^{10}       -x^{12}        \\
0           &   -x                +x^6                    -x^{11}               \\
1-x+x^2     &          x^3    +x^5    + x^7-x^8+x^9-x^{10}                      \\
0           &  0                                                                \\
0           &  0                                                                \\
1      +x^3 &    x    +x^3+x^4+x^5    + x^7                                     \\
1-x+x^2     &  1      +x^3            + x^7-x^8+x^9
\end{bmatrix}
\end{bmatrix}.
\end{align*}
Considering the $(8,9)$ entry, this matrix can be transformed into
\begin{align*}
\begin{bmatrix}
(1+x^{15}) \cdot 1_8 &
(1+x^5) \cdot \begin{bmatrix}
0           &   -x    -x^3        +x^6     +x^8           -x^{11}       -x^{13} \\
0           & -1  -x^2        +x^5    + x^7        -x^{10}       -x^{12}        \\
0           &   -x                +x^6                    -x^{11}               \\
0           & -1              +x^5                 -x^{10}                      \\
0           &  0                                                                \\
0           &  0                                                                \\
0           & -1              +x^5                 -x^{10}                      \\
1-x+x^2     &  0                                                                \\ 
\end{bmatrix}
\end{bmatrix}.
\end{align*}
Because the entries in the rightest column are divided by $(1+x^5)(1-x^5+x^{10})=1+x^{15}$
and $1+x^{15}$ are divided by $(1+x^5)(1-x+x^2)$,
the assertion is proved.
\end{proof}

\subsection{Proof of Theorem \ref{main}}

Now, the remained task is to calculate the summands 
appearing in Proposition \ref{tanninsi}.
The processes of the calculations are written in the next subsection. 
Using the results in Subsection \ref{calc_coker},
we can prove Theorem \ref{main}.

\begin{proof}[Proof of Theorem \ref{main}]
We state each cokernel in Proposition \ref{tanninsi}.
One can easily check that Theorem \ref{main} holds.

(I: $\Lambda_{A_n,k,1}$)
By Lemmas \ref{basic_coker} (2) and \ref{coker_I}, we have
\begin{align*}
& \Coker (1_k-X_k^{n+1}) \cong \Z^d, \quad \Coker (1_k-X_k) \cong \Z, \\
& \Coker ((1_k-X_k)(1_k+X_k^{(n+1)/2}))
\cong \begin{cases}
\Z \oplus (\Z/2\Z)^{d-1} & (r \in 2\Z) \\
\Z^{(d+2)/2}       & (r \notin 2\Z) \\
\end{cases} \quad (n \notin 2\Z).
\end{align*}

(II: $\Lambda_{A_n,2k,2}$)
By Lemmas \ref{basic_coker} (2) and \ref{coker_II}, we have
\begin{align*}
&\Coker (1_{2k}+X_{2k}^{k-(n+1)/2}) 
\cong \begin{cases}
\Z^d           & (r \in 4\Z) \\
(\Z/2\Z)^{2d}  & (r \in 2+4\Z) \\
\Z^{d/2}       & (r \notin 2\Z) 
\end{cases}, \\
&\Coker \begin{bmatrix} (1_{2k}-X_{2k})(1_{2k}+X_{2k}^{(n+1)/2}) & 
1_{2k}+X_{2k}^{k-(n+1)/2} \end{bmatrix} 
\cong \begin{cases}
(\Z/2\Z)^{d-1} \oplus (\Z/4\Z) & (r \in 4\Z) \\
(\Z/2\Z)^{d+1}                 & (r \in 2+4\Z) \\
\Z^{d/2}                       & (r \notin 2\Z) 
\end{cases}.
\end{align*}

(III: $\Lambda_{A_n,2k-1,2}$)
By Lemma \ref{basic_coker} (2), we have
\begin{align*}
\Coker (1_{2k-1}+X_{2k-1}^{k-(n+2)/2}) 
\cong (\Z/2\Z)^{2d}, \quad
\Coker \begin{bmatrix} 1_{2k-1}-X_{2k-1} & 
2 \cdot 1_{2k-1} \end{bmatrix} 
\cong \Z/2\Z.
\end{align*}

(IV: $\Lambda_{D_n,k,1}$)
By Lemmas \ref{basic_coker} (2) and \ref{coker_s}, we have
\begin{align*}
&\Coker (1_k+X_k^{n-1}) \cong \begin{cases}
(\Z/2\Z)^d     & (k \in 2\Z, \ r \in 2\Z) \\
\Z^{d/2}       & (k \in 2\Z, \ r \notin 2\Z) \\
(\Z/2\Z)^d     & (k \notin 2\Z) 
\end{cases}, \\
&\Coker s_{2n-2}(X_k) \cong \begin{cases}
\Z^{d-1}   \oplus (\Z/r\Z) & (k \in    2\Z) \\
\Z^d                       & (k \notin 2\Z)
\end{cases} \quad (n \notin 2\Z), \\
&\Coker s_{n-1}(X_k) \cong \begin{cases}
\Z^{(d-2)/2} \oplus (\Z/r\Z) & (k \in    2\Z) \\
(\Z/2\Z)^{d-1}             & (k \notin 2\Z)
\end{cases} \quad (n \in 2\Z).
\end{align*}

(V: $\Lambda_{D_n,2k,2}$)
By Lemmas \ref{coker_V_1}, \ref{coker_V_2} and \ref{coker_V_3}, we have
\begin{align*}
&\Coker \begin{bmatrix} 1_{2k}+X_{2k}^{n-1} & 1_{2k}-X_{2k}^k \end{bmatrix}
\cong \begin{cases}
(\Z/2\Z)^d     & (k \in 2\Z, \ r \in 2\Z) \\
\Z^{d/2}       & (k \in 2\Z, \ r \notin 2\Z) \\
(\Z/2\Z)^d     & (k \notin 2\Z) 
\end{cases}, \\
&\Coker \begin{bmatrix} s_{2n-2}(X_{2k}) & 1_{2k}+X_{2k}^{k-(n-1)} \end{bmatrix}
\cong \begin{cases}
\Z^d                      & (k \in    2\Z, \ r \in    4\Z) \\
(\Z/2\Z)^{2d-1}           & (k \in    2\Z, \ r \in    2+4\Z) \\
\Z^{d/2}                  & (k \in    2\Z, \ r \notin 2\Z) \\
\Z^{d-1} \oplus (\Z/r\Z)  & (k \notin 2\Z)
\end{cases} \quad (n \notin 2\Z), \\
&\Coker \begin{bmatrix} 1_{2k}+X_{2k}^{n-1} & (1-X_{2k}^k)s_{n-1}(X_{2k}) \end{bmatrix}
\cong \begin{cases}
\Z^{d/2}                  & (k \in    2\Z) \\
\Z^{d-1} \oplus (\Z/r\Z)  & (k \notin 2\Z)
\end{cases} \quad (n \in 2\Z).
\end{align*}

(VI: $\Lambda_{D_4,3k,3}$)
By Lemma \ref{basic_coker} (2), we have
\begin{align*}
\Coker (1_{3k}+X_{3k}^{3}) \cong \begin{cases}
\Z^3       & (k \in 2\Z) \\
(\Z/2\Z)^3 & (k \notin 2\Z) 
\end{cases}, \quad
\Coker (1_{3k}+X_{3k}) \cong \begin{cases}
\Z         & (k \in 2\Z) \\
\Z/2\Z     & (k \notin 2\Z) 
\end{cases}.
\end{align*}

(VII: $\Lambda_{E_6,k,1}$)
By Lemmas \ref{basic_coker} (2) and \ref{coker_VII}, we have
\begin{align*}
& \Coker (1_k-X_k^{12}) \cong \Z^d, \quad
\Coker ((1_k-X_k)(1_k+X_k^3+X_k^6+X_k^9)) \cong \begin{cases}
\Z \oplus (\Z/4\Z)^{d-1}     & (d=1,3) \\
\Z^{(d+2)/2} \oplus (\Z/2\Z)^{(d-2)/2} & (d=2,6) \\
\Z^{(3d+4)/4}                & (d=4,12) \\
\end{cases}, \\
& \Coker (1_k+X_k^6) \cong \begin{cases}
(\Z/2\Z)^d & (d=1,3,2,6) \\
\Z^{d/2}   & (d=4,12)    \\
\end{cases}, \quad
\Coker (1_k+X_k^2) \cong \begin{cases}
 \Z/2\Z    & (d=1,3)  \\
(\Z/2\Z)^2 & (d=2,6)  \\
\Z^2       & (d=4,12) \\
\end{cases}.
\end{align*}

(VIII: $\Lambda_{E_6,2k,2}$)
By Lemmas \ref{basic_coker} (2), \ref{coker_VIII_1}, and \ref{coker_VIII_2}, we have
\begin{align*}
&\Coker (1_{2k}+X_{2k}^{k-6})
\cong \begin{cases}
\Z^d          & (d=1,3)  \\
(\Z/2\Z)^{2d} & (d=2,6)  \\
\Z^{d/2}      & (d=4,12) \\
\end{cases}, \\
&\Coker \begin{bmatrix} (1_{2k}-X_{2k})(1_{2k}+X_{2k}^3+X_{2k}^6+X_{2k}^9) 
& 1_{2k}+X_{2k}^{k-6} \end{bmatrix} 
\cong \begin{cases}
\Z^d                & (d=1,3)  \\
(\Z/2\Z)^{(3d+2)/2} & (d=2,6)  \\
\Z^{d/2}            & (d=4,12) \\
\end{cases}, \\
&\Coker \begin{bmatrix} 1_{2k}+X_{2k}^6 & 1_{2k}+X_{2k}^{k-6} \end{bmatrix} 
\cong \begin{cases}
(\Z/2\Z)^d & (d=1,3,2,6)  \\
\Z^{d/2}   & (d=4,12) \\
\end{cases}, \\
&\Coker \begin{bmatrix} 1_{2k}+X_{2k}^2 & 1_{2k}+X_{2k}^{k-6} \end{bmatrix} 
\cong \begin{cases}
(\Z/2\Z)   & (d=1,3)  \\
(\Z/2\Z)^2 & (d=2,6)  \\
\Z^2       & (d=4,12) \\
\end{cases}.
\end{align*}

(IX: $\Lambda_{E_7,k,1}$)
By Lemmas \ref{basic_coker} (2) and \ref{coker_s}, we have
\begin{align*}
\Coker (1_k+X_k^9) \cong \begin{cases}
(\Z/2\Z)^d & (d=1,3,9)     \\
\Z^{d/2}   & (d=2,6,18)
\end{cases}, \quad
\Coker (1_k-X_k+X_k^2) \cong \begin{cases}
0          & (d=1)    \\
(\Z/2\Z)^2 & (d=3,9)  \\
\Z/3\Z     & (d=2)    \\
\Z^2       & (d=6,18)
\end{cases}.
\end{align*}

(X: $\Lambda_{E_8,k,1}$)
By Lemmas \ref{basic_coker} (2) and \ref{coker_X}, we have
\begin{align*}
&\Coker (1_k+X_k^{15}) \cong \begin{cases}
(\Z/2\Z)^d    & (d=1,3,5,15)  \\
\Z^{d/2}      & (d=2,6,10,30) 
\end{cases}, \\
&\Coker ((1_k-X_k+X_k^2)(1_k+X_k^5)) \cong \begin{cases}
(\Z/2\Z)^d    & (d=1,3,5)  \\
(\Z/2\Z)^7    & (d=15)     \\
\Z^{d/2}      & (d=2,6,10) \\
\Z^{7}        & (d=30)
\end{cases}. 
\end{align*}

The proof is completed.
\end{proof}

\subsection{Calculation of summands in Proposition \ref{tanninsi}}\label{calc_coker}

We calculate each cokernel in Proposition \ref{tanninsi}.
First, we state general properties of the cokernels of matrices.

\begin{Def}\label{perm_matrix_orbit}
Let $m \ge 1$, $p \in \Z$, $d=\gcd(p,m)$ and
$\sigma \in \mathfrak{S}_m$ be the unique permutation such that
$X_m^p$ is the permutation matrix of $\sigma$.
We can deduce that $\sigma$ can be decomposed into $d$ cyclic permutations as 
\begin{align*}
\sigma=(1,\sigma(1),\sigma^2(1),\ldots,\sigma^{q-1}(1)) \cdots 
(d,\sigma(d),\sigma^2(d),\ldots,\sigma^{q-1}(d)),
\end{align*}
where $q = m/d$.
So we can define a permutation $\eta \in \mathfrak{S}_m$ by
\begin{align*}
\eta(uq+v)=\sigma^{v-1}(u+1)
\quad (u \in \{0,\ldots,d-1\}, \ v \in \{1,\ldots,q\}).
\end{align*}
Now we define $Y_{m,p}$ as the permutation matrix of $\eta$.
\end{Def}

\begin{Lem}\label{basic_coker}
Let $m,l \ge 1$, $p \in \Z$ and $d = \gcd(p,m)$, $q=m/d$ and $f(x),g(x) \in \Z[x]$.
\begin{itemize}
\item[(1)]
We have
$\Coker f(X_m^p) \cong (\Coker f(X_q))^d$ and
$\Coker \begin{bmatrix} f(X_m^p) & g(X_m^p) \end{bmatrix} \cong 
(\Coker \begin{bmatrix} f(X_q) & g(X_q) \end{bmatrix})^d$.
\item[(2)]
We have
\begin{align*}
&\Coker (1_m-X_m^p) \cong \Z^d, \quad
\Coker (l \cdot (1_m-X_m^p)) \cong \Z^d \oplus (\Z/l\Z)^{m-d}, \\
&\Coker \begin{bmatrix} 1_m-X_m^p & l \cdot 1_{m} \end{bmatrix} \cong (\Z/l\Z)^d, \quad
\Coker (1_m+X_m^p) \cong \begin{cases}
(\Z/2\Z)^d & (q \notin 2\Z) \\
\Z^d       & (q \in 2\Z)    
\end{cases}.
\end{align*}
\item[(3)]
If $m \in 2\Z$, then $\Coker f(X_m) \cong \Coker f(-X_m)$ and
$\Coker \begin{bmatrix} f(X_m) & g(X_m) \end{bmatrix}
\cong \Coker \begin{bmatrix} f(-X_m) & g(-X_m) \end{bmatrix}$.
\item[(4)]
We have $\Im (1 \pm X_m^p) = \Im (1 \pm X_m^d)$ and
\begin{align*}
\Coker \begin{bmatrix} f(X_m) & 1-X_m^p \end{bmatrix} \cong \Coker f(X_d), \quad
\Coker \begin{bmatrix} f(X_m) & g(X_m) & 1-X_m^p \end{bmatrix} 
\cong \Coker \begin{bmatrix} f(X_d) & g(X_d) \end{bmatrix}.
\end{align*} 
\item[(5)]
We have
$\Coker (l \cdot s_m(-X_m)) \cong \Z^{m-1} \oplus (\Z/l\Z)$ and 
$\Coker \begin{bmatrix} s_m(-X_m) & l \cdot 1_m \end{bmatrix} \cong (\Z/l\Z)^{m-1}$.
\end{itemize}
\end{Lem}

\begin{proof}
(1)
By definition, we have 
$Y_{m,p}^{-1}f(X_m^p)Y_{m,p} = f(X_q)^{\oplus d}$.
The proof for the latter assertion is similar.

(2)
If $p=1$, they are obvious.
The remained case can be reduced to direct sums of the case $p=1$ by (1).

(3)
Because $m \in 2\Z$, we can consider the diagonal matrix
$J_m=\operatorname{diag}(1,-1,1,-1,\ldots,1,-1) \in \GL_m(\Z)$.
We have $J_m X_m J_m = -X_m$.

(4), (5)
Straightforward.
\end{proof}

Note that the following lemma can be used only if $d \ge 2$.

\begin{Lem}\label{d>=2}
Let $m \ge 1$, $p \in \Z$ and $d = \gcd(p,m)$ and $f(x), g(x) \in \Z[x]$.
If $d \ge 2$, we have
\begin{align*}
\Coker ((1_m-X_m)f(X_m^p)) 
&\cong (\Coker f(X_q))^{d-1} \oplus \Coker ((1_q-X_q)f(X_q)), \\
\Coker \begin{bmatrix} (1_m-X_m)f(X_m^p) & g(X_m^p) \end{bmatrix}
&\cong (\Coker \begin{bmatrix} f(X_q) & g(X_q)\end{bmatrix} )^{d-1}
\oplus \Coker \begin{bmatrix}(1_q-X_q)f(X_q) & g(X_q) \end{bmatrix}.
\end{align*}
\end{Lem}

\begin{proof}
Because $d \ge 2$, we have
\begin{align*}
Y_{m,p}^{-1} (1_m-X_m)f(X_m^p) Y_{m,p} 
=\begin{bmatrix}
 f(X_q) &         &         &         & -X_q f(X_q) \\
-f(X_q) &  f(X_q) &                                 \\
        & \cdots  & \cdots  &                       \\
        &         & -f(X_q) &  f(X_q) &             \\
        &         &         & -f(X_q) &  f(X_q) 
\end{bmatrix}
\end{align*}
and this can be transformed into
\begin{align*}
\begin{bmatrix}
f(X_q)^{\oplus (d-1)} & \\
& (1-X_q)f(X_q)
\end{bmatrix}.
\end{align*}
preserving the blocks. We have the first assertion.

Now, we have shown that
there exists $P_d(x),Q_d(x) \in \GL_d(\Z[x])$ such that 
\begin{align*}
P_d(X_q) Y_{m,p}^{-1} (1_m-X_m)f(X_m^p) Y_{m,p} Q_d(X_q) 
= \begin{bmatrix}
f(X_q)^{\oplus (d-1)} & \\
& (1-X_q)f(X_q)
\end{bmatrix}. 
\end{align*}
Take such $P_d(x),Q_d(x)$.
Then $Y_{m,p}^{-1}g(X_m^p)Y_{m,p}=g(X_q)^{\oplus d}$ and 
$P_d(x) (g(x) \cdot 1_{d}) P_d(x)^{-1}=g(x) \cdot 1_{d}$ imply
\begin{align*}
P_d(X_q)Y_{m,p}^{-1}g(X_m^p)Y_{m,p}P_d(X_q)^{-1}=g(X_q)^{\oplus d}.
\end{align*}
Thus, the matrix
\begin{align*}
P_d(X_q) Y_{m,p}^{-1} 
\begin{bmatrix} (1_m-X_m)f(X_m^p) & g(X_m^p) \end{bmatrix}
\begin{bmatrix} Y_{m,p} & 0 \\ 0 & Y_{m,p} \end{bmatrix}
\begin{bmatrix} Q_d(X_q) & 0 \\ 0 & P_d(X_q)^{-1} \end{bmatrix}
\end{align*}
is equal to
\begin{align*}
\begin{bmatrix}
f(X_q)^{\oplus (d-1)} & 0 & g(X_q)^{\oplus (d-1)} & 0 \\
0 & (1-X_q)f(X_q) & 0 & g(X_q)
\end{bmatrix}
\end{align*}
and it verifies the second assertion.
\end{proof}

\begin{Lem}\label{coker_s}
Let $m \ge 1$, $p \in \Z$ and put $d=\gcd(p,k)$, $r=p/d$.
Then we have
\begin{align*}
\Coker s_p(X_m) &\cong \begin{cases}
\Z^{d-1}   \oplus (\Z/r\Z) & (m \in    2\Z) \\
\Z^{d}                     & (m \notin 2\Z, \ p \in    2\Z) \\
(\Z/2\Z)^{d-1}             & (m \notin 2\Z, \ p \notin 2\Z) \\
\end{cases}.
\end{align*}
\end{Lem}

\begin{proof}
If $m \in 2\Z$, then by Lemma \ref{basic_coker} (3),
we have $\Coker s_p(X_m) \cong \Coker s_p(-X_m)$.
Using the fact $1-x^p$ can be divided by $s_p(-x)$ and Lemma \ref{basic_coker} (4), 
we have
\begin{align*}
\Coker s_p(-X_m) \cong \Coker \begin{bmatrix} s_p(-X_m) & 1-X_m^p \end{bmatrix}
\cong \Coker s_p(-X_d) 
= \Coker (r \cdot s_d(-X_d)).
\end{align*}
By Lemma \ref{basic_coker} (5), 
$\Coker (r \cdot s_d(-X_d)) \cong \Z^{d-1} \oplus (\Z/r\Z)$.
The proof for $m \in 2\Z$ is completed.

If $m \notin 2\Z$ and $p \in 2\Z$, then $1-x^p$ can be divided by $s_p(x)$.
Using this fact and Lemma \ref{basic_coker} (4), we have
\begin{align*}
\Coker s_p(X_m) \cong \Coker \begin{bmatrix} s_p(X_m) & 1-X_m^p \end{bmatrix}
\cong \Coker s_p(X_d).
\end{align*}
Because $d=\gcd(m,p) \notin 2\Z$ and $p \in 2\Z$, 
we have $\Coker s_p(X_d)=\Coker 0_d=\Z^d$.
The proof for the case $m \notin 2\Z$ and $p \in 2\Z$ is completed.

If $m \notin 2\Z$ and $p \notin 2\Z$,
then $s_p(x)$ divides $1+x^p$, and $1+x^p$ divides $1+x^{pm}$.
Thus we have
\begin{align*}
\Coker s_p(X_m) 
\cong \Coker \begin{bmatrix} s_p(X_m)  & 1+X_m^{pm}  \end{bmatrix} 
\cong \Coker \begin{bmatrix} s_p(X_m)  & 2 \cdot 1_m \end{bmatrix}.
\end{align*}
The polynomial 
$s_p(x)-s_p(-x)$ can be divided by 2 and $1-x^p$ can be divided by $s_p(-x)$.
Therefore, we have 
\begin{align*}
\Coker \begin{bmatrix} s_p(X_m)  & 2 \cdot 1_m \end{bmatrix}
\cong \Coker \begin{bmatrix} s_p(-X_m) & 2 \cdot 1_m \end{bmatrix} 
\cong \Coker \begin{bmatrix} s_p(-X_m) & 1-X_m^p & 2 \cdot 1_m \end{bmatrix} 
\end{align*}
By Lemma \ref{basic_coker} (4) and the facts 
$s_p(-X_d)=r \cdot s_d(-X_d)$ and $r=p/d \notin 2\Z$,
\begin{align*}
\Coker \begin{bmatrix} s_p(-X_m) & 1-X_m^p & 2 \cdot 1_m \end{bmatrix}
&\cong \Coker \begin{bmatrix} s_p(-X_d) & 2 \cdot 1_d \end{bmatrix}\\
&\cong \Coker \begin{bmatrix} r \cdot s_d(-X_d) & 2 \cdot 1_d \end{bmatrix}
\cong \Coker \begin{bmatrix} s_d(-X_d) & 2 \cdot 1_d \end{bmatrix}.
\end{align*}
By Lemma \ref{basic_coker} (5), it is isomorphic to $(\Z/2\Z)^{d-1}$.
The proof for the case $k \notin 2\Z$ and $p \notin 2\Z$ is completed.
\end{proof}

We calculate the remained cokernels appearing in Proposition \ref{tanninsi}
using the previous lemmas.

\begin{Lem}[type I]\label{coker_I}
Let $n,k \ge 1$ be integers and $n \notin 2\Z$.
Put $d = \gcd(n+1,k)$, $r=(n+1)/d$. Then we have
\begin{align*}
\Coker ((1_k-X_k)(1_k+X_k^{(n+1)/2})) \cong \begin{cases}
\Z \oplus (\Z/2\Z)^{d-1}   & (r \in    2\Z) \\
\Z^{(d+2)/2}               & (r \notin 2\Z) 
\end{cases}.
\end{align*}
\end{Lem}

\begin{proof}
Put $q=k/d$.
We can deduce
\begin{align*}
\gcd((n+1)/2,k)=\begin{cases}
d   & (r \in 2\Z) \\
d/2 & (r \notin 2\Z).
\end{cases}
\end{align*}

Assume $r \in 2\Z$ first, then we have $q \notin 2\Z$. 
Therefore, Lemma \ref{basic_coker} (4) yields
\begin{align*}
\Coker ((1_k-X_k)(1_k+X_k^{(n+1)/2})) \cong \Coker ((1_k-X_k)(1_k+X_k^d)).
\end{align*}
If $d=1$, then we have $k \notin 2\Z$ and $\gcd(2,k)=1$. 
Thus the cokernel is $\Coker (1_k-X_k^2) \cong \Z$.
If $d \ge 2$, Lemma \ref{d>=2} and $q \notin 2\Z$ yield
\begin{align*}
\Coker ((1_k-X_k)(1_k+X_k^d)) 
\cong \Coker ((1_q+X_q))^{d-1} \oplus \Coker (1_q-X_q^2) 
\cong (\Z/2\Z)^{d-1} \oplus \Z.
\end{align*}
The assertion is proved for the case $r \in 2\Z$.

Assume $r \notin 2\Z$ next. 
Lemma \ref{basic_coker} (4) yields
\begin{align*}
\Coker ((1_k-X_k)(1_k+X_k^{(n+1)/2})) \cong \Coker ((1_k-X_k)(1_k+X_k^{d/2})).
\end{align*}
If $d/2=1$, then we have $k \in 2\Z$ and $\gcd(2,k)=2$. 
Thus the cokernel is $\Coker (1_k-X_k^2) \cong \Z^2$.
If $d/2 \ge 2$, Lemma \ref{d>=2} and $k/(d/2)=2q \in 2\Z$ yield
\begin{align*}
\Coker ((1_k-X_k)(1_k+X_k^{d/2})) 
\cong (\Coker (1_{2q}+X_{2q}))^{(d-2)/2} \oplus \Coker (1_{2q}-X_{2q}^2) 
\cong \Z^{(d-2)/2} \oplus \Z^2=\Z^{(d+2)/2}.
\end{align*}
The assertion is proved for the case $r \notin 2\Z$.
\end{proof}

\begin{Lem}[type II]\label{coker_II}
Let $n,k \ge 1$ be integers, and $n \notin 2\Z$.
Put $d = \gcd(n+1,k)$, $r=(n+1)/d$.
Then we have
\begin{align*}
\Coker \begin{bmatrix} (1_{2k}-X_{2k})(1_{2k}+X_{2k}^{(n+1)/2}) & 
1_{2k}+X_{2k}^{k-(n+1)/2} \end{bmatrix} 
\cong \begin{cases}
(\Z/2\Z)^{d-1} \oplus (\Z/4\Z) & (r \in 4\Z) \\
(\Z/2\Z)^{d+1}                 & (r \in 2+4\Z) \\
\Z^{d/2}                       & (r \notin 2\Z) 
\end{cases}.
\end{align*}
\end{Lem}

\begin{proof}
Put $q=k/d$.
We can deduce that 
\begin{align*}
\gcd((n+1)/2,2k)=\begin{cases}
2d   & (r \in 4\Z) \\
 d   & (r \in 2+4\Z) \\
 d/2 & (r \notin 2\Z) 
\end{cases}, 
\quad
\gcd(k-(n+1)/2,2k)=\begin{cases}
 d   & (r \in 4\Z) \\
2d   & (r \in 2+4\Z) \\
 d/2 & (r \notin 2\Z) 
\end{cases}.
\end{align*}

Consider the case $r \in 4\Z$ first.
Then Lemma \ref{basic_coker} (4) yields
\begin{align*}
\Coker \begin{bmatrix} 
(1_{2k}-X_{2k})(1_{2k}+X_{2k}^{(n+1)/2}) & 1_{2k}+X_{2k}^{k-(n+1)/2} 
\end{bmatrix} 
&\cong \Coker \begin{bmatrix} 
(1_{2k}-X_{2k})(1_{2k}+X_{2k}^{2d}) & 1_{2k}+X_{2k}^d 
\end{bmatrix} \\
&\cong \Coker \begin{bmatrix} 
2 \cdot (1_{2k}-X_{2k}) & 1_{2k}+X_{2k}^d 
\end{bmatrix}. 
\end{align*}
Assume $d=1$, then we have
\begin{align*}
\Coker \begin{bmatrix} 
2 \cdot (1_{2k}-X_{2k}) & 1_{2k}+X_{2k}^d
\end{bmatrix}
= \Coker \begin{bmatrix} 
2 \cdot (1_{2k}-X_{2k}) & 1_{2k}+X_{2k}
\end{bmatrix}
\cong \Coker \begin{bmatrix} 
4 \cdot 1_{2k} & 1_{2k}+X_{2k} 
\end{bmatrix} .
\end{align*}
Using Lemma \ref{basic_coker} (3) and then (2), we can deduce that
\begin{align*}
\Coker \begin{bmatrix} 
4 \cdot 1_{2k} & 1_{2k}+X_{2k} 
\end{bmatrix} 
 \cong \Coker \begin{bmatrix} 
4 \cdot 1_{2k} & 1_{2k}-X_{2k} 
\end{bmatrix} \cong \Z/4\Z.
\end{align*}
If $d \ge 2$, from Lemma \ref{d>=2}, 
$\Coker \begin{bmatrix} 
2 \cdot (1_{2k}-X_{2k}) & 1_{2k}+X_{2k}^d 
\end{bmatrix}$ is isomorphic to
\begin{align*}
(\Coker \begin{bmatrix} 2 \cdot 1_{2q} & 1_{2q}+X_{2q} \end{bmatrix})^{d-1}
\oplus (\Coker \begin{bmatrix} 2\cdot(1_{2q}-X_{2q}) & 1_{2q}+X_{2q} \end{bmatrix}).
\end{align*}
The first summand is calculated as 
\begin{align*}
\Coker \begin{bmatrix} 2 \cdot 1_{2q} & 1_{2q}+X_{2q} \end{bmatrix}
& \cong \Coker \begin{bmatrix} 2 \cdot 1_{2q} & 1_{2q}-X_{2q} \end{bmatrix}
\cong \Z/2\Z.
\end{align*}
Similarly to the case $d=1$,
the second summand can be calculated as
\begin{align*}
\Coker \begin{bmatrix} 2\cdot(1_{2q}-X_{2q}) & 1_{2q}+X_{2q} \end{bmatrix}
\cong \Z/4\Z.
\end{align*}
Now the proof for the case $r \in 4\Z$ is completed.

Second, we assume $r \in 2+4\Z$.
Then Lemma \ref{basic_coker} (4) yields
\begin{align*}
\Coker \begin{bmatrix} 
(1_{2k}-X_{2k})(1_{2k}+X_{2k}^{(n+1)/2}) & 1_{2k}+X_{2k}^{k-(n+1)/2} 
\end{bmatrix} 
&\cong \Coker \begin{bmatrix} 
(1_{2k}-X_{2k})(1_{2k}+X_{2k}^d) & 1_{2k}+X_{2k}^{2d} 
\end{bmatrix} \\
&\cong \Coker \begin{bmatrix} 
(1_{2k}-X_{2k})(1_{2k}+X_{2k}^d) & 2 \cdot 1_{2k} 
\end{bmatrix}, 
\end{align*}
where the last equality comes from $1-x^{2d}$ can be divided by $(1-x)(1+x^d)$.
Assume $d=1$. Then we have
\begin{align*}
\Coker \begin{bmatrix} 
(1_{2k}-X_{2k})(1_{2k}+X_{2k}^d) & 2 \cdot 1_{2k} 
\end{bmatrix}
&\cong \Coker \begin{bmatrix} 
1_{2k}-X_{2k}^2 & 2 \cdot 1_{2k} 
\end{bmatrix} \cong (\Z/2\Z)^2.
\end{align*}
If $d \ge 2$, Lemma \ref{d>=2} implies 
\begin{align*}
\Coker \begin{bmatrix} 
(1_{2k}-X_{2k})(1_{2k}+X_{2k}^d) & 2 \cdot 1_{2k} 
\end{bmatrix}
&\cong (\Coker \begin{bmatrix} 
1_{2q}+X_{2q} & 2 \cdot 1_{2q} 
\end{bmatrix})^{d-1} \oplus \Coker \begin{bmatrix} 
1_{2q}-X_{2q}^2 & 2 \cdot 1_{2q} 
\end{bmatrix}.
\end{align*}
Using Lemma \ref{basic_coker} (2), 
each summand can be calculated as below;
\begin{align*}
\Coker \begin{bmatrix} 1_{2q}+X_{2q} & 2 \cdot 1_{2q} \end{bmatrix} 
\cong \Coker \begin{bmatrix} 1_{2q}-X_{2q} & 2 \cdot 1_{2q} \end{bmatrix}
\cong \Z/2\Z, \quad
\Coker \begin{bmatrix} 1_{2q}-X_{2q}^2 & 2 \cdot 1_{2q} \end{bmatrix}
\cong (\Z/2\Z)^2.
\end{align*}
The proof for the case $r \in 2+4\Z$ is completed.

The remained case is $r \notin 2\Z$.
Then Lemma \ref{basic_coker} (4) yields
\begin{align*}
\Coker \begin{bmatrix} 
(1_{2k}-X_{2k})(1_{2k}+X_{2k}^{(n+1)/2}) & 1_{2k}+X_{2k}^{k-(n+1)/2} 
\end{bmatrix} 
&\cong \Coker \begin{bmatrix} 
(1_{2k}-X_{2k})(1_{2k}+X_{2k}^{d/2}) & 1_{2k}+X_{2k}^{d/2} 
\end{bmatrix} \\
&\cong \Coker (1_{2k}+X_{2k}^{d/2}) \cong \Z^{d/2}. 
\end{align*}
The proof for the case $r \notin 2\Z$ is completed.
\end{proof}

\begin{Lem}[type V]\label{coker_V_1}
Let $n \ge 4$ and $k \ge 1$ be integers.
Put $d=\gcd(2n-2,k)$.
Then we have
\begin{align*}
\Coker \begin{bmatrix} 1_{2k}+X_{2k}^{n-1} & 1_{2k}-X_{2k}^k \end{bmatrix}
&\cong \begin{cases}
\Z^{d/2}   & (k \in 2\Z, \ r \notin 2\Z) \\
(\Z/2\Z)^d & \mathrm{(otherwise)}
\end{cases}.
\end{align*}
\end{Lem}

\begin{proof}
Put $q=(2n-2)/d$.
We can deduce
\begin{align*}
\gcd(n-1,k) = \begin{cases}
d/2   & (k \in 2\Z, \ r \notin 2\Z) \\
d     & \mathrm{(otherwise)}
\end{cases}, \quad
\frac{k}{\gcd(n-1,k)} = \begin{cases}
2q \in    2\Z & (k \in 2\Z, \ r \notin 2\Z) \\
 q \notin 2\Z & \mathrm{(otherwise)}
\end{cases}.
\end{align*}
From Lemma \ref{basic_coker} (4) and then (2),
\begin{align*}
\Coker \begin{bmatrix} 1_{2k}+X_{2k}^{n-1} & 1_{2k}-X_{2k}^k \end{bmatrix}
\cong \Coker (1_k+X_k^{n-1}) 
\cong \begin{cases}
\Z^{d/2}   & (k \in 2\Z, \ r \notin 2\Z) \\
(\Z/2\Z)^d & \mathrm{(otherwise)}
\end{cases}.
\end{align*}
The assertion is proved.
\end{proof}

\begin{Lem}[type V]\label{coker_V_2}
Let $n \ge 4$ and $k \ge 1$ be integers and $n \notin 2\Z$.
Put $d=\gcd(2n-2,k)$.
Then we have
\begin{align*}
\Coker \begin{bmatrix} s_{2n-2}(X_{2k}) & 1_{2k}+X_{2k}^{k-(n-1)} \end{bmatrix}
&\cong \begin{cases}
\Z^d                      & (k \in    2\Z, \ r \in    4\Z) \\
(\Z/2\Z)^{2d-1}           & (k \in    2\Z, \ r \in    2+4\Z) \\
\Z^{d/2}                  & (k \in    2\Z, \ r \notin 2\Z) \\
\Z^{d-1} \oplus (\Z/r\Z)  & (k \notin 2\Z)
\end{cases}.
\end{align*}
\end{Lem}

\begin{proof}
We can deduce
\begin{align*}
\gcd(k-(n-1),2k) &= \begin{cases}
 d   & (k \in    2\Z, \ r \in    4\Z) \\
2d   & (k \in    2\Z, \ r \in    2+4\Z) \\
 d/2 & (k \in    2\Z, \ r \notin 2\Z) \\
 d   & (k \notin 2\Z)
\end{cases}. 
\end{align*}

If $k \in 2\Z$ and $r \in 4\Z$, then we have $(2n-2)/d \in 2\Z$.
By Lemma \ref{basic_coker} (4) and the fact $s_{2n-2}(x)$ can be divided by $1+x^d$ and 
then Lemma \ref{basic_coker} (2),
we have
\begin{align*}
\Coker \begin{bmatrix} s_{2n-2}(X_{2k}) & 1_{2k}+X_{2k}^{k-(n-1)} \end{bmatrix}
\cong \Coker \begin{bmatrix} s_{2n-2}(X_{2k}) & 1_{2k}+X_{2k}^d \end{bmatrix} 
\cong \Coker (1_{2k}+X_{2k}^d)
\cong \Z^d.
\end{align*}
The proof for the case $k \in 2\Z$ and $r \in 4\Z$ is completed.

If $k \in 2\Z$ and $r \in 2+4\Z$, then we have $(2n-2)/2d \in 1+2\Z$.
We can deduce $s_{2n-2}(x)-s_{2d}(x)$ can be divided by $1+x^{2d}$.
Therefore, by Lemma \ref{basic_coker} (4), the previous fact and 
Lemma \ref{basic_coker} (3), 
we have
\begin{align*}
\Coker \begin{bmatrix} s_{2n-2}(X_{2k}) & 1_{2k}+X_{2k}^{k-(n-1)} \end{bmatrix}
&\cong \Coker \begin{bmatrix} s_{2n-2}(X_{2k}) & 1_{2k}+X_{2k}^{2d} \end{bmatrix} \\
&\cong \Coker \begin{bmatrix} s_{2d}(X_{2k}) & 1_{2k}+X_{2k}^{2d} \end{bmatrix}
\cong \Coker \begin{bmatrix} s_{2d}(-X_{2k}) & 1_{2k}+X_{2k}^{2d} \end{bmatrix}.
\end{align*}
Using the fact $1-x^{2d}$ can be divided by $s_{2d}(-x)$
and Lemma \ref{basic_coker} (5), we can deduce
\begin{align*}
\Coker \begin{bmatrix} s_{2d}(-X_{2k}) & 1_{2k}+X_{2k}^{2d} \end{bmatrix}
&\cong \Coker 
\begin{bmatrix} s_{2d}(-X_{2k}) & 2 \cdot 1_{2k} \end{bmatrix} 
\cong (\Z/2\Z)^{2d-1}.
\end{align*}
The proof for the case $k \in 2\Z$ and $r \in 2+4\Z$ is completed.

If $k \in 2\Z$ and $r \notin 2\Z$, then we have $(2n-2)/(d/2) \in 2\Z$.
By Lemma \ref{basic_coker} (4) and the fact $s_{2n-2}(x)$ can be divided by $1+x^{d/2}$ 
and then Lemma \ref{basic_coker} (2),
we have
\begin{align*}
\Coker \begin{bmatrix} s_{2n-2}(X_{2k}) & 1_{2k}+X_{2k}^{k-(n-1)} \end{bmatrix}
\cong \Coker \begin{bmatrix} s_{2n-2}(X_{2k}) & 1_{2k}+X_{2k}^{d/2} \end{bmatrix} 
\cong \Coker (1_{2k}+X_{2k}^{d/2}) 
\cong \Z^{d/2}.
\end{align*}
The proof for the case $k \in 2\Z$ and $r \notin 2\Z$ is completed.

If $k \notin 2\Z$, then we have $k-(n-1) \notin 2\Z$.
By Lemma \ref{basic_coker} (3) and then (4), we have
\begin{align*}
\Coker \begin{bmatrix} s_{2n-2}(X_{2k}) & 1_{2k}+X_{2k}^{k-(n-1)} \end{bmatrix}
&\cong \Coker \begin{bmatrix} s_{2n-2}(-X_{2k}) & 1_{2k}-X_{2k}^{k-(n-1)} \end{bmatrix} 
\cong \Coker (s_{2n-2}(-X_d)).
\end{align*}
Now $s_{2n-2}(-X_d)=r \cdot s_d(-X_d)$ and Lemma \ref{basic_coker} (5) imply
\begin{align*}
\Coker (s_{2n-2}(-X_d))
= \Coker (r \cdot s_d(-X_d)) 
\cong \Z^{d-1} \oplus (\Z/r\Z).
\end{align*}
The proof for the case $k \notin 2\Z$ is completed.
\end{proof}

\begin{Lem}[type V]\label{coker_V_3}
Let $n \ge 4$ and $k \ge 1$ be integers and $n \in 2\Z$.
Put $d=\gcd(2n-2,k)$.
Then we have
\begin{align*}
\Coker \begin{bmatrix} 1_{2k}+X_{2k}^{n-1} & (1-X_{2k}^k)s_{n-1}(X_{2k}) \end{bmatrix}
&\cong \begin{cases}
\Z^{d/2}                  & (k \in    2\Z) \\
\Z^{d-1} \oplus (\Z/r\Z)  & (k \notin 2\Z)
\end{cases}.
\end{align*}
\end{Lem}

\begin{proof}
We can deduce
\begin{align*}
\gcd(n-1,2k) &= \begin{cases}
d/2 & (k \in    2\Z) \\
d   & (k \notin 2\Z)
\end{cases}.
\end{align*}

If $k \in 2\Z$, then we have $k/(d/2) \in 2\Z$.
Lemma \ref{basic_coker} (4) and the fact $1-x^k$ can be divided by $1+x^{d/2}$ yield
\begin{align*}
\Coker \begin{bmatrix} 1_{2k}+X_{2k}^{n-1} & (1-X_{2k}^k)s_{n-1}(X_{2k}) \end{bmatrix}
&\cong \Coker 
\begin{bmatrix} 1_{2k}+X_{2k}^{d/2} & (1-X_{2k}^k)s_{n-1}(X_{2k}) \end{bmatrix} \\
&\cong \Coker (1_{2k}+X_{2k}^{d/2}) \cong \Z^{d/2}.
\end{align*}
The proof for the case $k \in 2\Z$ is completed.

If $k \notin 2\Z$, then we have $d \notin 2\Z$.
Lemma \ref{basic_coker} (4) and the fact $1+x^k$ can be divided by $1+x^d$ yield
\begin{align*}
\Coker \begin{bmatrix} 1_{2k}+X_{2k}^{n-1} & (1-X_{2k}^k)s_{n-1}(X_{2k}) \end{bmatrix}
&\cong 
\Coker \begin{bmatrix} 1_{2k}+X_{2k}^d & (1-X_{2k}^k)s_{n-1}(X_{2k}) \end{bmatrix} \\
&\cong
\Coker \begin{bmatrix} 1_{2k}+X_{2k}^d & 2 \cdot s_{n-1}(X_{2k}) \end{bmatrix}.
\end{align*}
By Lemma \ref{basic_coker} (3) and then (4), we have
\begin{align*}
\Coker \begin{bmatrix} 1_{2k}+X_{2k}^d & 2 \cdot s_{n-1}(X_{2k}) \end{bmatrix}
&\cong \Coker \begin{bmatrix} 1_{2k}-X_{2k}^d & 2 \cdot s_{n-1}(-X_{2k}) \end{bmatrix} 
\cong \Coker \begin{bmatrix} 2 \cdot s_{n-1}(-X_d) \end{bmatrix},
\end{align*}
and then $s_{n-1}(-X_d)=(r/2)s_d(-X_d)$ and Lemma \ref{basic_coker} (5) imply
\begin{align*}
\Coker \begin{bmatrix} 2 \cdot s_{n-1}(-X_d) \end{bmatrix}
= \Coker \begin{bmatrix} 2 \cdot (r/2)s_d(-X_d) \end{bmatrix} 
= \Coker \begin{bmatrix} r \cdot s_d(-X_d) \end{bmatrix} 
\cong \Z^{d-1} \oplus (\Z/r\Z).
\end{align*}
The proof for the case $k \notin 2\Z$ is completed.
\end{proof}

\begin{Lem}[type VII]\label{coker_VII}
Let $k \ge 1$ be an integer. Put $d=\gcd(12,k)$.
Then we have
\begin{align*}
\Coker ((1_k-X_k)(1_k+X_k^3+X_k^6+X_k^9)) \cong \begin{cases}
\Z \oplus (\Z/4\Z)^{d-1}               & (d=1,3) \\
\Z^{(d+2)/2} \oplus (\Z/2\Z)^{(d-2)/2} & (d=2,6) \\
\Z^{(3d+4)/4}                          & (d=4,12) \\
\end{cases}.
\end{align*}
\end{Lem}

\begin{proof}
From Lemma \ref{basic_coker} (4),
we can deduce 
\begin{align*}
\Coker ((1_k-X_k)(1_k+X_k^3+X_k^6+X_k^9))
&= \Coker \begin{bmatrix} (1_k-X_k)(1_k+X_k^3+X_k^6+X_k^9) & 1_k-X_k^{12} \end{bmatrix} \\
&\cong \Coker ((1_d-X_d)(1_d+X_d^3+X_d^6+X_d^9)).
\end{align*}

Assume $d=1,2,4$.
Then $\Coker ((1_d-X_d)(1_d+X_d^3+X_d^6+X_d^9))=\Coker 0_d \cong \Z^d$.

Assume $d=3,6,12$.
Lemma \ref{d>=2} can be used and then the cokernel is isomorphic to
\begin{align*}
&(\Coker (1_{d/3}+X_{d/3}+X_{d/3}^2+X_{d/3}^3))^2
\oplus \Coker ((1_{d/3}-X_{d/3})(1_{d/3}+X_{d/3}+X_{d/3}^2+X_{d/3}^3)) \\
&= (\Coker ((12/d) \cdot s_{d/3}(-X_{d/3})))^2 \oplus \Coker 0_{d/3}.
\end{align*}
From Lemma \ref{coker_s}, it is isomorphic to
\begin{align*}
(\Z^{d/3-1} \oplus (\Z/(12/d)\Z))^2 \oplus \Z^{d/3} 
= \Z^{d-2} \oplus (\Z/(12/d)\Z))^2.
\end{align*}

We can easily check that the assertion holds.
\end{proof}

\begin{Lem}[type VIII]\label{coker_VIII_1}
Let $k \ge 1$ be an integer. Put $d=\gcd(12,k)$.
Then we have
\begin{align*}
\Coker \begin{bmatrix} (1_{2k}-X_{2k})(1_{2k}+X_{2k}^3+X_{2k}^6+X_{2k}^9) 
& 1_{2k}+X_{2k}^{k-6} \end{bmatrix} 
&\cong \begin{cases}
\Z^d                & (d=1,3)  \\
(\Z/2\Z)^{(3d+2)/2} & (d=2,6)  \\
\Z^{d/2}            & (d=4,12) \\
\end{cases}
\end{align*}
\end{Lem}

\begin{proof}
If $d=1,3$, then we have $\gcd(k-6,2k)=d$ and $1+x^d$ can divide $(1-x)(1+x^3+x^6+x^9)$.
From Lemma \ref{basic_coker} (4), we can deduce
\begin{align*}
&\Coker \begin{bmatrix} (1_{2k}-X_{2k})(1_{2k}+X_{2k}^3+X_{2k}^6+X_{2k}^9) 
& 1_{2k}+X_{2k}^{k-6} \end{bmatrix} \\
&\cong \Coker \begin{bmatrix} (1_{2k}-X_{2k})(1_{2k}+X_{2k}^3+X_{2k}^6+X_{2k}^9)  
& 1_{2k}+X_{2k}^d \end{bmatrix} 
\cong \Coker (1_{2k}+X_{2k}^d) \cong \Z^d.
\end{align*}
The proof for the case $d=1,3$ is completed.

If $d=2,6$, then we have $\gcd(12,2k)=2d$.
From Lemma \ref{basic_coker} (4), we can deduce
\begin{align*}
&\Coker \begin{bmatrix} (1_{2k}-X_{2k})(1_{2k}+X_{2k}^3+X_{2k}^6+X_{2k}^9) 
& 1_{2k}+X_{2k}^{k-6} \end{bmatrix} \\
&\cong \Coker \begin{bmatrix} (1_{2k}-X_{2k})(1_{2k}+X_{2k}^3+X_{2k}^6+X_{2k}^9) 
& 1_{2k}-X_{2k}^{12} & 1_{2k}+X_{2k}^{k-6} \end{bmatrix} \\
&\cong \Coker \begin{bmatrix} (1_{2k}-X_{2k})(1_{2k}+X_{2k}^3+X_{2k}^6+X_{2k}^9) 
& 1_{2k}-X_{2k}^{2d} & 1_{2k}+X_{2k}^{k-6} \end{bmatrix}.
\end{align*}
From Lemma \ref{basic_coker} (4) and the fact that $2d$ divides $k-6$, it is isomorphic to
\begin{align*}
&\Coker \begin{bmatrix} (1_{2d}-X_{2d})(1_{2d}+X_{2d}^3+X_{2d}^6+X_{2d}^9) 
& 1_{2k}+X_{2d}^{k-6} \end{bmatrix} \\
&\cong \Coker \begin{bmatrix} (1_{2d}-X_{2d})(1_{2d}+X_{2d}^3+X_{2d}^6+X_{2d}^9) 
& 2 \cdot 1_{2d} \end{bmatrix}.
\end{align*}
If $d=2$, this can be calculated as
\begin{align*}
\Coker \begin{bmatrix} (1_{2d}-X_{2d})(1_{2d}+X_{2d}^3+X_{2d}^6+X_{2d}^9)
& 2 \cdot 1_{2d} \end{bmatrix} 
&=\Coker \begin{bmatrix} (1_4-X_4)(1_4+X_4^3+X_4^6+X_4^9)
& 2 \cdot 1_4 \end{bmatrix} \\
&\cong \Coker \begin{bmatrix} 0_4 & 2 \cdot 1_4 \end{bmatrix} 
\cong (\Z/2\Z)^4.
\end{align*}
If $d=6$, we have
\begin{align*}
\Coker \begin{bmatrix} (1_{2d}-X_{2d})(1_{2d}+X_{2d}^3+X_{2d}^6+X_{2d}^9)
& 2 \cdot 1_{2d} \end{bmatrix} 
=\Coker \begin{bmatrix} (1_{12}-X_{12})(1_{12}+X_{12}^3+X_{12}^6+X_{12}^9)
& 2 \cdot 1_{12} \end{bmatrix}.
\end{align*}
Apply Lemmas \ref{d>=2} and \ref{coker_s}, then it is isomoprhic to 
\begin{align*}
&(\Coker \begin{bmatrix} 1_4+X_4+X_4^2+X_4^3
& 2 \cdot 1_4 \end{bmatrix})^2
\oplus \Coker \begin{bmatrix} (1_4-X_4)(1_4+X_4+X_4^2+X_4^3)
& 2 \cdot 1_4 \end{bmatrix} \\
&= (\Coker \begin{bmatrix} s_4(-X_4)
& 2 \cdot 1_4 \end{bmatrix})^2
\oplus \Coker \begin{bmatrix} 0_4
& 2 \cdot 1_4 \end{bmatrix} 
\cong ((\Z/2\Z)^3)^2 \oplus (\Z/2\Z)^4 
=(\Z/2\Z)^{10}.
\end{align*}
The proof for the case $d=2,6$ is completed.

If $d=4,12$, then we have 
$\gcd(k-6,6)=d/2$ and $1+x^{d/2}$ can divide $(1-x)(1+x^3+x^6+x^9)$.
From Lemma \ref{basic_coker} (4), we can deduce
\begin{align*}
&\Coker \begin{bmatrix} (1_{2k}-X_{2k})(1_{2k}+X_{2k}^3+X_{2k}^6+X_{2k}^9) 
& 1_{2k}+X_{2k}^{k-6} \end{bmatrix} \\
&\cong \Coker \begin{bmatrix} (1_{2k}-X_{2k})(1_{2k}+X_{2k}^3+X_{2k}^6+X_{2k}^9) 
& 1_{2k}+X_{2k}^{d/2} \end{bmatrix} 
\cong \Coker (1_{2k}+X_{2k}^{d/2}) \cong \Z^{d/2}.
\end{align*}
The proof for the case $d=4,12$ is completed.
\end{proof}

\begin{Lem}[type VIII]\label{coker_VIII_2}
Let $k \ge 1$ be an integer. Put $d=\gcd(12,k)$.
Then we have
\begin{align*}
\Coker \begin{bmatrix} 1_{2k}+X_{2k}^6 & 1_{2k}+X_{2k}^{k-6} \end{bmatrix} 
&\cong \begin{cases}
(\Z/2\Z)^d & (d=1,3,2,6)  \\
\Z^{d/2}   & (d=4,12) \\
\end{cases}, \\
\Coker \begin{bmatrix} 1_{2k}+X_{2k}^2 & 1_{2k}+X_{2k}^{k-6} \end{bmatrix} 
&\cong \begin{cases}
(\Z/2\Z)   & (d=1,3)  \\
(\Z/2\Z)^2 & (d=2,6)  \\
\Z^2       & (d=4,12) \\
\end{cases}.
\end{align*}
\end{Lem}

\begin{proof}
Using Lemma \ref{basic_coker} (4) and then (2), it is easy to see
\begin{align*}
\Coker \begin{bmatrix} 1_{2k}+X_{2k}^6 & 1_{2k}+X_{2k}^{k-6} \end{bmatrix} 
&\cong \Coker \begin{bmatrix} 1_{2k}+X_{2k}^6 & 1_{2k}-X_{2k}^k \end{bmatrix}
\cong \Coker (1_k+X_k^6) \\
&\cong \begin{cases}
(\Z/2\Z)^d & (d=1,3,2,6)  \\
\Z^{d/2}   & (d=4,12) \\
\end{cases}.
\end{align*}
The remained assertion can be also proved similarly.
\end{proof}

\begin{Lem}[type X]\label{coker_X}
Let $k \ge 1$ be an integer. Put $d=\gcd(30,k)$.
Then we have
\begin{align*}
\Coker ((1_k-X_k+X_k^2)(1_k+X_k^5))
&\cong \begin{cases}
(\Z/2\Z)^d   & (d=1,3,5)  \\
(\Z/2\Z)^7   & (d=15)     \\
\Z^{d/2}     & (d=2,6,10) \\
\Z^7         & (d=30)  
\end{cases}.
\end{align*}
\end{Lem}

\begin{proof}
Assume $d=1,3,2,6$, then we have $\gcd(5,k)=1$.
Lemma \ref{basic_coker} (4) yields
\begin{align*}
\Coker ((1_k-X_k+X_k^2)(1_k+X_k^5)) &\cong \Coker ((1_k-X_k+X_k^2)(1_k+X_k)) \\
&\cong \Coker (1_k+X_k^3) 
\cong \begin{cases}
(\Z/2\Z)^d   & (d=1,3) \\
\Z^{d/2}     & (d=2,6) \\
\end{cases}.
\end{align*}

Assume $d=5,15,10,30$, then we have $\gcd(5,k)=5$.
We have $Y_{k,5}^{-1}(1_k-X_k+X_k^2)(1_k+X_k^5)Y_{k,5}$ is equal to
\begin{align*}
\begin{bmatrix}
 1_{k/5}+X_{k/5}&                &                & X_{k/5}+X_{k/5}^2&-X_{k/5}-X_{k/5}^2\\ 
-1_{k/5}-X_{k/5}& 1_{k/5}+X_{k/5}&                &                  & X_{k/5}+X_{k/5}^2\\
 1_{k/5}+X_{k/5}&-1_{k/5}-X_{k/5}& 1_{k/5}+X_{k/5}&                  &                  \\
                & 1_{k/5}+X_{k/5}&-1_{k/5}-X_{k/5}& 1_{k/5}+X_{k/5}  &                  \\
                &                & 1_{k/5}+X_{k/5}&-1_{k/5}-X_{k/5}  & 1_{k/5}+X_{k/5}
\end{bmatrix}.
\end{align*}
It is transformed into
\begin{align*}
\begin{bmatrix}
(1_{k/5}+X_{k/5})^{\oplus 3} &                   &                   \\
                             & 1_{k/5}-X_{k/5}^2 & X_{k/5}+X_{k/5}^2 \\
                             &-1_{k/5}-X_{k/5}   & 1_{k/5}-X_{k/5}^2
\end{bmatrix}
\ \mbox{and} \ 
\begin{bmatrix}
(1_{k/5}+X_{k/5})^{\oplus 3} &                   &                   \\
                             &                   & 1_{k/5}+X_{k/5}^3 \\
                             &-1_{k/5}-X_{k/5}   & 
\end{bmatrix}.
\end{align*}
Thus we have
\begin{align*}
\Coker ((1_k-X_k+X_k^2)(1_k+X_k^5)) 
&\cong (\Coker (1_{k/5}+X_{k/5}))^4 \oplus \Coker (1_{k/5}+X_{k/5}^3) \\
&\cong \begin{cases}
(\Z/2\Z)^4 \oplus (\Z/2\Z)^{d/5}  & (d=5,15) \\
\Z^4       \oplus \Z^{d/10}       & (d=10,30)
\end{cases}.
\end{align*}

Now, the assertion can be proved easily.
\end{proof}

\section{Maximal rigid and cluster-tilting objects}

In this section, 
we deal with maximal rigid and cluster-tilting objects of 
the stable module categories of finite-dimensional mesh algebras
and use them as invariants of stable equivalences.
In the rest, we assume that $K$ is an algebraically closed field.
Let $Q=Q_{\Delta,l,t}$ be a translation quiver in Definition \ref{quiver_def}
and $\Lambda=\Lambda_{\Delta,l,t}$ be the corresponding finite-dimensional mesh algebra.

We consider an automorphism $\mu \colon \Lambda \to \Lambda$
and construct an autoequivalence $\mu_*$ on $\mod \Lambda$ from $\mu$ as follows.

\begin{Def}
Let $Q=Q_{\Delta,l,t}$, $\Lambda=\Lambda_{\Delta,l,t}$, 
and $\mu \colon \Lambda \to \Lambda$ be an automorphism on the $K$-algebra $\Lambda$.
We define an autoequivalence $\mu_* \colon \mod \Lambda \to \mod \Lambda$
as $\mu_*=(? \otimes_\Lambda ({_1 \Lambda _{\mu^{-1}}}))$,
where the right action of $\Lambda$ on ${_1 \Lambda _{\mu^{-1}}}$ is defined by
$x \cdot \lambda = x\mu^{-1}(\lambda)$ 
in $x \in {_1 \Lambda _{\mu^{-1}}}$ and $\lambda \in \Lambda$.

If functor $\mu_*$ is restricted to the projective $\Lambda$-modules,
then $\mu_*$ also acts on $\umod \Lambda$.
\end{Def}

If $\rho \in \Aut Q$ is an automorphism on a translation quiver $Q$, 
we naturally extend $\rho$ to a natural automorphism 
$\rho \colon \Lambda \to \Lambda$,
and define the functor $\rho_* \colon \mod \Lambda \to \mod \Lambda$ as above.
The functor $\rho_*$ is restricted to the projective $\Lambda$-modules.
We consider the quotient quiver of $Q$ by $\rho$,
and define the push-down funtor and the pull-up functor.

\begin{Def}
Let $Q=Q_{\Delta,l,t}$, $\Lambda=\Lambda_{\Delta,l,t}$, 
and $\rho \in \Aut Q$ such that its order on $Q$ is $m$.
We say $\rho$ is \textit{free} 
if $\rho^j(u) \ne u$ for any $u \in Q_0$ and $j=1,\ldots,m-1$.
If $\rho$ is free, we write $Q/\langle \rho \rangle$ 
for the quotient translation quiver of $Q$ by $\rho$,
and $\Lambda/\langle \rho \rangle$ for the corresponding quotient mesh algebra.

The \textit{push-down} functor 
$\Phi_\rho \colon \mod \Lambda \to \mod (\Lambda/\langle \rho \rangle)$
is defined as follows; 
for $M \in \mod \Lambda$, $\Phi_\rho(M)$ is a $\Lambda/\langle \rho \rangle$-module
such that $\Phi_\rho(M)e_{\bar{u}}=\bigoplus_{j=0}^{m-1}Me_{\rho^j(u)}$ 
for $\bar{u} \in (Q/\langle \rho \rangle)_0$ and that
the action of $\bar{\alpha} \in Q_1$ on $\Phi_\rho(M)$ is the direct sum of
the actions of $\alpha,\rho(\alpha),\ldots,\rho^{m-1}(\alpha) \in Q_1$ on $M$.
The \textit{pull-up} functor 
$\Psi_\rho \colon \mod (\Lambda/\langle \rho \rangle) \to \mod \Lambda$
is defined as follows; 
for $M' \in \mod (\Lambda/\langle \rho \rangle)$, 
$\Psi_\rho(M')$ is a $\Lambda$-module
such that $\Psi_\rho(M')e_u=M'e_{\bar{u}}$ for $u \in Q_0$ and that
the action of $\alpha \in Q_1$ on $\Psi_\rho(M')$ coincides with
the action of $\bar{\alpha} \in (Q/\langle \rho \rangle)_1$ on $M'$.

The functors $\Phi_\rho$ and $\Psi_\rho$ are restricted to the projective modules,
thus they induce the functors 
between $\umod \Lambda$ and $\umod (\Lambda/\langle \rho \rangle)$.
\end{Def}

We recall the Serre functor of a triangulated category here. 
The Serre functor of a Hom-finite $K$-linear additive 
triangulated category $\mathcal{T}$ is a functor $\bS$ 
such that $\Hom_\mathcal{T}(X,Y) \cong D \Hom_\mathcal{T}(Y,\bS X)$ 
holds functorially for $X,Y \in \mathcal{C}$ ($D$ denotes the $K$-dual $\Hom_K(?,K)$).
It is unique up to functorial isomorphisms,
and commutes with triangle equivalences.
Explicitly, $\umod \Lambda_{\Delta,l,t}$ has $\bS=[-1] \circ \nu$ as the Serre functor,
where $\nu$ is the Nakayama functor 
$({?} \otimes_{\Lambda} D\Lambda_{\Delta,l,t}) \colon \umod \Lambda_{\Delta,l,t}
\to \umod \Lambda_{\Delta,l,t}$ (see \cite[IV.2.4, IV.2.13]{ASS}).

We have the following properties. 

\begin{Prop}\label{l=l'}
Let $\Lambda=\Lambda_{\Delta,l,t}, \Lambda'=\Lambda_{\Delta',l',t'}$ be 
finite-dimensional mesh algebras
and $F \colon \umod \Lambda \to \umod \Lambda'$ 
be a stable equivalence as triangulated categories.
\begin{itemize}
\item[(1)]
Let $\bS,\bS'$ be the Serre functors of 
$\umod \Lambda, \umod \Lambda'$, and
$[1],[1]'$ be the shifts of 
$\umod \Lambda, \umod \Lambda'$.
We have $F([-2] \circ \bS) \cong ([-2]' \circ \bS')F$ on $\umod \Lambda'$.
\item[(2)]
Let $Q=\Lambda_{\Delta,l,t}$.
Then $[-2] \circ \bS$ satisfies the following.
\begin{itemize}
\item[(i)]
If $\Lambda$ is type II, i.e. $\Lambda=\Lambda_{A_n,2k,2}$,
then $[-2] \circ \bS \cong (\tau_*\theta_*)^{-1}$ on $\umod \Lambda$
for some automorphism $\theta \colon \Lambda \to \Lambda$
such that $\theta(e_u)=e_u$ for $u \in Q_0$
and that $\theta(\alpha)$ is $\alpha$ or $-\alpha$ for $\alpha \in Q_1$.
\item[(ii)]
If $\Lambda$ is type III, i.e. $\Lambda=\Lambda_{A_n,2k-1,2}$,
then $[-2] \circ \bS \cong (\tau_*\kappa_*)^{-1}$ on $\umod \Lambda$
for the automorphism $\kappa \colon \Lambda \to \Lambda$
such that $\kappa(e_u)=e_u$ for $u \in Q_0$
and that $\kappa(\alpha)=-\alpha$ for $\alpha \in Q_1$.
\item[(iii)]
In the other cases, we have $[-2] \circ \bS \cong \tau_*^{-1}$ on $\umod \Lambda$. 
\end{itemize}
\item[(3)]
Assume $\Delta \ne A_1$.
The order of the functor $[-2] \circ \bS$ on $\umod \Lambda$ 
up to functorial isomorphisms is $l$
if $\Lambda$ is not type III, and is $l$ or $2l$ if $\Lambda$ is type III.
\end{itemize}
\end{Prop}

\begin{proof}
(1)
It is well-known.

(2)
We only prove (i).
The other assertions are similarly proved.
We have $[3] \cong \nu \circ \tau_* \circ \theta_*$ 
from \cite[Proposition 5.5]{Dugas}.
We also have $\bS=[-1] \circ \nu$.
It is easy to see $[1] \circ \tau_* \cong \tau_* \circ [1]$, 
thus we have $(\tau_*\theta_*)^{-1} \cong [-2] \circ \bS$.

(3)
By (2) and \cite[Proposition 4.4]{AS},
it is easy to see that $([-2] \circ \bS)^l \cong \mathrm{id}$
if $\Lambda$ is not type III
and that $([-2] \circ \bS)^{2l} \cong \mathrm{id}$ if $\Lambda$ is type III. 

On the other hand, a simple $\Lambda_{\Delta,l,t}$-module $S_u$ is a nonzero object
in $\umod \Lambda_{\Delta,l,t}$ by the assumption $\Delta \ne A_1$.
It is easy to see that $([-2] \circ \bS)^i(S_u) \cong S_{\tau^{-i} u}$ and that
$S_{\tau^{-i} u} \cong S_u$ in $\umod \Lambda_{\Delta,l,t}$ implies $\tau^{-i} u=u$.
By the construction, $l$ is the minimal integer $i \ge 1$ such that $\tau^{-i} u=u$ holds
for every $u \in (Q_{\Delta,l,t})_0$.

Now the assertion is easily obtained.
\end{proof}

Now we state the definition of cluster-tilting objects.

\begin{Def}\label{rigid}
Let $\Lambda$ be a finite-dimensional self-injective $K$-algebra, 
$\mathcal{C}$ be $\mod \Lambda$ or $\umod \Lambda$.
Assume that $T$ is an object in $\mathcal{C}$ and let
$\add_{\mathcal{C}} T \subset \mathcal{C}$ be the full subcategory 
of objects which are direct summands of $T^m$ for some $m$.
\begin{itemize}
\item[(1)]
We say $T$ is \textit{rigid} if $\Ext_\Lambda^1(T,T)=0$.
We say $T$ is \textit{maximal rigid} 
if $T$ satisfies the following; 
$T$ is rigid, and if $U \in \mod \Lambda$ satisfies that $T \oplus U$ is rigid
then $U \in \add_{\mathcal{C}}T$.
We say $T$ is \textit{cluster-tilting} if 
$\add_{\mathcal{C}} T = \{ M \in \mathcal{C} \mid \Ext_\Lambda^1(M,T)=0 \}
= \{ M \in \mathcal{C} \mid \Ext_\Lambda^1(T,M)=0 \}$.
\item[(2)]
Let $F \colon \mathcal{C} \to \mathcal{C}$ be an autoequivalence.
We say $T$ is \textit{$F$-stable} if $F(T) \cong T$ in $\mathcal{C}$.
We say $T$ is \textit{$F$-stable rigid} if $T$ is $F$-stable and rigid.
We say $T$ is \textit{maximal $F$-stable rigid} 
if $T$ satisfies the following; 
$T$ is $F$-stable rigid, and 
if $U \in \mod \Lambda$ satisfies that $T \oplus U$ is $F$-stable rigid
then $U \in \add_{\mathcal{C}}T$.
We say $T$ is \textit{$F$-stable cluster-tilting} if 
$T$ is $F$-stable and cluster-tilting.
\end{itemize} 
\end{Def}

It is clear that a cluster-tilting object is always maximal rigid.
We also define the symbols of the number of indecomposable direct summands.

\begin{Def}
Let $\Lambda$ be a finite-dimensional self-injective $K$-algebra and 
$T$ be an object in $\mod \Lambda$.
Decompose $T$ as 
$T \cong \bigoplus_{i=1}^l T_i^{m_i}$ in $\mod \Lambda$
with $T_1,\ldots,T_l$ pairwise nonisomorphic indecomposable direct summands and
$m_i \ge 1$ for all $i$.
Then we write $|T|=l$.
Moreover, $T$ is called \textit{basic} if $m_i=1$ for all $i$.
If exactly $l'$ of $l$ modules $T_1,\ldots,T_l$ are nonprojective,
then we write $|T|_{\mathrm{np}}=l'$.
\end{Def}

Note that if $T$ is a cluster-tilting object in $\mod \Lambda$, 
then $T$ contains $\Lambda$ as a direct summand,
and thus $|T|_{\mathrm{np}}=|T|-m$, 
where $m$ is the number of the isomorphic classes of indecomposable projective 
$\Lambda$-modules.
Clearly, the basic cluster-tilting objects in $\mod \Lambda$ correspond bijectively to 
the basic cluster-tilting objects in $\umod \Lambda$.
 
The following proposition on the number of the indecomposable direct summands of
a cluster-tilting object is very important.

\begin{Prop}\label{ind_coincide}
\cite[5.3.3, Corollary]{Iyama}
Let $\Lambda$ be a finite-dimensional self-injective $K$-algebra.
If $T_1$ and $T_2$ are cluster-tilting objects in $\mod \Lambda$,
then we have $|T_1| = |T_2|$ 
and $|T_1|_{\mathrm{np}}=|T_2|_{\mathrm{np}}$.
\end{Prop}

Now we recall an important result 
on cluster-tilting objects for preprojective algebras.
This associates \textit{reduced expressions} of the \textit{longest element} of
the \textit{Coxeter group} to cluster-tilting objects.

\begin{Def}
Let $\Delta$ be a Dynkin diagram with $n$ vertices. 
We assume that the vertices are numbered as in Subsection 2.3.

The \textit{Coxeter group} $W=W_\Delta$ associated to $\Delta$ is defined as follows;
the generators are $s_1,\ldots,s_n$ and the relations are
(i) $s_i^2=1_W$, 
(ii) $s_is_j=s_js_i$ if there exists no edge between vertices $i$ and $j$ in $\Delta$,
(iii) $s_is_js_i=s_js_is_j$ if there exists exactly one edge between vertices $i$ and $j$
in $\Delta$.

For an element $w \in W$, the \textit{length} of $w$ is a minimum number $l$
such that there exists a sequence $(i_1,\ldots,i_l)$ such that
$w=s_{i_1} \cdots s_{i_l}$.
An element $w \in W$ with the maximum length is called a \textit{longest element}.
\end{Def}

The Coxeter group $W=W_\Delta$ associated to a Dynkin diagram is a finite group,
and in this case, there uniquely exists a longest element in $W$.
The length of the longest element is $nc/2$,
where $c=c_\Delta$ is the Coxeter number of $\Delta$.
For the detail of Coxeter groups, see \cite{BBa}.

\begin{Prop}\label{long_elem}\cite[Theorem III.3.5, Corollary III.3.6]{BIRS}
Let $\Delta$ be a Dynkin diagram with its vertices $\{1,\ldots,n\}$,
$W=W_\Delta$ be the Coxeter group,
$c=c_\Delta$ be the Coxeter number,
and $\Lambda=\Lambda_{\Delta,1,1}$ be the preprojective algebra.
Put the ideal $I_i=\Lambda(1-e_i)\Lambda \subset \Lambda$ for $i=1,\ldots,n$.

Let $s_{i_1}s_{i_2} \cdots s_{i_{nc/2}}$ be
a reduced expression of the longest element of $W$
and put $T'_m=e_{i_m} (\Lambda/I_{i_m} \cdots I_{i_2} I_{i_1})$
for $m=1,\ldots,nc/2$.
Then we have a basic cluster-tilting object
$T'=\bigoplus_{m=1}^{nc/2} T'_m$
in $\mod \Lambda$,
where each $T'_m$ has a simple top and is indecomposable for $m=1,\ldots,nc/2$.
Especially, we have $|T'|=nc/2$ and $|T'|_{\mathrm{np}}=n(c-2)/2$.  
\end{Prop}
%
%

We will extend Proposition \ref{long_elem} for general mesh algebras.

\begin{Lem}\label{orbit_mult}
Let $Q=Q_{\Delta,l,t}$, $\Lambda=\Lambda_{\Delta,l,t}$, 
and $\rho \in \Aut Q$ be free with its order $m$.
Suppose $\eta \in \Aut Q$ and $\bar{\eta} \in \Aut (Q/\langle \rho \rangle)$
satisfy $q_\rho \eta=\bar{\eta} q_\rho$, where
$q_\rho \colon Q \to Q/\langle \rho \rangle$ 
is the quotient morphism of translation quivers.
\begin{itemize}
\item[(1)] 
Assume that $T$ is a $\rho_*$-stable, $\eta_*$-stable rigid (resp. cluster-tilting) object
with $T=\bigoplus_{i=0}^{m-1}\rho_*^i(U)$ in $\mod \Lambda$.
Then $T':=\Phi_\rho(U)$ is 
$\bar{\eta}_*$-stable rigid (resp. cluster-tilting) object 
in $\mod (\Lambda/\langle \rho \rangle)$ and 
we have $|T'| \ge |T|/m$ and $|T'|_{\mathrm{np}} \ge |T|_{\mathrm{np}}/m$.
Moreover, if $T$ is basic, then $T'$ is basic and  
$|T'| = |T|/m$ and $|T'|_{\mathrm{np}} = |T|_{\mathrm{np}}/m$.
\item[(2)]
Assume $T'$ is an $\bar{\eta}_*$-stable rigid (resp. cluster-tilting) object 
in $\mod (\Lambda/\langle \rho \rangle)$
with $T' \cong \Phi_\rho(U)$ for some $U \in \mod \Lambda$. 
Let $T:=\Psi_\rho(T')$,
then $T \cong \bigoplus_{i=0}^{m-1} \rho_*^i(U)$ and 
$T$ is a $\rho_*$-stable, $\eta_*$-stable rigid (resp. cluster-tilting) object 
in $\mod \Lambda$.
Moreover, if $T$ is basic, then we have $|T|=m|T'|$ and 
$|T|_{\mathrm{np}}=m|T'|_{\mathrm{np}}$.
\end{itemize}
\end{Lem}

\begin{proof}
(1) 
Assume that $T$ is a $\rho_*$-stable, $\eta_*$-stable rigid object
with $T=\bigoplus_{i=0}^{m-1}\rho_*^i(U)$ in $\mod \Lambda$.
Note that the functor 
$\Phi_\rho \colon \mod \Lambda \to \mod (\Lambda/\langle \rho \rangle)$ sends
the projective $\Lambda$-modules to 
the projective $\Lambda/\langle \rho \rangle$-modules and that it is exact,
we have
\begin{align*}
\Ext_{\Lambda/\langle \rho \rangle}^1(T',T')
=\Ext_{\Lambda/\langle \rho \rangle}^1(\Phi_\rho(U),\Phi_\rho(U)) \cong 
\bigoplus_{i=0}^{m-1}\Ext_{\Lambda}^1(\rho_*^i(U),U) = \Ext_{\Lambda}^1(T,U)=0,
\end{align*}
and $T'$ is rigid.

Now we consider numbers of indecomposable direct summands.
Note that $\Phi_\rho(M)$ may not be indecomposable even 
if $M$ is indecomposable in $\mod \Lambda$.

Assume that two indecomposable modules $M_1,M_2$ in $\mod \Lambda$
satisfy that $\Phi_\rho(M_1)$ and $\Phi_\rho(M_2)$ have
a common indecomposable direct summand $M'$.
Then $\Psi_\rho(\Phi_\rho(M_1)) \cong \bigoplus_{i=0}^{m-1} \rho_*^i (M_1)$
and $\Psi_\rho(\Phi_\rho(M_2)) \cong \bigoplus_{i=0}^{m-1} \rho_*^i (M_2)$
have a common nonzero direct summand $\Psi_\rho(M')$.
Thus $M_1$ is isomorphic to $\rho_*^i(M_2)$ for some $i=0,1,\ldots,m-1$,
and we have $\Phi_\rho(M_1) \cong \Phi_\rho(M_2)$.
Therefore, we have $|T'| \ge |T|/m$ and 
$|T|_{\mathrm{np}} \ge |T|_{\mathrm{np}}/m$.

If $T$ is basic, then $U$ is basic and 
any two modules of $U, \rho_*(U), \ldots, \rho_*^{m-1}(U)$ 
have no nonzero common direct summands.
In this case, $\Phi_\rho$ sends the pairwise 
nonisomorphic indecomposable direct summands of $U$ 
to the pairwise nonisomorphic indecomposable direct summands of $T$, 
see \cite[3.5, Lemma]{Gabriel}.
Therefore, $T'$ must be also basic and 
we have $|T'| = |U| = |T|/m$ and 
$|T'|_{\mathrm{np}} = |U|_{\mathrm{np}}/m = |T|_{\mathrm{np}}/m$.

Now we additionally assume that $T$ is cluster-tilting.
It is easy to see that $T'$ contains $\Lambda/\langle \rho \rangle$ as a direct summand.
By \cite[5.1, Theorem]{Iyama}, the global dimension of $\End_\Lambda T$ is at most 3.
Thus there exists a projective resolution 
\begin{align*}
0 \to \Hom_{\Lambda}(T,U_3) \to 
\Hom_{\Lambda}(T,U_2) \to \Hom_{\Lambda}(T,U_1) \to \Hom_{\Lambda}(T,U)
\to \Hom_{\Lambda}(T,U)/{\rad_{\Lambda}(T,U)} \to 0.
\end{align*}
Put $T'_j=\Phi_\rho(U_j)$ for $j=1,2,3$. 
We have
\begin{align*}
0 \to \Hom_{\Lambda/\langle \rho \rangle}(T',T'_3) \to 
\Hom_{\Lambda/\langle \rho \rangle}(T',T'_2) \to 
\Hom_{\Lambda/\langle \rho \rangle}(T',T'_1) 
& \to \Hom_{\Lambda/\langle \rho \rangle}(T',T') \\ 
& \to \Hom_{\Lambda/\langle \rho \rangle}(T',T')
/{\rad_{\Lambda/\langle \rho \rangle}(T',T')} \to 0,
\end{align*}
because $T=\bigoplus_{i=0}^{m-1}\rho_*^i(U)$.
This sequence shows the global dimension of 
$\End_{\Lambda/\langle \rho \rangle} T'$ is at most 3.
Using \cite[5.1, Theorem]{Iyama} again, $T'$ is cluster-tilting.

(2) 
Let $T'$ be an $\bar{\eta}_*$-stable rigid object 
in $\mod (\Lambda/\langle \rho \rangle)$
with $T' \cong \Phi_\rho(U)$ for some $U \in \mod \Lambda$. 
By the construction of the functor,
it is easy to see that 
$\Psi_\rho(\Phi_\rho(U)) \cong \bigoplus_{i=0}^{m-1} \rho_*^i(U) \cong T$.
We have
\begin{align*}
\Ext_{\Lambda}^1(T,T)
= \bigoplus_{i=0}^{m-1}\Ext_{\Lambda}^1(\rho_*^i(U),T)
\cong \Ext_{\Lambda/\langle \rho \rangle}^1(\Phi_\rho(U),\Phi_\rho(T))
\cong \Ext_{\Lambda/\langle \rho \rangle}^1(T',(T')^m) = 0
\end{align*}
and $T$ is rigid.

Now we additionally assume that $T'$ is cluster-tilting.
Let $M$ be an object in $\mod \Lambda$ with $\Ext_\Lambda^1(T,M)=0$.
We have 
\begin{align*}
\Ext_{\Lambda/\langle \rho \rangle}^1(T',\Phi_\rho(M))
\cong \Ext_{\Lambda/\langle \rho \rangle}^1(\Phi_\rho(U),\Phi_\rho(M)) \cong 
\bigoplus_{i=0}^{m-1}\Ext_\Lambda^1(\rho_*^i(U),M)=\Ext_\Lambda^1(T,M)=0.
\end{align*}
Because $T'$ is cluster-tilting,
$\Phi_\rho(M)$ is in $\add_{\mod (\Lambda/\langle \rho \rangle)}T'$.
This implies that $\Psi_\rho(\Phi_\rho(M)) \cong \bigoplus_{i=0}^{m-1}\rho_*^i(M)$ 
is in $\add_{\mod \Lambda}T$,
and especially, $M$ is in $\add_{\mod \Lambda}T$.
We can similarly show that 
if an object $X$ in $\mod \Lambda$ satisfies $\Ext_\Lambda^1(X,T)=0$,
then $X$ is a direct summand of $T$.
Therefore, $T$ is cluster-tilting.

The remained part is deduced by the part (1).
\end{proof}

\begin{Lem}\label{grade}
In the setting of Proposition \ref{long_elem},
let $k \ge 1$ be an integer and consider the two functors
$\Phi_\tau \colon \mod \Lambda_{\Delta,k,1} \to \mod \Lambda_{\Delta,1,1}$
and $\Psi_\tau \colon \mod \Lambda_{\Delta,1,1} \to \mod \Lambda_{\Delta,k,1}$.
Then there exists an object $U$ in $\mod \Lambda_{\Delta,k,1}$ satisfying 
that $\Phi_\tau(U) \cong T'$ and that $\Psi_\tau(T')$ is basic
and that each indecomposable direct summand of $\Psi_\tau(T')$ has a simple top.
\end{Lem}

\begin{proof}
We show that there exists an object $U$ in $\mod \Lambda_{\Delta,k,1}$ satisfying that 
(i) $\Phi_\tau(U) \cong T'$
and that (ii) $U$ is basic and that 
(iii) each indecomposable direct summand of $U$ has a simple top
and that (iv) 
any two modules of $U,\tau_*(U),\ldots,\tau_*^{k-1}(U)$
have no common indecomposable direct summand.
If it is shown, the assertion is proved by 
$\Psi_\tau(T') \cong \bigoplus_{i=0}^{k-1} \tau_*^i(U)$.

We can define a $\Z$-grading on the preprojective algebra $\Lambda_{\Delta,1,1}$
as follows;
the degree of the idempotent $e_i$ for each vertex $i \in (Q_{\Delta,1,1})_0$ is 0
and the degree of each arrow $(i \to j) \in (Q_{\Delta,1,1})_1$ is 0 if $i < j$ and
1 if $i > j$.
For a finite-dimensional $\Z$-graded $\Lambda_{\Delta,1,1}$-module $M'$,
we associate the following
(non-graded) finite-dimensional $\Lambda_{\Delta,k,1}$-module $M$;
the vector space $Me_{(i,a+k\Z)}=\bigoplus_{b \in a+k\Z}(M'e_i)_b$
for each $(i,a+k\Z) \in (Q_{\Delta,k,1})_0$,
where $(M'e_i)_b$ is the degree $b$ part of the $K$-vector space $M'e_i$;
and the action of each arrow in $Q_{\Delta,k,1}$ on $M$ is naturally defined
by taking the direct sum.
Then we have $\Phi_\tau(M) \cong M'$ as non-graded $\Lambda_{\Delta,1,1}$-modules.
It is easy to see that if $M'$ has a simple top then $M$ also has a simple top.

Therefore, we show that $T'_m$ can be a $\Z$-graded $\Lambda_{\Delta,1,1}$-module.
Because the idempotents $e_1,\ldots,e_n$ and the ideals $I_1,\ldots,I_n$ are $\Z$-graded
by the $\Z$-grading on $\Lambda_{\Delta,1,1}$,
the module $T'_m$ can be also $\Z$-graded.
Thus, we can take an object $U$ in $\mod \Lambda_{\Delta,k,1}$ such that 
$\Phi_\tau(U)$ is isomorphic to 
$\bigoplus_{m=1}^{nc/2} T'_m = T'$.
By Proposition \ref{long_elem}, $T'$ is basic and 
each indecomposable direct summand of $T'$ has a simple top,
so $U$ is also basic and each indecomposable direct summand of $U$ has a simple top.
Therefore, (i), (ii), and (iii) are proved.

We prove the claim (iv).
Assume that $\tau_*^i(U)$ and 
$\tau_*^j(U)$ have a common indecomposable direct summand $X$ 
for some $i,j \in \{0,1,\ldots,k-1\}$ with $i \ne j$.
Then $\tau_*^{-i}(X)$ and $\tau_*^{-j}(X)$
are nonisomorphic indecomposable direct summands of $U$,
because $X$ has a simple top.
Therefore, $\Phi_\tau(X) ^2$ is a direct summand of $\Phi_\tau(U) \cong T'$,
but it is impossible because $T'$ is basic. 
The claim (iv) is proved.
\end{proof}

Now we can show the existence of a basic $([-2] \circ \bS)$-stable 
(see Proposition \ref{l=l'}) cluster-tilting object in $\mod \Lambda$
if $\Lambda$ is a finite-dimensional mesh algebra except of type III: 
$\Lambda_{A_n,2k-1,2}$ ($n \in 2\Z$)
and count the number of its indecomposable direct summands.

\begin{Th}\label{CT_not_III}
Let $\Lambda=\Lambda_{\Delta,l,t}$ be a finite-dimensional mesh algebra 
except of type III
and $n$ be the number of the vertices of $\Delta$, 
$c=c_\Delta$ and $k=l/t$.
Then $\mod \Lambda$ has a basic $\tau_*$-stable cluster-tilting object $T$
with $|T|_{\mathrm{np}}=n(c-2)k/2$.
Moreover, if $\Lambda$ is type II, we can take this $T$ 
as a $\tau_* \theta_*$-stable object. 
\end{Th}

\begin{proof}
First, we consider the case $t=1$. We have $l=k$.
It is easy to see that $Q_{\Delta,1,1}=Q_{\Delta,k,1}/\langle \tau \rangle$.
From Proposition \ref{long_elem}, 
there exists a basic $\tau_*$-stable cluster-tilting object $T'$.
The functor 
$\Psi_\tau \colon \mod \Lambda_{\Delta,1,1} \to \mod \Lambda_{\Delta,k,1}$
sends $T'$ to a basic $\tau_*$-stable cluster-tilting object $T=\Psi_\tau(T')$
with $|T|_{\mathrm{np}}=n(c-2)k/2$ from Lemmas \ref{orbit_mult} (2) and \ref{grade}.

Second, we consider the case $t=2$. We have $l=2k$.
It is easy to see that $Q_{\Delta,1,1}=Q_{\Delta,2k,1}/\langle \tau \rangle$.

By assumption, $\Delta$ is $A_n$ ($n \notin 2\Z$), $D_n$, or $E_6$.
From \cite{BBa}, the following sequence $\gamma$ with its length $nc/2$ 
gives the longest element of 
the Coxeter group of $\Delta$;
if $\Delta$ is $A_n$ ($n \notin 2\Z$), we define
\begin{align*}
&\alpha_m = ((n+1)/2-(m-1),(n+1)/2+(m-1)) \quad (m=2,\ldots,(n+1)/2), \\ 
&\beta_1 = ((n+1)/2), \quad
\beta_m=\alpha_m \cdot \beta_{m-1} \cdot \alpha_m \quad (m=2,\ldots,(n+1)/2), \quad
\gamma=\beta_1 \cdot \beta_2 \cdots \beta_{(n+1)/2};
\end{align*}
if $\Delta$ is $D_n$, we define
\begin{align*}
\beta_1 = (n-1,n), \quad
\beta_m=(n-m) \cdot \beta_{m-1} \cdot (n-m) \quad (m=2,\ldots,n-1), \quad
\gamma=\beta_1 \cdot \beta_2 \cdots \beta_{n-1};
\end{align*}
if $\Delta$ is $E_6$, we define $\beta=(1,2,5,4,6,3)$ and $\gamma=\beta^6$,
where $(a_1,\ldots,a_p)\cdot(b_1,\ldots,b_q)$ denotes 
the composition $(a_1,\ldots,a_p,b_1,\ldots,b_q)$.
The corresponding cluster-tilting object 
in $\mod \Lambda_{\Delta,1,1}$ constructed by Proposition \ref{long_elem} 
is $\psi_*$-stable by \cite[Lemma 3.4.2]{BIRS} 
and each of its indecomposable direct summands has a simple top.
We have a basic $\psi_*$-stable cluster-tilting object 
$T'$ in $\mod \Lambda_{\Delta,1,1}$
with $|T'|_{\mathrm{np}}=n(c-2)/2$.

By Lemma \ref{orbit_mult} (2) and \ref{grade}, 
the functor $\Psi_\tau \colon \mod \Lambda_{\Delta,1,1} \to \mod \Lambda_{\Delta,2k,1}$
sends $T'$ to a basic $\psi_*$-stable, $\tau_*$-stable
(especially $(\tau^k \psi)_*$-stable) cluster-tilting object $T=\Psi_\tau(T')$, and
we have $|T|_{\mathrm{np}}=2k|T'|_{\mathrm{np}}=n(c-2)k$.
Lemma \ref{grade} tells us also that 
each indecomposable direct summand of $T$ has a simple top,
and thus $T$ has no $(\tau^k \psi)_*$-stable indecomposable direct summand. 
Therefore, there exists a direct summand $V$ of $T$ 
such that $T=V \oplus (\tau^k\psi)_*(V)$.

Next, it is easy to see that $Q_{\Delta,2k,2}=Q_{\Delta,2k,1}/\langle \tau^k \psi \rangle$.
By Lemma \ref{orbit_mult} (1), the functor $\Phi_{\tau^k \psi} \colon 
\mod \Lambda_{\Delta,2k,1} \to \mod \Lambda_{\Delta,2k,2}$ sends 
$V$ to a basic $\tau_*$-stable cluster-tilting object 
$T''=\Phi_{\tau^k \psi}(V)$, and
we have $|T''|_{\mathrm{np}}=|T|_{\mathrm{np}}/2=n(c-2)k/2$.

Finally, we consider the case $t=3$. We have $\Delta=D_4$.
The sequence $(3,4,1,2,3,4,1,2,3,4,1,2)$ with its length 12 gives 
the longest element of the Coxeter group of $\Delta=D_4$, 
and the corresponding cluster tilting module is $\chi_*$-stable
and each of its indecomposable direct summands has a simple top.
From these, we can similarly construct 
a basic $\tau_*$-stable cluster-tilting object 
in $\mod \Lambda_{\Delta,3k,3}$
and count the number of indecomposable direct summands to the case $t=2$.

If $\Lambda$ is type II, it is straightforward to see 
that this $T''$ is also $\theta_*$-stable,
because every indecomposable direct summands of $T''$ 
is the quotient of some indecomposable projective $\Lambda$-module
by a product of ideals of the form $\Lambda(1-e_i)\Lambda$,
and because the automorphism $\theta \colon \Lambda \to \Lambda$
satisfies that $\theta(e_u)=e_u$ for $u \in Q_0$
and that $\theta(\alpha)$ is $\alpha$ or $-\alpha$ for $\alpha \in Q_1$.
\end{proof}

The remained task is on basic $([-2] \circ \bS)$-stable 
cluster-tilting objects 
for the type III: $\Lambda_{A_n,2k-1,2}$ ($n \in 2\Z)$,
and the answer is the following proposition.  

\begin{Th}\label{no_CT_III}
Let $n,k \ge 1$ be integers and assume $n \in 2\Z$.
Then we have 
\begin{align*}
\max \{ |T|_{\mathrm{np}} \mid 
\mbox{$T$ is a $\tau_*\kappa_*$-stable rigid object in $\mod \Lambda_{A_n,2k-1,2}$}\} 
= n(n-2)(2k-1)/4.
\end{align*}
Moreover, there is no $\tau_*\kappa_*$-stable cluster-tilting object 
in $\mod \Lambda_{A_n,2k-1,2}$.
\end{Th}

\begin{Rem}\label{caution_III}
Note that $\tau_*\kappa_*$ commutes with stable equivalences of mesh algebras of type III
by Proposition \ref{l=l'},
and thus Theorem \ref{no_CT_III} gives an invariant under stable equivalences.
\end{Rem}

For the proof of Theorem \ref{no_CT_III}, 
we use some results on $\Lambda_{A_n,1,2}$ from \cite{AS} for $n \in \{4,6,8,\ldots\}$.
In the part (3), \cite[Corollary 5.18]{AS} says $\umod \Lambda_{A_n,1,2}$ that 
2-Calabi--Yau if the characteristic of $K$ is 2,
but in this case, 
we have $[3] \cong \id$ by the part (1), 
so it is also 5-Calabi--Yau.

\begin{Prop}\label{loop_mesh}
Let $n \in \{4,6,8,\ldots\}$.
\begin{itemize}
\item[(1)] \cite[Corollary 5.5]{AS}
We have $[-3] \cong \kappa_*$ on $\umod \Lambda_{A_n,1,2}$.
\item[(2)] \cite[Theorem 5.10]{AS}
We have $[6] \cong \mathrm{id}$ on $\umod \Lambda_{A_n,1,2}$.
\item[(3)] \cite[Corollary 5.18, Theorem 5.19]{AS}
The stable module category $\umod \Lambda_{A_n,1,2}$ is $5$-Calabi--Yau.
\end{itemize}
\end{Prop}

\begin{Lem}\label{loop_3}
Let $n \in \{4,6,8,\ldots\}$ and $Q=Q_{A_n,1,1}$, $\Lambda=\Lambda_{A_n,1,1}$
and $Q'=Q_{A_n,1,2}$, $\Lambda'=\Lambda_{A_n,1,2}$.
Consider the functors $\Phi_\phi \colon \mod \Lambda \to \mod \Lambda'$
and $\Psi_\phi \colon \mod \Lambda' \to \mod \Lambda$.
\begin{itemize}
\item[(1)]
If a $\Lambda'$-module $M'$ is rigid, then $M' \oplus \kappa_*(M')$ is also rigid. 
Especially, every maximal rigid $\Lambda'$-module is $\kappa_*$-stable.
\item[(2)]
For any $\Lambda'$-module $M'$,
then we have a short exact sequence 
$0 \to M' \to \Phi_\phi(\Psi_\phi(M')) \to \kappa_*(M') \to 0$.
\item[(3)]
If a $\Lambda'$-module $M'$ is rigid, 
then we have $\Phi_\phi(\Psi_\phi(M')) \cong M' \oplus \kappa_*(M')$.
\item[(4)]
Let $M'$ be a $\Lambda'$-module with $M' \cong \Phi_\phi(M)$ for
some $\Lambda$-module $M$.
Then we have $M' \cong \kappa_*(M')$ in $\mod \Lambda'$.
\end{itemize}
\end{Lem}

\begin{proof}
(1)
We show that $\Ext_{\Lambda'}^1(M',\kappa_*(M'))=0$ first.
By Proposition \ref{loop_mesh} (3), $\umod \Lambda'$ is 5-Calabi--Yau.
Therefore, we have
$\Ext_{\Lambda'}^1(M',\kappa_*(M')) \cong
\Ext_{\Lambda'}^1(M',M'[3]) \cong \Ext_{\Lambda'}^4(M',M')\cong
D\Ext_{\Lambda'}^1(M',M')=0$.
We have $\Ext_{\Lambda'}^1(M',M')=0$ and $\Ext_{\Lambda'}^1(M',\kappa_*(M'))=0$.
Applying the involutive functor $\kappa_*$ to them, we have 
$\Ext_{\Lambda'}^1(\kappa_*(M'),\kappa_*(M'))=0$ and 
$\Ext_{\Lambda'}^1(\kappa_*(M'),M')=0$.
Thus $M' \oplus \kappa_*(M')$ is rigid.

(2)
By the definition of functors, $\Phi_\phi(\Psi_\phi(M'))$
is the following $\Lambda'$-module $N'$;
\begin{itemize}
\item for each vertex $i \in (Q')_0$, $N'e_i=M'e_i \oplus M'e_i$ holds, and
\item
for each arrow $(\alpha \colon i \to j) \in (Q')_1$,
let $f \colon M'e_i \to M'e_j$ the linear map defined by the action of $\alpha$ on $M'$,
then the action of $\alpha$ on $N'$ is given by
the matrix 
\begin{align*}
\begin{bmatrix} 0 & f \\ f & 0 \end{bmatrix} \ (\mbox{if $\alpha$ is the unique loop}),
\quad
\begin{bmatrix} f & 0 \\ 0 & f \end{bmatrix} \ (\mbox{otherwise}).
\end{align*}
\end{itemize}
We can construct a monomorphism from $M'$ to $N'=\Phi_\phi(\Psi_\phi(M'))$
as $M'e_i$ is embedded diagonally to $N'e_i=M'e_i \oplus M'e_i$.
By straightforward calculations, we can see that 
the cokernel of this monomorphism is isomorphic to $\kappa_*(M')$.

(3)
It is easily seen by (1) and (2).

(4)
If $M'$ is indecomposable projective, 
then $M' \cong \kappa_*(M')$ in $\mod \Lambda$ is easy to see.
We may assume that $M'$ is indecomposable and not projective.
By \cite[Corollary 5.5]{AS},
the functor $[-3]$ in $\umod \Lambda$ is given 
by the automorphism $\phi \colon \Lambda \to \Lambda$
coming from the quiver automorphism $\phi \colon Q \to Q$,
and $[3]$ satisfies the same property.
We have $M[3] \cong \phi_*(M)$ in $\umod \Lambda$ and
the assertion is proved as
$\kappa_*(M') \cong M'[3] \cong \Phi_\phi(M)[3] 
\cong \Phi_\phi(M[3]) \cong \Phi_\phi(\phi_*(M)) \cong M'$
in $\umod \Lambda$.
Because the dimensions of $M'$ and $\kappa_*(M')$ coincide and 
$M'$ is indecomposable and not projective, we have $M' \cong \kappa_*(M')$ 
in $\mod \Lambda$.
\end{proof}

The following proposition gives a way to obtain a maximal rigid object in 
$\mod \Lambda_{A_n,1,2}$. 

\begin{Prop}\label{preproj_loop}
Let $n \in \{4,6,8,\ldots\}$ and $\Lambda=\Lambda_{A_n,1,1}$
and $\Lambda'=\Lambda_{A_n,1,2}$.
We consider the functors $\Phi_\phi \colon \mod \Lambda \to \mod \Lambda'$
and $\Psi_\phi \colon \mod \Lambda' \to \mod \Lambda$.

Let $T$ be a maximal $\phi_*$-stable rigid object in $\mod \Lambda$
with $T \cong U \oplus \phi_*(U)$ for some $U$.
Then $T':=\Phi_\phi(U)$ is a maximal rigid object in $\mod \Lambda'$
and is $\kappa_*$-stable.
\end{Prop}  

\begin{proof}
Assume that $M'$ in $\mod \Lambda'$ satisfies that $T' \oplus M'$ is rigid.
Apply the functor $\kappa_*$, then 
$\kappa_*(T') \oplus \kappa_*(M') \cong T' \oplus \kappa_*(M')$ is rigid,
where $\kappa_*(T') \cong T'$ holds in $\mod \Lambda'$ by Lemma \ref{loop_3} (4).
By Lemma \ref{loop_3} (1), $T' \oplus M' \oplus \kappa_*(M')$ is also rigid,
and by Lemma \ref{loop_3} (3), 
it is isomorphic to $\Phi_\phi(U \oplus \Psi_\phi(M'))$.
Applying $\Psi_\phi$, 
we have a rigid object
$U \oplus \phi_*(U) \oplus \Psi_\phi(M')^2 \cong T \oplus \Psi_\phi(M')^2$ 
in $\mod \Lambda$ by Lemma \ref{orbit_mult} (2).
Because $T$ is maximal rigid in $\mod \Lambda$, 
$\Psi_\phi(M')$ must be in $\add_{\mod \Lambda} T$.
Apply the functor $\Phi_\phi$,
then $\Phi_\phi(\Psi_\phi(M')) \cong M' \oplus \kappa_*(M')$ is in 
$\add_{\mod \Lambda'} \Phi_\phi(T)
=\add_{\mod \Lambda'} \Phi_\phi(U)=\add_{\mod \Lambda'} T'$.
Therefore, $M'$ must be in $\add_{\Lambda'}T'$,
and the assertion is proved.
The $\kappa_*$-stableness is deduced by Lemma \ref{loop_3} (1).
\end{proof}

We also recall some results on 2-Calabi--Yau triangulated categories.
It is well-known that $\umod \Lambda_{A_n,1,1}$ is 2-Calabi--Yau.
In the part (2), if $T_1$ and $T_2$ are basic cluster-tilting objects in $\umod \Lambda$
and satisfy $T_1 \cong U \oplus V_1$ and $T_2 \cong U \oplus V_2$ 
with $V_1 \not \cong V_2$ indecomposable,
then we say that $T_2$ is the \textit{mutation} of $T_1$ at $V_1$.

\begin{Prop}\label{2-CY_prop}
Let $\Lambda$ be a finite-dimensional self-injective $K$-algebra
with $\umod \Lambda$ is 2-Calabi--Yau.
\begin{itemize}
\item[(1)]
\cite[Theorem 4.7, Theorem 4.9]{IY}
Let $V$ be a rigid object in $\umod \Lambda$.
We define two full subcategories $\cD \subset \cZ \subset \umod \Lambda$ as 
$\cD=\add_{\umod \Lambda} V$, $\cZ=\{ X \in \umod \Lambda \mid \Ext_\Lambda^1(V,X)=0 \}$
and $[\cD]$ as the ideal of $\umod \Lambda$ consisting of all morphisms
factoring through some object in $\cD$.
Then the category $\cZ/[\cD]$ has 
a natural structure of a 2-Calabi--Yau triangulated category
and the natural triangulated functor $\cZ \to \cZ/[\cD]$ gives 
one-to-one correspondence between 
the basic cluster-tilting (resp. rigid) objects of $\umod \Lambda$ containing $V$ and 
the basic cluster-tilting (resp. rigid) objects of $\cZ/[\cD]$.
\item[(2)]\cite[Corollary 4.9]{AIR}
If there exists a finite set of cluster-tilting objects in $\umod \Lambda$ closed 
under mutations, then the finite set contains 
all cluster-tilting objects in $\umod \Lambda$.
\item[(3)]\cite[Corollary 3.7]{ZZ}
If $T_1$ and $T_2$ are maximal rigid objects in $\umod \Lambda$,
we have $|T_1|_{\mathrm{np}}=|T_2|_{\mathrm{np}}$.
\item[(4)]\cite[Theorem 2.6]{ZZ}
If there exists a cluster-tilting object in $\umod \Lambda$, then 
any maximal rigid object is cluster-tilting.
\end{itemize}
\end{Prop}

Though $\umod \Lambda_{A_n,1,2}$ is not necessarily 2-Calabi--Yau,
we can show the following.

\begin{Prop}\label{2-CY_prop_loop}
Proposition \ref{2-CY_prop} (3) and (4) also hold 
even if $\Lambda=\Lambda_{A_n,1,2}$ with $n \in 2\Z$.
\end{Prop}

\begin{proof}
We may assume that there exists a maximal rigid object $T$ in $\umod \Lambda$.

By Lemma \ref{loop_3} (1),
every maximal rigid object in $\umod \Lambda_{A_n,1,2}$ is $\kappa_*$-stable,
or equivalently, $[3]$-stable.
With Proposition \ref{loop_mesh} (2) and (3),
we can show \cite[Corollary 2.5]{ZZ} similarly in this situation;
namely, every rigid object belongs to the full subcategory 
$\mathcal{C} \subset \umod \Lambda$, 
where $\mathcal{C}$ consists of the objects $M$ satisfing that
there exists a triangle of the form
$T' \to M \to T''[1] \to T'[1]$ with $T',T'' \in \add_{\umod \Lambda} T$ 
in $\umod \Lambda$.

By this property, we can prove the functor 
$\Hom_{\umod \Lambda}(T,?) \colon \mathcal{C}/[\add_{\umod \Lambda} T[1]] \to 
\mod E$ is an equivalence
similarly to \cite[Subsection 2.1]{KR},
where $[\add_{\umod \Lambda} T[1]]$ 
is the ideal of $\umod \Lambda$ consisting of all morphisms
factoring through some object in $\add_{\umod \Lambda} T[1]$,
and $E=\End_{\umod \Lambda}(T)$.
Now, the proof and the result of \cite[Theorem 4.1]{AIR} are valid in this situation.

Let $T_1$ be another maximal rigid object in $\umod \Lambda$,
then $\Hom_{\umod \Lambda}(T,T_1)$ is a support $\tau$-tilting object in $\mod E$
by \cite[Theorem 4.1]{AIR}, 
and we have $|T_1|_{\mathrm{np}} \le |T|_{\mathrm{np}}$.
By exchanging the roles of $T$ and $T_1$, 
we also have $|T|_{\mathrm{np}} \le |T_1|_{\mathrm{np}}$.
Thus $|T_1|_{\mathrm{np}} = |T|_{\mathrm{np}}$ is proved,
and this is the statement of Proposition \ref{2-CY_prop} (3).

We additionally assume that $T$ is cluster-tilting. 
Then we have $\mathcal{C} = \umod \Lambda$,
see \cite[Subsection 2.1]{KR}.
In this case, \cite[Theorem 4.1]{AIR} implies that 
every maximal rigid object in $\umod \Lambda$ is cluster-tilting.
The statement of Proposition \ref{2-CY_prop} (4) is true.
\end{proof}

Now, we begin the proof of Theorem \ref{no_CT_III}.

\begin{proof}[Proof of Theorem \ref{no_CT_III}]
If $n=2$, then it is easy to see that 
every rigid object in $\mod \Lambda_{A_n,1,2}$ is projective.
Thus the assertion is clear.
From now on, we assume $n \ge 4$.

First, we prove the case $k=1$.
The translation $\tau \colon Q_{A_n,1,2} \to Q_{A_n,1,2}$ is identity.
By Lemma \ref{loop_3},
every maximal rigid object $M$ in $\umod \Lambda_{A_n,1,2}$ is $\kappa_*$-stable.
Thus by Proposition \ref{2-CY_prop_loop}, 
it is enough to show that $|T'|_{\mathrm{np}}=n(n-2)/4$ 
for some maximal rigid object $T'$ in $\mod \Lambda_{A_n,1,2}$
which is not cluster-tilting.

We show that there exists a maximal $\phi_*$-stable rigid object $V$ 
in $\mod \Lambda_{A_n,1,1}$ with $|V|_{\mathrm{np}}=n(n-2)/2$. 

We define a sequence $\gamma^\epsilon$ 
for $\epsilon=(\epsilon_1,\ldots,\epsilon_{n/2}) \in \{ \pm \}^{n/2}$ 
as follows;
\begin{align*}
& \alpha_m=(n/2-(m-1),(n+2)/2+(m-1)) \quad (m=2,\ldots,n/2), \\
& \beta_1^+=(n/2,(n+2)/2,n/2), \quad \beta_1^-=((n+2)/2,n/2,(n+2)/2), \\
& \beta_m^\pm = \alpha_m \cdot \beta_{m-1}^\pm \cdot \alpha_m \quad (m=2,\ldots,n/2), \quad
\gamma^\epsilon
=\beta_1^{\epsilon_1}\cdot\beta_2^{\epsilon_2}\cdots\beta_{n/2}^{\epsilon_{n/2}}.
\end{align*}
We write $T^\epsilon$ for the corresponding basic cluster-tilting object
in $\mod \Lambda_{A_n,1,1}$ for the sequence $\gamma^\epsilon$
constructed in Proposition \ref{long_elem}.

Fix $m \in \{1,\ldots,n/2\}$.
Suppose $\epsilon,\epsilon' \in \{ \pm \}^{n/2}$ satisfy that
$\epsilon_m=1$, $\epsilon'_m=-1$, and $\epsilon_i = \epsilon'_i$ for $i \ne m$.
Then we have the following things;
\begin{itemize}
\item $T^\epsilon$ and $T^{\epsilon'}$ are different 
in exactly one indecomposable direct summand;
\item The unique indecomposable module 
that is a direct summand of $T^\epsilon$ and not of $T^{\epsilon'}$
depends on only $m$, not on the choice of $\epsilon$ and $\epsilon'$,
which is denoted by $U_m^+$;
\item The unique indecomposable module 
that is a direct summand of $T^{\epsilon'}$ and not of $T^\epsilon$
depends on only $m$, not on the choice of $\epsilon$ and $\epsilon'$,
which is denoted by $U_m^-$;
\item The Loewy lengths of $U_m^+$ and $U_m^-$ are $2m-1$, and 
$\phi_*(U_m^+)=U_m^- \not \cong U_m^+$.
\end{itemize}
Thus, there exists a unique basic rigid $\Lambda_{A_n,1,1}$-module $V$
such that, for all $\epsilon \in \{ \pm \}^{n/2}$,
$T^\epsilon=U_1^{\epsilon_1} \oplus \cdots \oplus U_{n/2}^{\epsilon_{n/2}} \oplus V$.
We can deduce that $V$ is $\phi_*$-stable and containing $\Lambda_{A_n,1,1}$
as a direct summand.

Now we show $V$ is a maximal $\phi_*$-stable rigid object 
in $\umod \Lambda_{A_n,1,1}$.
It is well-known that $\umod \Lambda_{A_n,1,1}$ is 2-Calabi--Yau.
Let $\cD \subset \cZ \subset \umod \Lambda_{A_n,1,1}$ 
as in Proposition \ref{2-CY_prop} (1)
and $F \colon \mathcal{Z} \to \mathcal{Z}/[\mathcal{D}]$
be the canonical functor.
Because $V$ is $\phi_*$-stable, $\phi_*$ also acts on $\cZ/[\cD]$.

For $\epsilon \in \{ \pm \}^{n/2}$, it is easy to see that
$F(U_1^{\epsilon_1} \oplus \cdots \oplus U_{n/2}^{\epsilon_{n/2}})$ 
does not contain any nonzero $\phi_*$-stable direct summand.
We can also deduce that a finite set 
$\{ F(U_1^{\epsilon_1} \oplus \cdots \oplus U_{n/2}^{\epsilon_{n/2}}) 
\mid \epsilon \in \{ \pm \}^{n/2} \}$
of cluster-tilting objects in $\cZ/[\cD]$ is closed under mutations,
and thus all cluster-tilting objects in $\cZ/[\cD]$
are contained in this finite set by Proposition \ref{2-CY_prop} (2).
Therefore, there is no cluster-tilting object in $\cZ/[\cD]$
containing a nonzero $\phi_*$-stable direct summand.
This implies that $V$ is a basic maximal $\phi_*$-stable rigid object 
in $\umod \Lambda_{A_n,1,1}$ and $\mod \Lambda_{A_n,1,1}$.
It is clear that $|V|_{\mathrm{np}}=n(n-2)/2$.

We can take some $V_1$ such that $V=V_1 \oplus \phi_*(V_1)$,
because each indecomposable direct summands of $V$ has a simple top.
By Proposition \ref{preproj_loop}, 
$T':=\Phi_\phi(V_1)$ is a maximal rigid object in $\mod \Lambda_{A_n,1,2}$.
Because $V$ is basic, we have $|T'|_{\mathrm{np}}=|V|_{\mathrm{np}}/2=n(n-2)/4$
by Lemma \ref{orbit_mult} (1).
If this $T'$ is cluster-tilting,
$\Psi_\phi(T') \cong V$ is a cluster-tilting object in $\mod \Lambda_{A_n,1,1}$
by Lemma \ref{orbit_mult} (2), but it is a contradiction.

Now, we have shown that 
$|T'|_{\mathrm{np}}=n(n-2)/4$ 
for any maximal $\kappa_*$-stable rigid object $T'$ in $\mod \Lambda_{A_n,1,2}$,
and that $\mod \Lambda_{A_n,1,2}$ has no $\kappa_*$-stable cluster-tilting object.
The proof for the case $k=1$ is completed.

Now, let $k \ge 1$ be general.
It is easy to see that
$Q_{A_n,1,2}=Q_{A_n,2k-1,2}/\langle \phi \rangle$ 
with $\phi \in \Aut Q_{A_n,2k-1,2}$ free, and
that $\tau_*\kappa_*$-stableness in $\mod \Lambda_{A_n,2k-1,2}$ 
implies $\phi_*$-stableness in $\mod \Lambda_{A_n,2k-1,2}$.
 
Let $T$ be a $\phi_*$-stable rigid object in $\mod \Lambda_{A_n,2k-1,2}$.
We prove that $|T|_{\mathrm{np}} \le n(n-2)(2k-1)/4$.
We may assume that there exists an object $U$ in $\mod \Lambda_{A_n,2k-1,2}$
such that $T=\bigoplus_{i=0}^{2k-2} \phi_*^i(U)$.
By Lemma \ref{orbit_mult} (1),
$T'=\Phi_\phi(U)$ is a rigid object in $\mod \Lambda_{A_n,1,2}$,
where $\Phi_\phi \colon \mod \Lambda_{A_n,2k-1,2} 
\to \mod \Lambda_{A_n,1,2}$.
We have 
$|T|_{\mathrm{np}} \le (2k-1)|T'|_{\mathrm{np}} \le n(n-2)(2k-1)/4$ 
by Lemma \ref{orbit_mult} (1) and the first statement for $k=1$.
We have seen that $T'$ is not cluster-tilting
and so $T$ cannot be cluster-tilting by Lemma \ref{orbit_mult} (1). 

On the other hand, we have shown that
there exists a basic $\phi_*$-stable rigid object $V$ in $\mod \Lambda_{A_n,1,1}$ with
$|V|_{\mathrm{np}} = n(n-2)/2$
as a direct summand of a cluster-tilting object in $\mod \Lambda_{A_n,1,1}$
obtained by Proposition \ref{long_elem}.
Similarly to the proof for the case $t=2$ in Theorem \ref{CT_not_III},
using Lemmas \ref{orbit_mult} (2) and \ref{grade},
we can construct a basic $\tau_*$-stable rigid object $\Phi_{\tau^k \psi}(\Psi_\tau(V))$ 
in $\mod \Lambda_{A_n,2k-1,2}$
with $|\Phi_{\tau^k \psi}(\Psi_\tau(V))|_{\mathrm{np}}=n(n-2)(2k-1)/4$,
where $\Phi_{\tau^k \psi} \colon \mod \Lambda_{A_n,2k-1,1} \to \mod \Lambda_{A_n,2k-1,2}$.
We can straightforwardly check this rigid object is $\tau_*\kappa_*$-stable.
The proof is completed.
\end{proof}

Theorems \ref{CT_not_III} and \ref{no_CT_III} imply the following.

\begin{Cor}\label{CT_iff}
Let $\Lambda=\Lambda_{\Delta,l,t}$ be a finite-dimensional mesh algebra.
Then $\mod \Lambda$ has a $([-2] \circ \bS)$-stable cluster-tilting object if and only if 
$\Lambda$ is not type III.
\end{Cor}
\section{Proof of Theorem \ref{not_stable_eq}}
 
The aim of this section is to prove Theorem \ref{not_stable_eq}.
As in the previous section, $K$ is supposed to be an algebraically closed field.
For simplicity, we call a triangle equivalence between stable module categories
a \textit{stable equivalence}.
First, we deduce the part (2) from the part (1) of Theorem \ref{not_stable_eq}.

\begin{proof}[Proof of $(1) \Rightarrow (2)$]
Assume that $\Lambda$ and $\Lambda'$ are derived equivalent.
Then they are stable equivalent \cite[Corollary 2.2]{Rickard-DS}.
From (1), it remains to show that $Q \cong Q'$ holds also in the case $\Delta=\Delta'=A_1$.
We can write $\rho=\tau^k$ and $\rho'=\tau^{k'}$ because $\psi=\id$.
In this case, 
$\Z^k \cong K_0(D^\rb(\mod \Lambda)) \cong K_0(D^\rb(\mod \Lambda')) \cong \Z^{k'}$
holds. 
We have $k=k'$ and thus $Q=\Z A_1/\langle \tau^k \rangle
=\Z A_1/\langle \tau^{k'} \rangle= Q'$.
\end{proof}

Now we begin the proof of Theorem \ref{not_stable_eq} (1).
For each mesh algebra, we have defined its \textit{type} I,\ldots,X in Definition 
\ref{quiver_def}.
We may exclude the case $\Delta=A_1$,
because $\umod \Lambda_{\Delta,l,t}$ is a zero category
if and only if $\Delta=A_1$.
We first use three kinds of invariants of mesh algebras under stable equivalences.
The values of these invariants are written in Table \ref{type}.

\begin{itemize}
\item[(a)]
The maximal number of pairwise nonisomorphic indecomposable nonprojective direct summands 
of a $([-2] \circ \bS)$-stable rigid object in $\mod \Lambda_{\Delta,l,t}$
(Theorems \ref{CT_not_III}, \ref{no_CT_III} and Remark \ref{caution_III}).
\item[(b)]
The order of the autoequivalence $[-2] \circ \bS$ on $\umod \Lambda_{\Delta,l,t}$
(Proposition \ref{l=l'}).
\item[(c)]
The quotient (a)/(b).
\end{itemize}

\begin{table}
\begin{tabular}{cl|ccc}
type & $(\Delta,l   ,t)$                  & (a)              & (b)    & (c)        \\
\hline
I    & $(A_n,   k,   1)$                  & $n(n-1)k/2$      & $k$    & $n(n-1)/2$ \\
II   & $(A_n,   2k,  2)$ ($n \notin 2\Z$) & $n(n-1)k/2$      & $2k$   & $n(n-1)/4$ \\
III  & $(A_n,   2k-1,2)$ ($n \in    2\Z$) & $n(n-2)(2k-1)/4$ & $2k-1$ or $4k-2$ &
$n(n-2)/4$ or $n(n-2)/8$ \\
IV   & $(D_n,   k,   1)$                  & $n(n-2)k$        & $k$    & $n(n-2)$   \\
V    & $(D_n,   2k,  2)$                  & $n(n-2)k$        & $2k$   & $n(n-2)/2$ \\
VI   & $(D_4,   3k,  3)$                  & $8k$             & $3k$   & $8/3$      \\
VII  & $(E_6,   k,   1)$                  & $30k$            & $k$    & $30$       \\
VIII & $(E_6,   2k,  2)$                  & $30k$            & $2k$   & $15$       \\
IX   & $(E_7,   k,   1)$                  & $56k$            & $k$    & $56$       \\
X    & $(E_8,   k,   1)$                  & $112k$           & $k$    & $112$      
\end{tabular}
\caption{The types and the invariants}
\label{type}
\end{table}

The following proposition is easy.

\begin{Prop}\label{same_type}
Assume that $\umod \Lambda_{\Delta,l,t} \cong \umod \Lambda_{\Delta',l',t'}$
with $\Delta, \Delta' \ne A_1$.
If $(\Delta,l,t)$ and $(\Delta',l',t')$ are the same type,
then we have $(\Delta,l,t)=(\Delta',l',t')$.
\end{Prop}

\begin{proof}
The values of (a), (b) and (c) determine $(\Delta,l,t)$.
\end{proof}

We will show the following proposition.
This and Proposition \ref{same_type} imply Theorem \ref{not_stable_eq}.

\begin{Prop}\label{detect}
Assume that $\umod \Lambda_{\Delta,l,t} \cong \umod \Lambda_{\Delta',l',t'}$
with $\Delta, \Delta' \ne A_1$.
Then $(\Delta,l,t)$ and $(\Delta',l',t')$ are the same type.
\end{Prop}

We first prove that Proposition \ref{detect} holds 
if one of two mesh algebras are type III or VI--X.

\begin{Lem}\label{detect_1}
Assume that $\umod \Lambda_{\Delta,l,t} \cong \umod \Lambda_{\Delta',l',t'}$
with $\Delta, \Delta' \ne A_1$.
If $(\Delta',l',t')$ is type III or VI--X, 
then $(\Delta,l,t)$ is the same type as $(\Delta',l',t')$.
\end{Lem}

\begin{proof}
By Proposition \ref{l=l'} (1),
existence of a $([-2] \circ \bS)$-stable cluster-tilting object is 
invariant under stable equivalences.
By Corollary \ref{CT_iff}, 
if $(\Delta',l',t')$ is type III, then $(\Delta,l,t)$ is type III.

Let $(\Delta',l',t')$ be type VI--X,
and assume $(\Delta,l,t)$ is not the same type as $(\Delta',l',t')$.
Comparing the values of (c), there are only three kinds of possibilities;
(i) $(\Delta,l,t)=(A_6,k,1)$ and $(\Delta',l',t')=(E_6,2k',2)$,
(ii) $(\Delta,l,t)=(D_5,k,1)$ and $(\Delta',l',t')=(E_6,2k',2)$,
(iii) $(\Delta,l,t)=(D_{16},2k,2)$ and $(\Delta',l',t')=(E_8,k',1)$.
From Theorem \ref{main},
it is straightforward to show that the Grothendieck groups of the 
stable categories do not coincide
in the possibilities (i)--(iii).
\end{proof}

We next prove that Proposition \ref{detect} holds 
if $\Delta'=A_2,A_3$.

\begin{Lem}\label{detect_2_1}
Assume that $\umod \Lambda_{\Delta,l,t} \cong \umod \Lambda_{\Delta',l',t'}$
with $\Delta, \Delta' \ne A_1$.
If $\Delta'$ is $A_2$ or $A_3$, then $(\Delta,l,t)$ is the same type as 
$(\Delta',l',t')$.
\end{Lem}

\begin{proof}
If $\Delta'=A_2$ and $t'=2$ (type III), 
the assertion is proved by Lemma \ref{detect_1}.

The remained cases are (i) $\Delta'=A_2$ and $t=1$ (type I), 
(ii) $\Delta'=A_3$ and $t=1$ (type I),
(iii) $\Delta'=A_3$ and $t=2$ (type II).
In these cases, the assertion is obtained by comparing the values of (c).
\end{proof}

Now, we only have to consider the types I, II, IV, V
with $\Delta \ne A_1,A_2,A_3$.

We can also use the order of the shift 
$[1] \colon \umod \Lambda_{\Delta,l,t} \to \umod \Lambda_{\Delta,l,t}$ as
an invariant. 
The following result follows from results in \cite{AS}.

\begin{Prop}\label{shift_order}
Let $\Lambda_{\Delta,l,t}$ be a finite-dimensional mesh algebra with 
$\Delta \ne A_1,A_2,A_3$,
and $p$ be the characteristic of the field $K$
and $\pi$ be the Nakayama permutation on $Q_{\Delta,l,t}$. 
Put $u$ is the order of $\pi \tau^{-1} \in \Aut Q_{\Delta,l,t}$.
Then the minimal integer $i \ge 1$ such that  
$[i] \colon \umod \Lambda_{\Delta,l,t} \to \umod \Lambda_{\Delta,l,t}$ 
is isomorphic to the identity functor (as additive functors) is
$3u$ if $p=2$, and $\operatorname{lcm}(3u,2)$ if $p \ne 2$. 
\end{Prop}

\begin{proof}
For a simple $\Lambda$-module $S$, the condition $S[i] \cong S$ 
in $\umod \Lambda$ implies $i \in 3\Z$
(see Proposition \ref{proj_resol}), due to $\Delta \ne A_1,A_2$.
Thus $[i]$ is not isomorphic to the identity on $\umod \Lambda$ if $i \notin 3\Z$.
Let $i \in 3\Z$.
Proposition \ref{proj_resol} (1) means that 
the 3rd syzygy of $\Lambda_{\Delta,l,t}$ as a $\Lambda$-$\Lambda$-bimodule
is isomorphic to a twisted bimodule ${_1(\Lambda_{\Delta,l,t})_\mu}$,
where $\mu$ is a $K$-algebra automorphism on $\Lambda_{\Delta,l,t}$.
By the assumption $\Delta \ne A_1,A_2,A_3$ and \cite[Lemma 5.11]{AS},
$[i]=(? \otimes _1(\Lambda_{\Delta,l,t})_{\mu^{i/3}})$ is
isomorphic to the identity functor on $\umod \Lambda$
if and only if 
$(? \otimes _1(\Lambda_{\Delta,l,t})_{\mu^{i/3}})$ is
isomorphic to the identity functor on $\mod \Lambda$. 
From \cite[Theorem 5.10]{AS}, the minimal such $i$ satisfying the latter condition 
is $3u$ if $p=2$, and $\operatorname{lcm}(3u,2)$ if $p \ne 2$.
The assertion is proved. 
\end{proof}

Let $(\Delta,l,t)$ be type I, II, IV, or V.
Moreover, let $c$ be the Coxeter number of $\Delta$, 
$k=l/t$, $d=\gcd(c,k)$, $r=c/d$, and $q=k/d$. 
We divide types I, II, IV, V into subtypes shown in Table \ref{subtype}.
We consider the following invariants in Table \ref{subtype}.

\begin{itemize}
\item[(d)] The Grothendieck group $K_0(\umod \Lambda_{\Delta,l,t})$.  
The columns 
``$\Z$'', ``$\Z/2\Z$'', and ``other'' indicate the multiplicity of $\Z$, $\Z/2\Z$,
and the other direct summands of $K_0(\umod \Lambda_{\Delta,l,t})$.
Here, the values of the nonempty cells are positive
(Theorem \ref{main}).
\item[(e)] The order of the shift $[1]$ on $\umod \Lambda_{\Delta,l,t}$
up to functorial isomorphisms as additive functors
(Proposition \ref{shift_order}).
\item[(f)] The quotient (a)/(e).
\end{itemize}  

\begin{table}
\begin{tabular}{lcl|ccc|cc}
$(\Delta,l,t)$
&subtype & condition                        & $\Z$          & $\Z/2\Z$   & other    &
(e)         & (f)         \\
\hline
$(A_n,k,1)$
&I-1     & $r\in 2\Z$, $d=1$                & $(nd-3d+2)/2$ &            &          &
$6q$        & $d$         \\
($n \ge 4$)
&I-2     & $r\in 2\Z$, $d\ne 1$             & $(nd-3d+2)/2$ & $d-1$      &          &
$6q$        & $d$         \\
&I-3     & $r\notin 2\Z$                    & $(nd-2d+2)/2$ &            &          &
$6q$        & $d$         \\
\hline
$(A_n,2k,2)$
&II-1    & $r\in 4\Z$, $d=1$                & $(nd-3d)/2$   &            & $\Z/4\Z$ &
$6q$        & $2d$        \\
($n=5,7,\ldots$)
&II-2    & $r\in 4\Z$, $d\ne 1$             & $(nd-3d)/2$   & $d-1$      & $\Z/4\Z$ &
$6q$        & $2d$        \\
&II-3    & $r\in 2+4\Z$                     &               & $nd-2d+1$  &          &
$6q$ ($3q$) & $2d$ ($4d$) \\
&II-4    & $r\notin 4\Z$                    & $(nd-d)/4$    &            &          &
$12q$       & $d$         \\
\hline
$(D_n,k,1)$
&IV-1    & $k\in 2\Z$, $r=2$                 & $d-1$         & $nd-3d+1$  &          &
$6q$        & $d$         \\
&IV-2    & $k\in 2\Z$, $r=4,6,\ldots$        & $d-1$         & $nd-3d$    & $\Z/r\Z$ &
$6q$        & $d$         \\
&IV-3    & $k\in 2\Z$, $r=1$                 & $(nd-d-2)/2$  &            &          &
$6q$        & $d$         \\
&IV-4    & $k\in 2\Z$, $r=3,5,\ldots$        & $(nd-d-2)/2$  &            & $\Z/r\Z$ &
$6q$        & $d$         \\
&IV-5    & $k\notin 2\Z$, $r\in 4\Z$         & $d$           & $nd-3d$    &          &
$6q$        & $d$         \\
&IV-6    & $k\notin 2\Z$, $r\notin 4\Z$      &               & $nd-d-1$   &          &
$6q$ ($3q$) & $d$ ($2d$)  \\
\hline
$(D_n,2k,2)$
&V-1     & $k\in 2\Z$, $r\in 4\Z$            & $d$           & $nd-3d$    &          &
$6q$        & $2d$        \\
&V-2     & $k\in 2\Z$, $r\in 2+4\Z$          &               & $nd-d-1$   &          &
$6q$ ($3q$) & $2d$ ($4d$) \\
&V-3     & $k\in 2\Z$, $r\notin 2\Z$         & $(nd-2d)/2$   &            &          &
$12q$       & $d$         \\
&V-4     & $k\notin 2\Z$, $r=2$              & $d-1$         & $nd-3d+1$  &          &
$6q$        & $2d$        \\
&V-5     & $k\notin 2\Z$, $r\ne 2$, $d\ne 1$ & $d-1$         & $nd-3d$    & $\Z/r\Z$ &
$6q$        & $2d$        \\
&V-6     & $k\notin 2\Z$, $d=1$              &               & $nd-3d$    & $\Z/r\Z$ &
$6q$        & $2d$        
\end{tabular}
\caption{The subtypes and the invariants}
\label{subtype}
\end{table}

The invariants (e) and (f) sometimes depend on the characteristic of $K$.
In fact if $K$ has characteristic 2, then these invariants are written inside of
parentheses.

The remaining cases in our proof of Proposition \ref{detect}
are shown by the following lemma.

\begin{Lem}\label{detect_2_2}
Assume that $\umod \Lambda_{\Delta,l,t} \cong \umod \Lambda_{\Delta',l',t'}$
with $\Delta,\Delta' \notin \{A_1,A_2,A_3\}$.
If $(\Delta',l',t')$ is type I, II, IV, or V,
then $(\Delta,l,t)$ is the same type as $(\Delta',l',t')$. 
\end{Lem}

\begin{proof}
Let $c$ be the Coxeter number of $\Delta$, 
$k=l/t$, $d=\gcd(c,k)$, $r=c/d$, and
$c'$ be the Coxeter number of $\Delta'$, 
$k'=l'/t'$, $d'=\gcd(c',k')$, $r'=c'/d'$.

(1) If $(\Delta,l,t)$ is type I and $(\Delta',l',t')$ is type II,
comparing (d), there are only two possibilities.

(1-1: I-1 and II-4)
We have $1=d=d'$ from (f), and substituting it for (d), 
we have $(n-1)/2=(n'-1)/4$ and thus $n'=2n-1$.
Substituting it for (c), we can deduce $n(n-1)/2=(2n-1)(2n-2)/4$, 
but there exists no such $n \ge 4$.

(1-2: I-3 and II-4)
We have $d=d'$ from (f), and substituting it for (d), 
we have $(nd-2d+2)/2=(n'd-d)/4$ and thus $d(n'-2n+3)=4$.
Because $n' \notin 2\Z$, we have $(d,n')=(2,2n-1),(1,2n+1)$.
If $(d,n')=(2,2n-1)$, we can deduce $n(n-1)/2=(2n-1)(2n-2)/4$ from (c),
but there exists no such $n \ge 4$.
If $(d,n')=(1,2n+1)$, we can deduce $n(n-1)/2=(2n+1)(2n)/4$ from (c),
but there exists no such $n \ge 4$.

(2) If $(\Delta,l,t)$ is type I and $(\Delta',l',t')$ is type IV,
comparing (d), there are only four possibilities.

(2-1: I-1 and IV-3)
We have $1=d=d'=2n'-2$ from (f), a contradiction.

(2-2: I-2 and IV-1)
We have $d=d'$ from (f), and substituting it for (d), 
we have $d-1=n'd-3d+1$ and thus $d(n'-4)=-2$.
It is a contradiction because $n' \ge 4$.

(2-3: I-2 and IV-5)
We have $d=d'$ from (f), and substituting it for (d), 
we have $d-1=n'd-3d$ and thus $d(n'-4)=-1$.
It is a contradiction because $n' \ge 4$.

(2-4: I-3 and IV-3)
We have $d=d'=2n'-2$ from (f), and substituting $d=d'$ for (d), 
we have $(nd-2d+2)/2=(n'd-d-2)/2$ and thus $d(n'-n+1)=4$.
It is a contradiction because $d=d'=2n'-2 \ge 6$.

(3) If $(\Delta,l,t)$ is type I and $(\Delta',l',t')$ is type V,
comparing (c), we have $n(n-1)/2=n'(n'-2)/2$.
It is easy to see that $(n-1)^2<n(n-1)<n^2$ and that $(n'-2)^2<n'(n'-2)<(n'-1)^2$,
thus it is necessary that $n=n'-1$.
Substituting it for (c), we have $(n'-1)(n'-2)=n'(n'-2)$.
It is a contradiction because $n' \ge 4$.

(4) If $(\Delta,l,t)$ is type II and $(\Delta',l',t')$ is type IV,
comparing (c), we have $n(n-1)/4=n'(n'-2)$.
It is easy to see that $(n-1)^2<n(n-1)<n^2$ and that $(2n'-3)^2<4n'(n'-2)<(2n'-2)^2$ 
because $n' \ge 4$,
thus it is necessary that $n=2n'-2$.
Substituting it for (c), we have $(2n'-2)(2n'-3)/4=n'(n'-2)$.
It is a contradiction because $n' \ge 4$.

(5) If $(\Delta,l,t)$ is type II and $(\Delta',l',t')$ is type V,
comparing (d), there are only three possibilities.

(5-1: II-2 and V-5 with $r'=4$)
We have $d=d'$ from (f), and substituting it for (d), 
we have $d-1=n'd-3d$ and thus $d(n'-4)=-1$.
It is a contradiction because $n' \ge 4$.

(5-2: II-3 and V-2)
We have $d=d' \in 2\Z$ from (f), and substituting it for (d), 
we have $nd-2d+1=n'd-d-1$ and thus $d(n'-n+1)=2$.
Because $d \in 2\Z$, we have $d=2$ and $n'=n$.
Substituting it for (c), we can deduce $n(n-1)/4=n(n-2)/2$.
It is a contradiction because $n=n' \ge 4$.

(5-3: II-4 and V-3)
We have $d=d'$ from (f), and substituting it for (d), 
we have $(nd-d)/4=(n'd-2d)/2$ and thus $n=2n'-3$.
Substituting it for (c), we can deduce $(2n'-3)(2n'-4)/4=n'(n'-2)/2$.
It is a contradiction because $n' \ge 4$.

(6) If $(\Delta,l,t)$ is type IV and $(\Delta',l',t')$ is type V,
comparing (d), there are only seven possibilities.

(6-1: IV-1 and V-1)
We have $d=2d'$ from (f), and substituting it for (d), 
we have $2d'-1=d'$, a contradiction because $d' \in 2\Z$.

(6-2: IV-1 and V-4)
We have $d=2d'$ from (f), and substituting it for (d), 
we have $2d'-1=d'-1$, a contradiction.

(6-3: IV-2 and V-5)
We have $d=2d'$ from (f), and substituting it for (d), 
we have $2d'-1=d'-1$, a contradiction.

(6-4: IV-3 and V-3)
We have $2n-2=d=d'$ from (f), and substituting $d=d'$ for (d), 
we have $(nd-d-2)/2=(n'd-2d)/2$ and thus $d(n'-n-1)=-2$.
It is a contradiction because $d=2n-2 \ge 6$.

(6-5: IV-5 and V-1)
We have $d=2d'$ from (f), and substituting it for (d), 
we have $2d'=d'$, a contradiction.

(6-6: IV-5 and V-4)
We have $d=2d'$ from (f), and substituting it for (d), 
we have $2d'=d'-1$, a contradiction.

(6-7: IV-6 and V-2)
We have $d=2d'$ from (f), and substituting it for (d), 
we have $2nd'-2d'-1=n'd'-d'-1$ and thus $n'=2n-1$.
Substituting it for (c), we can deduce $n(n-2)=(2n-1)(2n-3)/4$, 
a contradiction.

From (1)--(6), we have the assertion.
\end{proof}

Now, Proposition \ref{detect} follows 
from Lemmas \ref{detect_1}, \ref{detect_2_1}, and \ref{detect_2_2}.
Consequently, Theorem \ref{not_stable_eq} follows 
from Propositions \ref{same_type} and \ref{detect}.
\qed

\section*{Acknowledgments}

The author is a Research Fellow of Japan Society for the Promotion of Science (JSPS).
This work was supported by JSPS KAKENHI Grant Number JP16J02249.

The author thanks his supervisor Osamu Iyama for thorough instruction,
and Manuel Saor\'{i}n for answering my questions and giving me useful advice.

\end{document}